\documentclass[11pt,leqno]{amsart}

\usepackage[top=2.5 cm,bottom=2 cm,left=2.5 cm,right=2 cm]{geometry}
\usepackage[utf8]{inputenc}

\usepackage{bbm}
\usepackage{esint}
\usepackage{graphicx}
\usepackage{hyperref}
\usepackage{color}

\newcounter{stepnb}

\newcommand{\firststep}{\setcounter{stepnb}{0}}
\newcommand{\step}[1]{{{\sc \addtocounter{stepnb}{1}\noindent  Step \arabic{stepnb}:} #1.}} 

\newtheorem{theorem}{Theorem}[section]
\newtheorem{lemma}[theorem]{Lemma}
\newtheorem{proposition}[theorem]{Proposition}

\newtheorem{definition}{Definition}[section]

\numberwithin{equation}{section}
\DeclareMathOperator*{\esssup}{ess\,sup}

\newcommand{\weaks}{\stackrel{*}{\rightharpoonup}}
\renewcommand{\i}{\ifmmode\mathit{\mathchar"7010 }\else\char"10 \fi}
\renewcommand{\j}{\ifmmode\mathit{\mathchar"7011 }\else\char"11 \fi}
\newcommand{\R}{\mathbb{R}}

\newcommand{\norm}[1]{\left\|#1\right\|}
\newcommand{\fhi}{\varphi}

\newcommand{\weak}{\rightharpoonup}

\newcommand{\pt}{\partial_t}

\newcommand{\px}{\partial_x }
\newcommand{\pxx}{\partial_{xx}^2}

\newcommand{\p}{\partial}

\newcommand{\vfi}{\varphi}
\newcommand{\eps}{\varepsilon}

\def\begi{\begin{itemize}}
\def\endi{\end{itemize}}
\def\bega{\begin{array}}
\def\enda{\end{array}}

\def\bel{\begin{equation}\label}
\def\eeq{\end{equation}}

{%

\begin{enumerate}}%
{\end{enumerate}}

{%

\begin{enumerate}}%
{\end{enumerate}}

\begin{document}

\title[Optimal strategies for a time-dependent harvesting problem]{Optimal strategies for a time-dependent harvesting problem}

\author{G. M. Coclite}
\address[Giuseppe Maria Coclite]{\newline
  Dipartimento di Matematica, Universit\`a di Bari,
  Via E.~Orabona 4,I--70125 Bari, Italy.}
\email[]{giuseppemaria.coclite@dm.uniba.it}
\urladdr{http://www.dm.uniba.it/Members/coclitegm/}

\author{M. Garavello}
\address[Mauro Garavello]{\newline
  Dipartimento di Matematica e Applicazioni,
  Universit\`a di Milano Bicocca,
  Via R. Cozzi 53, I--20125 Milano, Italy.}
\email[]{mauro.garavello@unimib.it}
\urladdr{http://www.matapp.unimib.it/~garavello}

\author{L. V. Spinolo}
\address[Laura V. Spinolo]{\newline
IMATI-CNR, Via Ferrata 1, I--27100 Pavia, Italy.}
\email[]{spinolo@imati.cnr.it}
\urladdr{http://arturo.imati.cnr.it/spinolo/}

\date{\today}



\thanks{}

\begin{abstract}
  We focus on an optimal control problem, introduced by Bressan and
  Shen in~\cite{BS1} as a model for fish harvesting. We consider the
  time-dependent case and we establish existence and uniqueness of an
  optimal strategy, and sufficient conditions for optimality. We
  also consider a related differential game that models the situation
  where there are several competing fish companies and we prove
  existence of Nash equilibria.  From the technical viewpoint, the
  most relevant point is establishing the uniqueness result. This
  amounts to prove precise a-priori estimates for solutions of
  suitable parabolic equations with measure-valued coefficients.  All
  the analysis is developed in the case when the fishing domain is
  one-dimensional.
\end{abstract}

\maketitle

\section{Introduction}
\label{sec:1}
This paper deals with a model for fish harvesting introduced by Bressan and Shen in~\cite{BS1}. 
The model involves an optimization problem for a payoff functional representing the profit of the fish company. 
We consider the time-dependent case and  prove existence and (local) uniqueness of optimal strategies. We also exhibit sufficient conditions for optimality and establish the existence of Nash equilibria in the case where there are several competing players, i.e. fish companies. We always focus on the case when the fishing domain is one-dimensional. 

Before discussing our results, we go over the main features of the model introduced in~\cite{BS1}. Consider a one-dimensional fishing domain, modeled by the interval $]0, R[$, and a time interval 
$]0, T[$. We denote by $\vfi=\vfi(t,x)$ the density of fish at time $t \in ]0, T[$ at the point $x \in ]0,R[$. We assume that, when no fishing activity is conducted, the evolution of the fish population is modelled  by the parabolic equation  
\begin{equation*}
\pt\vfi~=~\pxx\vfi+\fhi f(t,x,\fhi), \quad \text{in $]0, T[ \times ]0, R[$.}
\end{equation*}
A reasonable choice for the source term $f$ is the logistic law
\begin{equation}
\label{eq:1.3.1}
 f(t,x,\fhi)~=~\alpha(t,x)\big(h(t,x)-\fhi \big),
\end{equation}
where $h(t,x)$ denotes the maximum fish population that can be 
supported by the habitat at the point $x$ and at the time $t$, and $\alpha$ is a 
reproduction rate. Equation~\eqref{eq:1.3.1} is augmented with the initial datum 
\begin{equation}
\label{eq:1.2.1}
\vfi(0,x)=\vfi_0(x), \quad   x \in ]0,R[
\end{equation}
and the homogeneous Neumann boundary conditions 
\begin{equation}
\label{eq:1.2}
\px \vfi(t,0)=\px\vfi(t,R)=0,\quad  t \in ]0, T[.
\end{equation}
To conclude the model discussion,  
we denote by $\mu=\mu(t,x)$ the intensity of the harvesting 
conducted by a fish company. We consequently modify the equation for the evolution of the fish density by setting 
\begin{equation}
\label{eq:1.6}
\pt\vfi~=~\pxx\vfi+\vfi f(t,x,\vfi)-\vfi \mu, \quad \text{in $]0, T[ \times ]0, R[$}
\end{equation}
and again we augment it with the conditions~\eqref{eq:1.2.1} and~\eqref{eq:1.2}. 
To define our optimal control problem, we first introduce the cost functional 
\begin{equation}
\label{eq:1.7}
\int_0^T \! \! \int_0^R c(t,x) \, \mu(t,x)\,dtdx.
\end{equation}
In the above expression, $c$ is a nonnegative, lower semicontinuous function representing the cost of the fishing effort. One could for instance have a cost $c$ which is monotone increasing with respect to the 
distance of
the point $x$ from the fish company hub. Also, the presence of a natural park where no fishing is allowed can be modeled by setting  $c(t,x)=+\infty$ in that region. 
We can now define our payoff functional by setting 
\begin{equation}
  \label{e:cosaepayoff}
  J( \mu):=~\int_0^T \! \! \int_0^R \vfi(t,x)\mu(t,x)dtdx-
  \Psi\left(\int_0^T \! \! \int_0^R c(t,x)\,\mu(t,x)\,dtdx\right).
\end{equation}
In the above expression, $\Psi$ is a nondecreasing, convex function (the simplest possible choice is the identity). The function $\fhi$ is the solution of the initial-boundary value problem obtained by coupling~\eqref{eq:1.6} with~\eqref{eq:1.2.1} and~\eqref{eq:1.2}. Note, in particular, that $\fhi$ \emph{depends} on $\mu$ and hence the functional $J$ is nonlinear. In the present paper we focus on the problem of maximizing the payoff functional $J$, i.e. finding an optimal fishing strategy $\mu$. We impose the constraints 
\begin{equation}
\label{eq:1.9}
\mu(t,x) \geq 0,\qquad  
\int_0^T
\! \! \int_0^R b(t,x)\, \mu(t,x)\,dtdx \leq 1.
\end{equation}
In the above expression, the nonnegative function $b$
models the maximum amount 
of harvesting power
within the capabilities of the fish company.
In practice, it may for instance depend on the number of fishermen
and on the size of the fishing boats.

To solve the above optimization problem we actually search for optimal strategies that are not necessarily functions, but more generally nonnegative Radon measures. This is motivated by two main considerations:
\begin{itemize}
\item from the \emph{analytic} viewpoint, we remark that the functional $J$ has only linear growth with respect to $\mu$, 
and hence we expect that, in general, an optimal strategy $\mu$ does not belong to $L^1 (]0, T[ \times ]0, R[)$. Note  that a quadratic harvesting 
cost such like  
\begin{equation*}
\int_0^T \! \! \int_0^R c(t,x ) \mu^2(t,x)\, dtdx
\end{equation*}
is entirely natural from the mathematical viewpoint and it would give an optimal strategy  
$\mu^{opt}\in L^2(]0,T[ \times ]0,R[)$. However, the linear cost \eqref{eq:1.7} provides  a more realistic model. 

\item From the \emph{modeling} viewpoint, it is reasonable to expect that, when for instance there is a natural park, the optimal strategy concentrates the fishing effort at the park border. 
An explicit analytic example where the optimal strategy contains atomic parts concentrated at discontinuity points of $c$   is exhibited in~\cite{BS2}. 
\end{itemize}
As mentioned before, the above model for fish harvesting was introduced by Bressan and Shen in~\cite{BS1}. In~\cite{BS1} the analysis focuses on the one-dimensional, steady state when~\eqref{eq:1.6} reduces to a second order ordinary differential equation. The authors establish existence and local uniqueness of optimal strategies and discuss the related differential game showing existence of Nash equilibria. See also~\cite{BS2} for related results. In~\cite{BCS} Bressan, Coclite and Shen established existence of optimal strategies for the steady case by considering multidimensional fishing domains. Finally, in~\cite{CG} Coclite and Garavello established existence of optimal strategies in multi-dimensional domains in the time-dependent case.

The main results of the present paper are the following:
\begin{itemize}
\item we establish existence, uniqueness and stability of  \emph{weak solutions} (in the sense of Definition~\ref{def:sol} in~\S~\ref{sec:3}) of the parabolic equation~\eqref{eq:1.6} in the case when $\mu$ is a Radon measure, see Theorem~\ref{th:main1}. We basically follow the same strategy as in~\cite{CG}, but we can impose much weaker assumptions on the coefficient $\mu$ owing to the fact that the domain is one-dimensional. 
\item We establish existence of an optimal strategy $\mu$ for the payoff functional $J$ in~\eqref{e:cosaepayoff} subject to the constraints~\eqref{eq:1.9}. We also establish sufficient conditions for optimality (see Theorem~\ref{th:main2}) and we show that the optimal strategy is locally unique, i.e. it is unique in the class of measures with sufficiently small total variation. The uniqueness result is stated as Theorem~\ref{th:main3}. Note that, while the existence proof is the same as in~\cite{BCS,CG}, the uniqueness proof is, from the technical viewpoint, the most relevant result of the present paper. 
\item By relying on the above local uniqueness result we establish existence of Nash equilibria for a differential game modeling the case where there are several competing fish companies that exploit the same environment, see Theorem~\ref{th:mainNash} for the precise result.  
\end{itemize}
The local uniqueness of optimal strategies was established by Bressan and Shen~\cite{BS1} in the steady case. The main novelties of our analysis compared to the one in~\cite{BS1} are the following:
\begin{itemize}
\item in both cases, the main point of the argument is showing that the functional $J$ is locally concave. This amounts to establish suitable a-priori estimates on the solutions of parabolic equations with measured-valued coefficients similar to~\eqref{eq:1.6}.  However, as mentioned before, in the steady case the parabolic equation~\eqref{eq:1.6} reduces to a second order ordinary differential equation: this makes the analysis considerably simpler than the time-dependent case. In particular, in the time-dependent case we establish precise estimates on solutions of parabolic equations with measured-valued coefficients by making extensive use of the Duhamel representation formula. 
\item The analysis in~\cite{BS1} is based on a technical assumption, i.e. condition~\cite[(5.15)]{BS1}. In the time-dependent case we replace~\cite[(5.15)]{BS1} with~\eqref{eq:initassN}, namely with the requirement that the initial fish density distribution is sufficiently close, in the $H^1$ norm, to the constant $h$, which represents the maximal fish density supported by the environment.   
\end{itemize}
The exposition is organized as follows. In~\S\ref{sec:2} we establish existence, uniqueness and stability results for a nonlinear parabolic problem with smooth coefficients. These results are pivotal to the analysis in~\S\ref{sec:3}, where we establish existence, uniqueness and stability results for the initial-boundary value problem~\eqref{eq:1.2.1},~\eqref{eq:1.2},~\eqref{eq:1.6} in the case when $\mu$ is a given nonnegative Radon measure. In~\S\ref{sec:4} we prove existence of an optimal strategy 
maximizing~\eqref{e:cosaepayoff} subject to the constraint~\eqref{eq:1.9}. We also establish sufficient conditions for optimality. In~\S\ref{sec:6} we establish the local uniqueness of the optimal strategy, while~\S\ref{s:pkey} is devoted to the proof of a technical result which is pivotal to the existence proof. Finally, in~\S\ref{sec:7} we introduce a differential game modeling the situation where there are several competing fish companies and we establish existence of Nash equilibria. Finally, in the Appendix we collect some results concerning the Duhamel representation formula and the fundamental solutions of the heat equation that we use in the paper.  
\subsection*{Notation}
For the reader's convenience, we collect here the main notation used in the present paper.

Throughout the paper, $C(a_1, \dots, a_k)$ denotes a constant which only depends on the quantities $a_1, \dots, a_k$: its precise value can vary from occurrence to occurrence. Also, $K$ denotes a universal constant (i.e., a number) and again its precise value can vary from occurrence to occurrence.
\subsubsection*{General mathematical symbols}
\begin{itemize}
\item $\R^+$: the interval $[0, + \infty[$. 
\item $C^0 ([0, R])$: the space of continuous functions defined on the interval $[0, R]$.
\item $H^1 (\, ]0, R[ \,)$: the Sobolev space $W^{1, 2} (\, ]0, R[ \, )$, endowed with the norm
$$
    \| u \|_{H^1(]0, R[ )} : =  \sqrt{ \| u \|^2_{L^2(]0, R[ )} + \| \partial_x u \|^2_{L^2(]0, R[ )}  }. 
$$
Note that the Sobolev space $H^1(]0, R[ )$ compactly embeds into
$C^0([0, R])$ and we have the inequality
\begin{equation}
\label{e:immersione}
    \| u \|_{C^0 ([0, R])} \leq C(R) \| u \|_{H^1 (]0, R[)}
    \qquad \text{for every $u \in H^1 (]0, R[)$}.
\end{equation}
\item $H^\ast (]0, R[ )$: the dual space of $H^1 ( ]0, R[  ).$
\item $C^\infty_c (\Omega)$: the space of smooth, compactly supported functions defined on the open set $\Omega$.   
\item $\mathcal M (]0, R[)$: the space of (signed) Radon measures on the interval $]0, R[$. We denote by 
  \begin{equation*}
    \| \mu \|_{\mathcal M(]0, R[)} : = |\mu| (]0, R[)
  \end{equation*}
the total variation of the (signed) Radon measure $\mu$. 
\item $\mathcal M_+ (]0, R[)$: the space of nonnegative Radon measures on the interval $]0, R[$.
\item a.e.~$x$, a.e.~$(t, x)$: for $\mathcal L^1$ almost every $x$, for 
$\mathcal L^2$ almost every $(t, x)$. Here $\mathcal L^1$, $\mathcal L^2$ denote the Lebesgue measure on $\R^1$ and $\R^2$, respectively. 
\end{itemize}
To conclude, we recall that, owing to the H\"older inequality,
\begin{equation}
\label{e:elledueelleuno}
    \| u \|_{L^1 (]0, R[)} \leq \sqrt{R} \| u \|_{L^2 (]0, R[)}
    \qquad \text{for every $u \in L^2 (]0, R[)$}.
\end{equation}
\subsubsection*{Notation introduced in the present paper}
\begin{itemize}
\item $\alpha_1$: the Lipschitz constant in~\eqref{eq:2.2}.
\item $M$: the constant defined by~\eqref{e:emme}.
\item $F$: the constant defined in~\eqref{e:finty}. 
\item $\alpha_2$: the Lipschitz constant in~\eqref{e:alpha2}.
\item $T$: the length of the time interval where we set our problem.
\item $R$: the length of the space interval where we set our problem.
\item $h$: the function $h$ in~\eqref{eq:2.3} and the constant $h$ in~\eqref{eq:2.3bis}. 
\item $h_\ast$: the constant defined as in~\eqref{e:accastar}. 
\end{itemize}
\subsubsection*{Hypotheses}
\begin{itemize}
\item {\bf (H.1)}: the hypothesis introduced at page \pageref{h:accauno}. 
\item {\bf (H.2)}:  the hypothesis introduced at page \pageref{h:accatre}.  
\item {\bf (H.3)}:  the hypothesis introduced at page \pageref{h:acca3}.  
\item {\bf (H.4)}:  the hypothesis introduced at page \pageref{h:accaquattro}.
\item {\bf (H.5)}:  the hypothesis introduced at page \pageref{h:acca5}.
\item {\bf (H.6)}:  the hypothesis introduced at page \pageref{h:accasei}. 
\item {\bf (H.7)}:  the hypothesis introduced at page \pageref{h:accasette}.  
\item {\bf (H.8)}:  the hypothesis introduced at page \pageref{h:accaotto}. 
\end{itemize}
\section{A nonlinear parabolic problem with smooth coefficients}
\label{sec:2}
In this section we focus on the nonlinear parabolic problem
\begin{equation}
\label{eq:auxiliary-system}
\begin{cases}
\pt \fhi = \pxx \fhi - a (t, x) \fhi + f(t, x, \fhi) \fhi, & \text{in $]0, T[ \times ]0, R[$} ,  \\
\px \fhi(t,0)=\px\vfi(t,R) = 0, & t \in ]0, T[, \\
\fhi(0,x) = \fhi_0(x), & x \in {]0,R[}.
    \end{cases}
\end{equation}
In the previous expression $a: ]0, T[ \times ]0, R[ \to \R$ is a nonnegative, smooth function. The nonlinear source term $f$ satisfies the following hypothesis.  
\begin{itemize}
\item[({\bf H.1})] \label{h:accauno}
 The function $f:[0,T] \times [0,R] \times \R \to \R$
  is $C^2$ and there are a constant $\alpha_1 >0$ and a 
 continuous, nonnegative function $h:[0,T]\times[0,R]\to \R^+$
 such that  
 \begin{align}
    \label{eq:2.2}
    -\alpha_1\le\p_\vfi f  (t,x,\vfi)<0, 
    & \qquad \hbox{ for all } (t,x,\vfi)\in [0,T]
    \times[0,R]\times\R,
    \\
    \label{eq:2.3}
    f(t,x,\vfi)> 0, 
    & \qquad \text{ if and only if }\quad \vfi< h(t,x).
  \end{align}
\end{itemize}
We first provide the definition of weak solution of~\eqref{eq:auxiliary-system}. 
\begin{definition}
\label{d:uno}
We term \emph{weak solution} of the initial-boundary value problem~\eqref{eq:auxiliary-system} a function  
\begin{equation}
\label{e:dentrospazi}
         \varphi \in L^2 \big( ]0, T[ ; H^1 ( \, ]0, R[ \, ) \big) \quad \text{such that} \quad
          \partial_t \fhi \in L^2 \big( ]0, T[ ; H^\ast ( \, ]0, R[ \, ) \big).
\end{equation} 
Also, we require that the following equality holds
for every test function $v \in C_c^\infty(]-\infty,T[ \times \R)$:
\begin{equation}
  \label{e:distrform}
  \int_0^T \!\!\!\int_0^R 
  \!\left(\pt v \,  \vfi-\px v\, \px \vfi \right) dx dt
  - \int_0^T\!\!\! \int_0^R \! v \, \vfi \, a \, dx dt 
  +\!\int_0^T \!\!\! \int_0^R \! v f (t, x,  \vfi) \vfi \, dxdt+
  \!\int_0^R \!v(0, x )\, \vfi_0 \, dx=0.
\end{equation}
\end{definition}
Note that~\eqref{e:dentrospazi} implies that, by possibly changing the value of $\fhi(t, \cdot)$ in a negligible set of times, we can assume that 
$\fhi \in C^0 \big( [0,T]; L^2 (]0, R[) \big)$, see for instance~\cite[Theorem 7.22]{Salsa}. In the following, we will always identify $\fhi$ satisfying~\eqref{e:dentrospazi} with its $L^2$-continuous representative. In this way we can define the value $\fhi(t, \cdot)$  for \emph{every} $t \in [0, T]$. We can now state the main result of this section, which establishes existence, uniqueness and stability for the initial-boundary value problem~\eqref{eq:auxiliary-system}. 
\begin{theorem}
  \label{p:classical} 
  Let hypothesis {\bf (H.1)} hold and assume furthermore that 
\begin{equation}
\label{e:condizionidatoiniziale}
    \fhi_0 \in L^\infty(]0, R[), \quad \fhi_0 (x) \ge 0 
    \quad \text{for a.e. $x \in \R$}
\end{equation}
and that $a: ]0, T[ \times ]0, R[ \to \R$ is a smooth function satisfying
\begin{equation}
\label{e:conda}
      a \ge 0.
\end{equation}
Then the initial-boundary value problem~\eqref{eq:auxiliary-system} admits a  
unique weak solution.
Also, this solution enjoys the following properties: first, 
\begin{equation}
\label{e:maxprin}
  0\leq \varphi(t, x) \leq M
  \quad \text{for a.e. $(t, x)\in ]0, T[ \times ]0, R[$,} 
\end{equation}
where 
\begin{equation}
\label{e:emme}
 M: = \max \{ \| h \|_{L^\infty}, \| \varphi_0 \|_{L^\infty}  \}. 
\end{equation}
Second, we have stability with respect to the initial datum and with respect to the coefficient $a$. More precisely, let  
$\widehat\vfi$ be the solution of the initial-boundary value problem
\begin{equation}
\label{e:ficappuccio}
\begin{cases}
\pt \hat \fhi =  \pxx \hat \fhi - \hat a (t, x) \hat \fhi + f(t, x, \hat \fhi) \hat
\fhi, & \text{in $]0, T[ \times ]0, R[$} ,  \\
\px \hat \fhi(t,0)=\px \hat \vfi(t,R) = 0, & t \in ]0, T[, \\
\hat \fhi(0,x) = \hat \fhi_0(x), & x \in {]0,R[},
    \end{cases}
\end{equation} where $\hat a$ satisfies the same hypotheses as $a$ and $\hat \fhi_0$
satisfies~\eqref{e:condizionidatoiniziale}. 
Then
    \begin{equation}
    \label{eq:th2.3}
    \begin{split}
            & \norm{\vfi(t,\cdot)-  
            \widehat\vfi(t,\cdot)}^2_{L^2(]0,R[)}+  
             \int_0^t  \norm{ \partial_x \vfi(s,\cdot) -  
             \partial_x \widehat\vfi(s,\cdot)}^2_{L^2(]0,R[)}ds\\
    & \qquad \le 
    C(\alpha_1, M, T, R, F) 
    \Bigg[  \sup_{t \in [0, T]}
    \| a(t, \cdot)  - \hat a (t, \cdot)  \|^2_{L^1 (]0, R[)} + 
     \norm{\vfi_0-\widehat\vfi_0}^2_{L^2(]0,R[)}  
     \Bigg]
     \\
    &  \qquad \qquad \qquad 
    \text{for every $0\le t\le T$.} \phantom{\int} \\
\end{split}
  \end{equation} 
  In the previous expression, the constant $F$ is defined by setting  
\begin{equation}
\label{e:finty}
   F =
    \max \Big\{ f(t, x, 0), \; (t, x) \in [0, T] \times [0, R] \Big\}. 
\end{equation}
\end{theorem}
The proof of Theorem~\ref{p:classical} is organized as follows. In Subsection ~\ref{sss:uni} we establish the 
stability estimate, which implies uniqueness of weak solutions. In Subsection ~\ref{ss:exi} we establish existence of a weak solution of the initial-boundary problem~\eqref{eq:auxiliary-system} by relying on an iteration algorithm. 
\subsection{Uniqueness and Stability}
\label{sss:uni}
First, we fix $\fhi$ and $\hat \fhi$ that are weak solutions of~\eqref{eq:auxiliary-system} and~\eqref{e:ficappuccio}, respectively, and we point out that the function 
\begin{equation}
\label{e:cosaepsi}
    \psi : = \fhi - \hat \fhi
\end{equation}
is a weak solution of the initial-boundary value problem
\begin{equation}
\label{e:fhimenohfhi}
\begin{cases}
\pt \psi =  \pxx \psi - \big[ a (t, x) \fhi- \hat a (t, x) 
       \hat \fhi \big] + f(t, x, \fhi) \fhi- f(t, x, \hat \fhi) \hat
\fhi, & \text{in $]0, T[ \times ]0, R[$} ,  \\
\px \psi(t,0)=\px \psi(t,R) = 0, & t \in ]0, T[, \\
\psi (0,x) = \fhi_0 (x)- \hat \fhi_0(x), & x \in {]0,R[}.
\end{cases}
\end{equation}
Next, we point out that
\begin{equation}
\label{e:boundsueffe}
  |f(t, x, \fhi) |
  \leq |f(t, x, \fhi) - f(t, x , 0)| +
  |f(t, x, 0)|
  \stackrel{\eqref{eq:2.2},\eqref{e:finty}}{\leq}
   \alpha_1 |\fhi| +F 
\end{equation}
and the above inequality implies
\begin{equation}
\label{e:effeinintegrale}
\begin{split}
      |f(t, x, \fhi) \fhi- f(t, x, \hat \fhi) \hat \fhi |
      & 
      \leq |f(t, x, \fhi) - f(t, x, \hat \fhi) | | \fhi| +
      | f(t, x, \hat \fhi) | |\fhi - \hat \fhi| \\
      & \stackrel{\eqref{eq:2.2},\eqref{e:cosaepsi}}{\leq}
      \alpha_1 |\psi| |\fhi| +  | f(t, x, \hat \fhi) | |\psi| \\
      & \stackrel{\eqref{e:boundsueffe}}{\leq}
      \alpha_1 |\psi| |\fhi| + \alpha_1 |\psi| |\hat \fhi| 
      + F |\psi|. \\
\end{split}      
\end{equation}
We conclude the proof by proceeding according to the following steps. \\
\firststep
\step{we give a formal proof of uniqueness and stability} We proceed formally, i.e. we pretend that everything is sufficiently regular to have that all the following manipulations are justified. We refer to {\sc Step 2} below for the rigorous justification of our argument. 

We multiply the equation at the first line of~\eqref{e:fhimenohfhi} times $\psi$, we integrate with respect to  
space and we use the homogeneous Neumann boundary conditions. We get 
\begin{equation}
  \label{e:uni:gronwall1}
  \frac{d}{dt } \int_0^R \frac{\psi^2}{2} dx  +
  \int_0^R (\px \psi)^2 dx = \int_0^R \! \!
  \big[  \hat a (t, x) 
  \hat \fhi -a (t, x) \fhi \big] \psi dx  +  \int_0^R \! \! 
  \big[ f(t, x, \fhi) \fhi- f(t, x, \hat \fhi) 
  \hat \fhi \big] \psi dx.
     \end{equation}
 Next, we recall that $\hat a \ge 0$ by assumption
 and we infer that 
 $$
     \big[  \hat a  
       \hat \fhi -a  \fhi \big] \psi =
       \hat a \big[ \hat \fhi - \fhi \big] \psi +
     \big[  \hat a  
        -a   \big] \fhi \psi  \stackrel{\eqref{e:cosaepsi}}{=}
        -\hat a \psi^2 + 
         \big[  \hat a  
        -a   \big] \fhi \psi 
       \stackrel{\hat a \ge 0}{\leq}
       |\hat a  
        -a  || \fhi|| \psi|. 
 $$
 Hence, from~\eqref{e:uni:gronwall1} we get 
 \begin{equation}
   \label{e:uni:gronwall2}
   \begin{split}
     \frac{d}{dt } \int_0^R \frac{\psi^2}{2} dx & + \int_0^R (\px
     \psi)^2 dx \leq \| a (t, \cdot) - \hat a (t, \cdot) \|_{L^1 (]0,
       R[)} \| \fhi(t, \cdot) \|_{L^\infty (]0, R[)} \| \psi (t,
     \cdot) \|_{L^\infty (]0, R[)} 
     \\ 
     & + \int_0^R \! \! \big[ f(t, x,
     \fhi) \fhi- f(t, x, \hat \fhi)
     \hat \fhi \big] \psi dx    
     \\
     & \stackrel{\eqref{e:effeinintegrale}}{\leq} \| a (t, \cdot) -
     \hat a (t, \cdot) \|_{L^1 (]0, R[)} \| \fhi (t, \cdot)
     \|_{L^\infty (]0, R[)} \| \psi (t, \cdot) \|_{L^\infty (]0, R[)}
     \\
     & + C(\alpha_1) \left[\| \fhi(t, \cdot) \|_{L^\infty (]0, R[)} +
       \|\hat\fhi(t, \cdot) \|_{L^\infty (]0, R[)}\right]
     \int_0^R \psi^2 dx + F \int_0^R \psi^2 dx
     \\
     & \stackrel{\eqref{e:immersione}}{\leq} C(R) \| a (t, \cdot) -
     \hat a (t, \cdot) \|_{L^1 (]0, R[)} \| \fhi (t, \cdot)
     \|_{L^\infty (]0, R[)}
     \| \psi (t, \cdot) \|_{H^1 (]0, R[)} \phantom{\int} 
     \\
     & \quad + C(\alpha_1) \left[\| \fhi(t, \cdot) \|_{L^\infty (]0, R[)}
     + \|\hat\fhi(t, \cdot) \|_{L^\infty (]0, R[)}\right]
     \int_0^R \psi^2 dx + F \int_0^R \psi^2 dx.
   \end{split}
 \end{equation}  
Next, we use the Young Inequality: we fix a parameter $\gamma>0$, to be determined in the following, 
and we point out that 
\begin{equation}
  \label{e:uni:young}
  \begin{split}
    \| a (t, \cdot) - \hat a (t, \cdot) \|_{L^1 (]0, R[)} & \| \fhi
    (t, \cdot) \|_{L^\infty (]0, R[)} \| \psi (t, \cdot) \|_{H^1 (]0,
      R[)} \leq \frac{1}{2 \gamma} \| a (t, \cdot) - \hat a (t, \cdot)
    \|^2_{L^1 (]0, R[)}
    \| \fhi (t, \cdot) \|^2_{L^\infty (]0, R[)} \\
    & + \frac{\gamma}{2} \| \psi (t, \cdot) \|^2_{H^1 (]0, R[)}.
  \end{split}
\end{equation}
We now choose $\gamma$  in such a way that 
\begin{equation}
\label{e:uni:gamma12}
    C(R) \frac{\gamma}{2} =
    \frac{1}{2}, 
\end{equation}
we recall that 
$$
    \| \psi (t, \cdot) \|^2_{H^1 (]0, R[)} = 
    \int_0^R \psi^2 dx + \int_0^R (\px \psi)^2 dx  
$$
and by using~\eqref{e:uni:gronwall2} we arrive at
\begin{equation}
  \label{e:uni:intermedia}
  \begin{split}
    \frac{d}{dt } \int_0^R \frac{\psi^2}{2} dx & +
    \int_0^R \frac{ (\px \psi)^2}{2} dx 
    \\
    \stackrel{\eqref{e:uni:young}}{\leq} & \left[ \frac{1}{2} +
      C(\alpha_1)  \left(\| \fhi (t, \cdot) \|_{L^\infty (]0, R[)}
      + \| \hat\fhi (t, \cdot) \|_{L^\infty (]0, R[)} \right)
      + F \right] \int_0^R \psi^2 dx 
    \\
    & + C(R) \| a (t, \cdot) - \hat a (t, \cdot) \|^2_{L^1 (]0, R[)}
    \frac{1}{2 \gamma}  \| \fhi (t, \cdot) \|^2_{L^\infty (]0, R[)} 
    \\
    \stackrel{\eqref{e:uni:gamma12}}{\leq}& \left[ \frac{1}{2}+
      C(\alpha_1)
      \left(\| \fhi (t, \cdot) \|_{L^\infty (]0, R[)}
        + \| \hat\fhi (t, \cdot) \|_{L^\infty (]0, R[)}\right)
      + F \right]\int_0^R \psi^2 dx \phantom{\int}
    \\
    & + C(R) \| a (t, \cdot) - \hat a (t, \cdot) \|^2_{L^1 (]0, R[)}
    \| \fhi (t, \cdot) \|^2_{L^\infty (]0, R[)} .  \phantom{\int}
  \end{split}
\end{equation}
Owing to the Gronwall Lemma, the above inequality implies that 
\begin{equation}
  \label{e:uni:stima1}
  \begin{split}
    & \int_0^R \frac{\psi^2}{2}(t, x) dx \leq \left[ \int_0^R
      \frac{\psi^2_0}{2} dx + C(R) \sup_{s \in ]0, t[ } \| a (s,
      \cdot) - \hat a (s, \cdot) \|^2_{L^1 (]0, R[)} \int_0^t \| \fhi
      (s, \cdot) \|^2_{L^\infty (]0, R[)} ds
    \right] \times
    \\
    & \times \exp \left( C(\alpha_1) \int_0^t \left(\| \fhi (s, \cdot)
      \|_{L^\infty (]0, R[)}
      +\| \hat\fhi (t, \cdot) \|_{L^\infty (]0, R[)}\right) ds + (F+K) t \right).
  \end{split}
\end{equation}
Note that by assumption $\fhi \in L^2 (]0, T[; H^1 (]0, R[))$ and hence 
\begin{eqnarray*}
  \int_0^t   \| \fhi (s, \cdot) \|^2_{L^\infty (]0, R[)} ds 
  \stackrel{\eqref{e:immersione}}{\leq} C(R)
  \int_0^t  \| \fhi (s, \cdot) \|^2_{H^1 (]0, R[)}  
  ds < + \infty, 
  \\
  \int_0^t   \! \| \fhi(s, \cdot) \|_{L^\infty (]0, R[)} ds 
  \stackrel{\eqref{e:immersione}}{\leq} C(R)
  \int_0^t  \! \| \fhi(s, \cdot) \|_{H^1(]0, R[)} ds
  \stackrel{\eqref{e:elledueelleuno}}{\leq} C(R) \sqrt{t}
  \left( \int_0^t   \| \fhi(s, \cdot) \|^2_{H^1(]0, R[)} ds\right)^{\frac{1}{2}}
  \! \!\!<+\infty.
\end{eqnarray*}
The same inequalities hold also for $\hat \fhi$.
To establish uniqueness, we take $\fhi_0 \equiv \hat \fhi_0$, which implies $\psi_0 \equiv 0$, and $a \equiv \hat a$.  From~\eqref{e:uni:stima1} we get $\psi \equiv 0$, which owing to~\eqref{e:cosaepsi} implies $\fhi \equiv \hat \fhi$ and hence establishes uniqueness. 

To establish the stability estimate~\eqref{eq:th2.3} we use the uniform bound~\eqref{e:maxprin} and from~\eqref{e:uni:stima1} we get
\begin{equation}
\label{e:uni:stima2}
\begin{split}
     & \int_0^R \frac{\psi^2}{2}(t, x) dx 
    \leq 
    \left[  
     \int_0^R \frac{\psi^2_0}{2} dx + 
     C(R)  \sup_{s \in ]0, t[ } \| a (s, \cdot) - \hat a (s, \cdot)  \|^2_{L^1 (]0, R[)}  
   M^2 t  
       \right] 
      \exp \Big(  
     C(\alpha_1, R)  M t   + (F+K) t 
      \Big).
      \end{split}
\end{equation}
By plugging the above inequality into~\eqref{e:uni:intermedia} and integrating in time 
we eventually arrive at~\eqref{eq:th2.3}. \\
\step{we provide the rigorous justification of our argument} First, we recall the definition~\eqref{e:cosaepsi} and we point out that, for every test function $v \in C^\infty_c (]- \infty, T] \times \R)$ we have 
\begin{equation}
\label{e:distrformpsi}
 \begin{split}
\int_0^T \!  \! \int_0^R &
 \left(\pt v \, \psi-\px v\, \px \psi \right) dx dt-
 \int_0^T  \!  \! \int_0^R \big[ a \fhi -\hat a \hat \fhi \big] v dx dt \\
 & 
+\int_0^T \!  \! \int_0^R v f (t, x,  \vfi) \vfi \, dxdt -
 \int_0^T \!  \! \int_0^R v f (t, x,  \hat \vfi) \hat \vfi \, dxdt -
\int_0^Rv(0, x )\, \psi_0 \, dx=0.
\end{split}
\end{equation}
We fix $t \in ]0, T]$ and  choose the test function $v$ by setting 
\begin{equation}
\label{e:cosaev}
   v(s, x) : = w_n (s) z_j (x),
\end{equation}
where $z_j$ is determined in the following and $w_n$
is a smooth cut-off function with compact support such that 
$$
w_n (s) =
\left\{
\begin{array}{ll}
1 & x \leq s ,\\
0 & x \ge s + 1/n,
\end{array}
\right.
\quad 
w'_n (s) \leq 0 \qquad \text{for every $s \in ]0, T[$}.
$$
We let $n \to + \infty$ and by using the fact that $\psi \in C^0 ([0, T], L^2 (]0, R[))$ we infer from~\eqref{e:distrformpsi} the following equality
\begin{equation}
  \label{e:distrformpsi2}
  \begin{split}
    \int_0^R & z_j (x) \psi(t,x) dx + \int_0^t \!  \! \int_0^R \px z_j \,
    \px \psi dx ds+
    \int_0^t  \!  \! \int_0^R \big[ a \fhi -\hat a \hat \fhi \big] z_j dx ds\\
    & - \int_0^t \!  \! \int_0^R z_j f (s, x, \vfi) \vfi \,
    dxds + \int_0^t \!  \! \int_0^R z_j f (s, x, \hat \vfi)
    \hat \vfi \, dxds + \int_0^R z_j \psi_0 \, dx=0.
  \end{split}
\end{equation}
Next, we point out that 
\begin{equation}
\label{e:rigorous}
  \int_0^R \psi(t, \cdot)  \psi_0 \, dx \stackrel{\text{H\"older}}{\leq}
  \| \psi (t, \cdot) \|_{L^2 (]0, R[)}
   \| \psi_0  \|_{L^2 (]0, R[)}
   \stackrel{\text{Young's }}{\leq}
   \frac{1}{2}  \| \psi (t, \cdot) \|^2_{L^2 (]0, R[)}
   + \frac{1}{2}  \| \psi_0  \|^2_{L^2 (]0, R[)}. 
\end{equation}
and we choose a sequence of test functions $z_j$ in such a way that
$$
    z_j \to \psi (t, \cdot) \quad \text{strongly in $H^1(]0, R[)$.}
$$
We let $j \to + \infty$ and by arguing as in {\sc Step 1} and using~\eqref{e:rigorous} we infer from~\eqref{e:distrformpsi2} the inequality
\begin{equation*}
\begin{split}
      \int_0^R &
 \frac{\psi^2}{2} dx + \int_0^t \!  \! \int_0^R \frac{(\px \psi)^2}{2}  dx dt
 \leq 
  \int_0^R 
 \frac{\psi_0^2}{2} dx  
 +  C(R)  
 \sup_{s \in ]0, t[ } \| a (s, \cdot) - \hat a (s, \cdot)  \|^2_{L^1 (]0, R[)}  
   \int_0^t   \| \fhi (s, \cdot) \|^2_{L^\infty (]0, R[)} ds \\
   & + C(\alpha_1)
   \int_0^t  \| \fhi (t, \cdot) \|_{L^\infty (]0, R[)} 
   \int_0^R \psi^2 dx dt +
   \left( F + \frac{1}{2} \right) \int_0^t \! \! \int_0^R \psi^2 dx dt.
\end{split}
\end{equation*}
We can then argue as in {\sc Step 1} and establish the uniqueness of weak solutions and the stability estimate~\eqref{eq:th2.3}.
\subsection{Existence}
\label{ss:exi}
In this paragraph we establish existence of a weak solution of the initial-boundary value problem~\eqref{eq:auxiliary-system}. More precisely, we proceed as follows. 
\begin{itemize}
\item \S~\ref{ss:exa}: we define an approximation algorithm and 
we establish a-priori bounds on the approximate solutions. To simplify the analysis, in~\S~\ref{ss:exa} we assume that the initial datum $\fhi_0$ is smooth. This hypothesis will be eventually removed in~\S~\ref{sss:concludiamo}. 
\item \S~\ref{sss:conex}: we establish compactness of the approximate solutions, we pass to the limit and we establish existence of a weak solution of~\eqref{eq:auxiliary-system}. 
\item \S~\ref{sss:concludiamo}: we conclude the proof of Theorem~\ref{p:classical} by establishing the maximum principle~\eqref{e:maxprin} and by removing the assumption that $\fhi_0$ 
is smooth. 
\end{itemize}
\subsubsection{Construction of approximate solutions}
\label{ss:exa}
In this paragraph we define the iteration algorithm that we will use to establish the existence of a solution $\varphi$ as in the statement of Theorem~\ref{p:classical}. 

First, we term $\varphi_1$ the 
solution of the following linear parabolic initial-boundary value problem:   
\begin{equation}
\label{eq:auxiliary-system1}
\begin{cases}
\pt \fhi_{1} = \pxx \fhi_{1} - a(t, x) \fhi_{1} + f_0 (t, x) \fhi_{1}, & (t, x) \in ]0, T[ \times ]0, R[ ,  \\
\px \fhi_{1}(t,0)=\px \vfi_{1} (t,R) = 0, & t \in ]0, T[, \\
\fhi_{1} (0,x) = \fhi_0(x), & x \in ]0,R[ \, ,\\
\end{cases}
\end{equation}
where $f_0(t, x) = f \big(t, x, \varphi_0( x) \big).$
Next, we argue iteratively: we assume that the function ${\varphi_n \! : \! [0, T] \! \times \! [0, R] \!\to \! \R}$ is given and we define the function $f_n: [0, T] \times [0, R] \to \R$ by setting 
\begin{equation}
\label{e:effenne}
    f_n (t, x) : = f \big( t, x, \varphi_n (t, x) \big). 
\end{equation}
We term $\varphi_{n+1}$ the solution of the following linear parabolic initial-boundary value problem:   
\begin{equation}
  \label{e:sy:itera}
  \begin{cases}
    \pt \fhi_{n+1} = \pxx \fhi_{n+1} - a (t, x)
    \fhi_{n+1} + f_n (t,x) \fhi_{n+1}, & (t, x) \in ]0, T[ \times ]0, R[ ,  \\
    \px \fhi_{n+1}(t,0)=\px \vfi_{n+1} (t,R) = 0, & t \in ]0, T[, \\
    \fhi_{n+1} (0,x) = \fhi_0(x), & x \in ]0,R[ \, ,\\
  \end{cases}
\end{equation}
The main ingredient in the iteration argument is the following lemma.  
\begin{lemma}
\label{l:iteration} 
Let the same assumptions as in the statement of Theorem~\ref{p:classical}  
hold. Assume furthermore that 
\begin{itemize}
\item the initial datum $\varphi_0 \in  C^\infty ([0, R])$; 
\item the function $\varphi_n$ satisfies  
\begin{equation}
\label{e:hypphienne}
       \varphi_n \in  C^\infty  ([0, T] \times [0, R]), \quad 
       \varphi_n (t, x) \ge 0 \quad \text{for every $(t, x)$}, \quad         
\end{equation}
\end{itemize}
Then the initial-boundary value problem~\eqref{e:sy:itera} admits a unique classical solution 
${ \varphi_{n+1} \in  C^\infty  ([0, T] \times [0, R])}$, which furthermore 
 satisfies the following estimates: first, 
\begin{equation}
\label{e:maxprink}
  \fhi_{n+1} (t, x) \ge 0 
  \quad \text{for every $(t, x)\in [0, T] \times [0, R]$.} 
\end{equation}
Second, 
\begin{equation}
\label{e:elleduek}
   \| \varphi_{n+1}(t, \cdot) \|^2 
   _{L^2(]0, R[) } +  \int_0^{t} \int_0^R (\partial_x \varphi_{n+1})^2 
   dx \, ds     \leq K \| \varphi_0 \|^2_{L^2(]0, R[) } e^{F t},   
   \quad \text{for every $t \in [0, T]$.}
\end{equation} 
\end{lemma} 
\begin{proof}
We proceed according to the following steps. \\
\firststep
\step{we establish existence of a classical solution} First, we recall that by assumption the function $\varphi_n$ is smooth.  The very definition~\eqref{e:effenne} of $f_n$ and \eqref{e:hypphienne} implies that $f_n$ is also smooth, and henceforth bounded on $[0, T] \times [0, R]$. The same holds for the coefficient $a$, which is smooth by assumption~\eqref{e:conda}.
Existence and uniqueness of a classical solution of the initial-boundary value problem~\eqref{eq:auxiliary-system} follows then from classical results on linear parabolic equations, see for instance~\cite[\S 7.1]{Evans}. \\
\step{we establish~\eqref{e:elleduek}}
First, we multiply the equation at the first line of~\eqref{e:sy:itera}
times $\varphi_{n+1}$ 
and we integrate in space and we use the homogeneous
Neumann boundary conditions. We arrive at 
\begin{equation}
\label{e:moltiplico}
\begin{split}
    \frac{1}{2} \frac{d}{dt} \int_0^R \varphi_{n+1}^2 (t, x) dx & +
    \int_0^R \big( \partial_x \varphi_{n+1} \big)^2 (t, x) dx \\=& -
     \int_0^R a (t, x) \varphi_{n+1}^2   (t, x) dx 
     + 
    \int_0^R f_n (t, x) \varphi_{n+1}^2   (t, x) dx \\
     \stackrel{a \ge 0}{\leq}&
     \int_0^R f_n (t, x) \varphi_{n+1}^2   (t, x) dx.
\end{split}
\end{equation}
Next, we recall the definition~\eqref{e:effenne}  of $f_n$, the fact that by assumption $a_n \ge 0$ and hypothesis~\eqref{eq:2.3}. We conclude that 
$$
    f_n (t, x) \leq \max \Big\{ f(t, x, \fhi): 
    (t, x) \in [0, T] \times [0, R], \; 0 \leq \fhi \leq h(t, x) \Big\}. 
$$
Owing to condition~\eqref{eq:2.2}, 
$$
    \max \Big\{ f(t, x, \fhi): 
    (t, x) \in [0, T] \times [0, R], \; 0 \leq \fhi \leq h(t, x) \Big\} \leq F,
$$
where $F$ is the same as in~\eqref{e:finty}. Hence, $f_n \leq F$ and from~\eqref{e:moltiplico} we arrive at 
\begin{equation}
\label{e:moltiplico2}
\begin{split}
    \frac{1}{2} \frac{d}{dt} \int_0^R \varphi_{n+1}^2 (t, x) dx & +
    \int_0^R \big( \partial_x \varphi_{n+1} \big)^2 (t, x) dx \leq 
    F \int_0^R \varphi_{n+1}^2   (t, x) dx. 
\end{split}
\end{equation}
Owing to the Gronwall Lemma,~\eqref{e:moltiplico2} implies that 
$$
   \int_0^R \varphi_{n+1}^2 (t, x) dx \leq 
    e^{Ft}\int_0^R \varphi_0^2 (x) dx + K e^{Ft}
$$
and by plugging the above inequality into~\eqref{e:moltiplico2} and integrating with respect to time we arrive at~\eqref{e:elleduek}. \\
\step{we establish~\eqref{e:maxprink}} First, we define the $C^2$ function $\beta: \R \to \R$ by setting
\begin{equation}
\label{e:beta}
    \beta (w) : =
    \left\{
    \begin{array}{ll}
    w^4 & w <0, \\
    0 & w \ge 0.
    \end{array}
    \right.
\end{equation}
We multiply the equation at the first line
of~\eqref{eq:auxiliary-system} times $\beta'(\varphi_{n+1})$
and we integrate in space. By using the relation $\beta'(w) w= 4 \beta(w)$ and proceeding as at the previous step we infer that
\begin{equation}
\label{e:gl2}
    \frac{d}{dt} \int_0^R \beta \big(\varphi_{n+1} \big)(t, \cdot) dx
    \leq C(F) \int_0^R \beta \big(\varphi_{n+1} \big)(t, \cdot) dx.  
\end{equation}
By assumption~\eqref{e:condizionidatoiniziale}, $\varphi_0 \ge 0$ and hence $\beta (\varphi_0) =0$. This implies 
that by combining~\eqref{e:gl2} with the Gronwall Lemma we get that $ \beta (\varphi_{n+1}) \equiv 0$, namely $\varphi_{n+1} \ge 0$. This concludes the proof of the lemma. 
\end{proof}
\subsubsection{Limit analysis}
\label{sss:conex}
This paragraph aims at establishing the following result. 
\begin{lemma}
\label{l:compactness}
Assume that the hypotheses of the Theorem~\ref{p:classical} are satisfied 
and furthermore that ${\fhi_0 \in C^\infty ([0, R])}$. Let $\fhi_n$ be the sequence of functions recursively defined in \S~\ref{ss:exa}. Then 
\begin{equation}
\label{e:convergence}
     \fhi_n \to \fhi \quad \text{strongly in $L^2 (]0, T[ \times ]0, R[)$}
\end{equation}
for some limit function $\fhi$ satisfying 
\begin{equation}
\label{e:reglim}
   \| \varphi (t, \cdot) \|^2 
   _{L^2(]0, R[) } +  \int_0^{T} \int_0^R (\partial_x \varphi)^2 
   dx \, ds     \leq C(T, F) \| \varphi_0 \|^2_{L^2(]0, R[) }  
   \quad \text{for every $t \in [0, T]$.}
\end{equation}
Also, $\fhi$ is a weak solution of~\eqref{eq:auxiliary-system}.  
\end{lemma}
\begin{proof}
We proceed according to the following steps. \\
\firststep
\step{we establish compactness of the sequence $\varphi_n$} We first outline our argument: we want to apply the Aubin-Lions Lemma~\cite{Si}. We recall the Rellich
Theorem, which gives that the inclusion $H^1(]0, R[) \hookrightarrow L^2 (]0, R[)$ is compact.
 We term $H^\ast (]0, R[)$ the dual space of $H^1(]0, R[)$, endowed with the standard dual norm. By putting together all the previous considerations we conclude that we have the following chain of inclusions, 
$$
    H^1 (]0, R[) \hookrightarrow L^2 (]0, R[) \hookrightarrow H^\ast (]0, R[), 
$$ 
and the first inclusion is compact and the second is continuous. Owing to the Aubin-Lions Lemma, to establish the compactness of the sequence $\{ \fhi_n \}$ in $L^2 (]0, T[ \times ]0, R[)$ it suffices to show that
\begin{eqnarray}
\label{e:suffcond}
     \text{$\{ \fhi_n \}$ is uniformly bounded in $L^2 (]0, T[; H^1(]0, R[))$}, 
     \label{e:unibd1}
     \\
     \text{$\{ \pt \fhi_n \}$ is uniformly bounded in $L^2 (]0, T[; H^\ast (]0, R[))$}.
     \label{e:unibd2}  
\end{eqnarray} 
Condition~\eqref{e:unibd1} immediately follows from~\eqref{e:elleduek}. To establish~\eqref{e:unibd2} we multiply the equation at the first line of~\eqref{e:sy:itera} times a test function $v \in H^1 (]0, R[)$ and we integrate with respect to $x$.
By using the Integration by Parts Formula and recalling that $\fhi_{n+1}$ satisfies homogeneous boundary conditions we obtain 
$$
   \int_0^R \pt \fhi_{n+1} \cdot v dx=
   - \int_0^R \px \fhi_{n+1} \cdot  \px v dx -
    \int_0^R a \fhi_{n+1} v dx - 
    \int_0^R f_n \fhi_{n+1} v dx
$$ 
and hence 
\begin{equation}
  \label{e:boundsufhit}
  \begin{split}
    \left| \int_0^R \pt \fhi_{n+1} v dx \right| & \leq \| \px
    \fhi_{n+1} (t, \cdot) \|_{L^2 (]0, R[)} \| \px v \|_{L^2 (]0, R[)}
    \\
    & \quad + \| a(t, \cdot) \|_{L^1 (]0, R[)} \| \fhi_{n+1} (t, \cdot)
    \|_{L^\infty (]0, R[)}
    \| v \|_{L^\infty (]0, R[)} 
    \\
    & \quad + \|v \|_{L^\infty (]0, R[)} \left| \int_0^R
      f_n \fhi_{n+1} dx \right| \\
    & \stackrel{\eqref{e:immersione}}{\leq} \| \px \fhi_{n+1} (t,
    \cdot) \|_{L^2 (]0, R[)} \| v \|_{H^1(]0, R[)}
    \phantom{\int} \\
    & \quad + C(R) \sup_{t \in ]0, T[} \| a(t, \cdot) \|_{L^1 (]0,
      R[)} \| \fhi_{n+1} (t, \cdot) \|_{H^1 (]0, R[)}
    \| v \|_{H^1(]0, R[)} \\
    & \quad + C(R) \|v \|_{H^1 (]0, R[)} \left| \int_0^R f_n
      \fhi_{n+1} dx \right|.
  \end{split}
\end{equation}
Owing to~\eqref{e:boundsueffe}, 
\begin{equation}
\begin{split}
\left| \int_0^R f_n \fhi_{n+1} dx \right| & \stackrel{\eqref{e:boundsueffe}}{\leq}
\alpha_1 \int_0^R |\fhi_n \fhi_{n+1}| dx + 
F \int_0^R |\fhi_{n+1}| dx \\
&
 \stackrel{\text{H\"older},\eqref{e:elledueelleuno}}{\leq}
 \alpha_1 \| \fhi_n (t, \cdot) \|_{L^2 (]0, R[)}
 \| \fhi_{n+1} (t, \cdot) \|_{L^2 (]0, R[)} +
 C(R, F)  \| \fhi_{n+1} (t, \cdot) \|_{L^2 (]0, R[)} \\
 & \stackrel{\eqref{e:elleduek}}{\leq}  
  C(\alpha_1, T, R, F) 
 \| \fhi_0 \|_{L^2 (]0, R[)}.   
\end{split}
\end{equation}
 By plugging the previous inequality into~\eqref{e:boundsufhit} and using~\eqref{e:elleduek}
 we conclude that
 \begin{equation}
 \label{e:arrivati!}
     \| \partial_t \fhi_{n+1} \|_{L^2 (]0, T[); H^\ast (]0, R[)}
     \leq C(\alpha_1, T, F, R) 
      \| \fhi_0 \|_{L^2 (]0, R[)} 
      \Big[ 1+   \sup_{t \in ]0, T[} \| a(t, \cdot) \|_{L^1 (]0, R[)} 
      \Big].
 \end{equation}
 This establishes~\eqref{e:unibd2} and hence shows that the sequence $\{ \fhi_n \}$ is compact in $L^2 (]0, T[\times ]0, R[)$. \\
 \step{we show that every
   accumulation point $\varphi$ is a weak solution of~\eqref{eq:auxiliary-system}} Owing to {\sc Step 1} we have that, up to subsequences, 
 \begin{equation}
 \label{eq:convergence2}
    \fhi_n \to \fhi \quad \text{strongly in $L^2 (]0, T[\times ]0, R[)$}
 \end{equation}
 for some accumulation point $\fhi$. Note that, owing to~\eqref{e:elleduek}, the sequence 
 $\{ \px \fhi_n \}$ is uniformly bounded in $L^2 (]0,T[ \times ]0, R[)$
 and hence, by the Urysohn Lemma,
 \begin{equation}
 \label{eq:convergencew}
     \px \fhi_n \weak \px \fhi \quad \text{weakly in $L^2 (]0, T[\times ]0, R[)$}. 
 \end{equation}
Owing to~\eqref{e:elleduek} and to the lower semicontinuity of the norm with respect to weak convergence, we have that $\fhi$ satisfies~\eqref{e:reglim} and hence, in particular, 
$\fhi \in L^2 (]0, T[; H^1 (]0, R[))$. 

We now show that $\fhi$ is a weak solution of~\eqref{eq:auxiliary-system}. Owing to the uniqueness result established in~\S~\ref{sss:uni}, this implies that the accumulation point $\fhi$ is unique and hence that the
convergence~\eqref{eq:convergence2} holds for the whole sequence
$\{ \fhi_n \}$. 

First, we point out that, owing to~\eqref{eq:2.2}, 
$$
    |f_n (t, x) - f(t, x, \fhi) |
    \stackrel{\eqref{e:effenne}}{=}
    |f(t, x, \fhi_n) - f(t, x, \fhi)| \stackrel{\eqref{eq:2.2}}{\leq} 
    \alpha_1 |\fhi_n - \fhi| \quad 
    \text{for a.e. $(t, x) \in ]0, T[ \times ]0, R[$}
 $$
and hence, owing to~\eqref{eq:convergence2},  
\begin{equation}
\label{eq:convergence3}
    f_n \to f(\cdot, \cdot, \fhi) 
     \quad \text{strongly in $L^2 (]0, T[\times ]0, R[)$}. 
\end{equation}
To conclude, we fix a test function
$v \in C^\infty_c (]-\infty, T[ \times \R)$,
we use the equation at the first line of~\eqref{e:sy:itera}, we integrate in space and time and we arrive at 
\begin{equation*}
\begin{split}
\int_0^T \!\int_0^R &
 \left(\pt v \,  \vfi_{n+1}-\px v\, \px \vfi_{n+1}
  \right) dx dt-\int_0^T\! \int_0^R v \, \vfi_{n+1} \, a \, dx dt 
+\int_0^T \! \int_0^R v f_n \vfi_{n+1} \, dxdt \\
& \quad +
\int_0^Rv(0, \cdot )\, \vfi_0 \, dx=0.
\end{split}
\end{equation*}
Owing to~\eqref{eq:convergence2},~\eqref{eq:convergencew} and~\eqref{eq:convergence3}, we can pass to the limit in all the terms in the previous expression and obtain~\eqref{e:distrform}. This shows that $\fhi$ is a 
weak solution of~\eqref{eq:auxiliary-system} and concludes the proof of Lemma~\ref{l:compactness}.  
\end{proof}
\subsubsection{Conclusion of the proof of Theorem~\ref{p:classical}}
\label{sss:concludiamo}
First, we establish~\eqref{e:maxprin}. The bound from below is a direct consequence 
of~\eqref{e:maxprink} and~\eqref{eq:convergence2}. To establish the bound from above, we recall that $a(t, x) \ge 0$ for every $(t, x)$. Next, we point out that, 
owing to~\eqref{eq:2.3} and~\eqref{e:emme}, $f(t, x, M) \leq 0$. These considerations  
imply that the function identically equal to $M$ is a supersolution of the initial-boundary value problem~\eqref{eq:auxiliary-system} and hence that $\fhi \leq M$. 

To conclude the proof of Theorem~\ref{p:classical}, we are left to remove the assumption that $\fhi_0 \in C^\infty([0, R])$ from the existence part. We fix $\fhi_0 \in L^\infty(]0, R[)$, $\fhi_0 \ge0$, and we construct a sequence of functions $\{ \fhi_{0k} \}$ such that 
\begin{equation}
  \fhi_{0k} \in C^\infty([0, R]), \quad
  \fhi_{0k} \to \fhi_0 \; \text{strongly in $L^2(]0, R[)$}, \quad
  0 \leq \fhi_{0k} (x) \leq \| \fhi_{0} \|_{L^\infty (]0, R[)} \; 
  \forall x \in ]0, R[.
\end{equation}
We apply Lemma~\ref{l:compactness} and we term $\fhi_k$
the corresponding sequence of solutions of the initial-boundary
value problem~\eqref{eq:auxiliary-system}
with initial condition given by
$\fhi_{0k}$. Note that $\fhi_k$ satisfies 
\begin{equation}
\label{e:ibvpkappa}
\begin{split}
\int_0^T \!\int_0^R &
 \left(\pt v \,  \vfi_k-\px v\, \px \vfi_k
  \right) dx dt-\int_0^T\! \int_0^R v \, \vfi_k \, a \, dx dt 
+\int_0^T \! \int_0^R v (t, x) f (t, x, \fhi_k) \, dxdt \\
& \quad +
\int_0^Rv(0, \cdot )\, \vfi_{0k} \, dx=0.
\end{split}
\end{equation}
Next, we choose $\hat a=a$ and we use the stability estimate~\eqref{eq:th2.3}. We conclude that the sequence 
$\{ \fhi_k \}$ is a Cauchy sequence in $L^2 (]0, T[; H^1 (]0, R[)$. We term $\fhi$ its limit and by passing to the limit in~\eqref{e:ibvpkappa} we get~\eqref{e:distrform}. This shows that 
$\fhi$ is a weak solution of~\eqref{eq:auxiliary-system} and hence concludes the proof of Theorem~\ref{p:classical}.   
\section{A parabolic problem with measure-valued coefficients}
\label{sec:3}
In this section we establish the well-posedness of the following parabolic initial-boundary
value problem with a time-dependent, measure-valued coefficient $\mu_t$:
\begin{equation}
  \label{eq:2}
  \begin{cases}
    \pt\vfi = \pxx\vfi - \vfi \mu_t + f(t,x,\vfi) \fhi,
    & \text{in $]0, T[ \times ]0, R[$},
    \\
      \px \fhi(t,0)=\px\vfi(t,R) = 0, & t \in \, ]0, T[,
     \\
    \vfi(0,x)=\vfi_0(x),
    & x\in{]0,R[}.
  \end{cases}
\end{equation}
We introduce the following hypothesis on the coefficient $\mu_t$.
\begin{itemize}
\item[({\bf H.2})] \label{h:accatre}
The measured valued coefficient  $\mu_t$ satisfies the following conditions: 
\begin{itemize}
\item[i)] for a.e. $t \in ]0, T[$, we have $\mu_t \in \mathcal M^+ (]0, R[).$
\item[ii)] For every Borel set $B \subset ]0, R[$, the map
$
      t \mapsto \mu_t  (B)
$
is $\mathcal L^1$-measurable. 
\item[iii)] We have 
\begin{equation}
\label{e:h:finito}
    \mathrm{ess \; sup}_{0\le t\le T} \| \mu_t \|_{\mathcal M (]0,R[)} < + \infty. 
\end{equation}
\end{itemize}
\end{itemize}
In the following, we denote
by $L^\infty (]0, T[; \mathcal M(]0, R[))$
the space of time-dependent measures 
satisfying {\bf(H2)}. Also, we denote by  
$\mu$ the measure defined by setting
\begin{equation}
\label{e:cosaemu}
    \mu (E) : = \int_0^T \! \! \int_0^R 
    \mathbbm{1}_E (t, x) d \mu_t (x) dt,
\end{equation}
where $\mathbbm{1}_E $ denotes the characteristic function of $E$. 
Owing to~\cite[Proposition 2.26]{AmbrosioFuscoPallara}, $\mu$ is a Borel measure on $]0, T[ \times ]0, R[$. The following result is established by using standard techniques. For completeness, we provide a sketch of the proof. 
\begin{lemma}
\label{l:measure}
   Let $\fhi$ be a function satisfying~\eqref{e:dentrospazi}. Then $\fhi$ is $\mu$-measurable and summable and 
   \begin{equation}
   \label{e:Fubini}
       \int_{]0, T[ \times ]0, R[} \fhi(t, x) d \mu(t, x)
       =
       \int_0^T \! \! \int_0^R 
       \fhi(t, x) d \mu_t (x) dt.
   \end{equation} 
\end{lemma}
Also, 
\begin{equation}
   \label{e:Fubini2}
      \left| \int_{]0, T[ \times ]0, R[} \fhi(t, x) d \mu(t, x) \right|
       \leq C(R)  \esssup_{0\le t\le T}
       \| \mu_t \|_{\mathcal M (]0,R[)} \| \fhi \|_{L^2 (]0, T[; H^1 (]0, R[)}.
  \end{equation} 
\begin{proof}
First, we recall that, owing to the inclusion
$H^1 (]0, R[)  \hookrightarrow C^0([0, R])$, the function $\fhi(t, \cdot)$ is
continuous for a.e. $t \in ]0, T[$.  
Next, we introduce a convolution kernel and we construct a sequence of continuous functions such that
\begin{eqnarray}
    \fhi_n \to \fhi \quad \text{in $L^2 (]0, T[; H^1 (]0, R[))$}, 
     \label{e:meas1} \\
    \fhi_n(t, \cdot) \to \fhi (t, \cdot) 
    \quad \text{in $C^0 ([0, R])$ for a.e. $ t \in ]0, T[$}.
    \label{e:meas2}
\end{eqnarray}
Owing to~\eqref{e:meas2}, we have 
\begin{equation}
\label{e:meas3}
\fhi_n (t, x) \to \fhi (t, x) 
\quad \text{for $\mu$-a.e. $(t, x)$}. 
\end{equation}
Since the functions $\fhi_n$ are continuous and henceforth $\mu$-measurable, the function $\fhi$ is $\mu$-measurable. We now show that 
$\fhi$ is $\mu$-summable. Owing to the analysis in~\cite[\S2.5]{AmbrosioFuscoPallara}, we have 
\begin{equation}
  \label{e:succdicosci}
  \begin{split}
    \left| \int_{]0, T[ \times ]0, R[} \right. & \left| \fhi_n -
      \fhi_m \right| (t, x) d \mu(t, x) 
    \stackrel{\text{\cite[\S2.5]{AmbrosioFuscoPallara}}}{=} 
      \int_0^T \! \! \int_0^R \left| \fhi_n - \fhi_m \right| (t, x) d
      \mu_t(x) dt 
    \\
    & \leq \int_0^T \| \fhi_n(t,\cdot) - \fhi_m (t,\cdot)\|_{C^0 ([0,
      R])}
    \| \mu_t \|_{\mathcal M (]0, R[)}dt
    \\
    & \stackrel{\eqref{e:immersione}}{\leq} C(R) \esssup_{0\le t\le T}
    \| \mu_t \|_{\mathcal M (]0,R[)} \int_0^T \| \fhi_n (t,\cdot)-
    \fhi_m(t,\cdot) \|_{H^1 (]0, R[)}
    dt
    \\
    & \stackrel{\eqref{e:elledueelleuno}}{\leq} C(R) \sqrt{T} \esssup_{0\le
      t\le T} \| \mu_t \|_{\mathcal M (]0,R[)}
    \| \fhi_n - \fhi_m \|_{L^2 (]0, T[; H^1 (]0, R[))} 
    \\
    & \stackrel{\eqref{e:meas1}}{\longrightarrow} 0 \quad \text{as $n,
      m \to + \infty$}.
  \end{split}
\end{equation}
This implies that $\{ \fhi_n \}$ is a Cauchy sequence in $L^1(]0, T[ \times ]0, R[, \mu)$. Owing to~\eqref{e:meas3}, the limit is $\fhi$, which is henceforth a $\mu$-summable function. Finally,  by passing to the limit in the equality 
$$
    \int_{]0, T[ \times ]0, R[}  \fhi_n (t, x) d \mu(t, x)  
      \stackrel{\text{\cite[\S2.5]{AmbrosioFuscoPallara}}}{=} 
       \int_0^T \! \! \int_0^R
       \fhi_n  (t, x) d \mu_t(x) dt  
      $$
      we arrive at~\eqref{e:Fubini} and~\eqref{e:Fubini2} follows by arguing as in~\eqref{e:succdicosci}. 
\end{proof}
We can now provide the definition of weak solution of~\eqref{eq:2}.
\begin{definition}
  \label{def:sol}
  Assume {\rm  {\bf (H1)}} and {\rm {\bf (H2)}}. 
  A function $\vfi:]0,T[ \times ]0,R[ \to\R$ is a \emph{weak solution} of~\eqref{eq:2} if $\fhi$ 
  satisfies~\eqref{e:dentrospazi} and for every test function $v \in C_c^\infty(]-\infty,T[ \times \R)$ we have
  \begin{equation}
\label{eq:2.9}
\begin{split}
\int_0^T \! \! \int_0^R & \left(\pt v \, \vfi - 
\px v\, \px \vfi\right) dx dt-\int_0^T \! \!  \int_0^Rv(t, x) \vfi(t, x)  d\mu_t(x)dt 
 +\int_0^T \! \!  \int_0^R v (t, x) \, f(t,x,\vfi) \fhi (t, x) dxdt \\
 & +\int_0^Rv(0,x)\vfi_0(x)dx=0.
\end{split}
\end{equation}
\end{definition}
Note that the second term in the above expression is well defined owing to Lemma~\ref{l:measure}. 
We now state the main result of the present section. 
\begin{theorem}
  \label{th:main1}
  Assume \emph{({\bf H.1})} and \emph{({\bf H.2})} and assume furthermore that the initial datum $\fhi_0$ satisfies~\eqref{e:condizionidatoiniziale}.
  Then the initial-boundary value problem \eqref{eq:2} admits a unique weak solution. Also, this solution enjoys the following 
  properties: first, 
  \begin{align}
    \label{eq:th2.1}
    &0\le \vfi(t,x) \le M ,\quad \text{for a.e.  $(t,x)\in ]0,T[ \times ]0,R[$},
  \end{align}
  where $M$ is the same constant as in~\eqref{e:emme}. Second, we have stability with respect to the initial datum and with respect to the coefficient $\mu$. Namely, if we term  
$\widehat\vfi$ the solution of the initial-boundary value problem
\begin{equation}
  \label{eq:fhicappuccio}
  \begin{cases}
    \pt \hat \vfi = \pxx \hat \vfi - \hat \vfi \hat \mu_t + f(t,x,\hat \vfi) \hat 
    \fhi,
    & \text{in $]0, T[ \times ]0, R[$},
    \\
      \px \hat \fhi(t,0)=\px  \hat \vfi(t,R) = 0, & t \in \, ]0, T[,
     \\
    \hat \vfi(0,x)=\hat \vfi_0(x),
    & x\in{]0,R[}.
  \end{cases}
\end{equation}
then  we have 
  \begin{equation}
    \label{e:stability2}
    \begin{split}
            & \norm{\vfi(t,\cdot)-  
            \widehat\vfi(t,\cdot)}^2_{L^2(]0,R[)}+  
             \int_0^t  \norm{ \partial_x \vfi(s,\cdot) -  
             \partial_x \widehat\vfi(s,\cdot)}^2_{L^2(]0,R[)}ds\\
    & \qquad \le 
    C(\alpha_1, M, T, R, F) 
    \Bigg[  \mathrm{ess \;sup}_{t \in [0, T]}
    \| \mu_t   - \hat \mu_t  \|^2_{\mathcal M(]0, R[)} + 
     \norm{\vfi_0-\widehat\vfi_0}^2_{L^2(]0,R[)}  
     \Bigg]
     \\
    &  \qquad \qquad \qquad 
    \text{for every $0\le t\le T$.} \phantom{\int} \\
\end{split}
  \end{equation} 
\end{theorem} 
\begin{proof}
To establish the uniqueness of the solution of~\eqref{eq:2} and the stability estimate~\eqref{e:stability2} we can follow the same argument as in \S~\ref{sss:uni}.
We are left to establish existence and we proceed according to the following steps. \\
\firststep
\step{we construct a sequence of approximate solutions} First, we take
a sequence $\{ a_j \}$ such that for every $j$ we have 
\begin{equation}
  \label{e:akappa}
  \begin{split}
    & \qquad \qquad \qquad a_j \in C^\infty ([0, T] \times [0, R]),
    \qquad a_j \ge 0, 
    \\ 
    & a_j (t, \cdot) \weaks \mu_t \; \text{
      weakly$^\ast$ in $\mathcal M (]0, R[)$, \quad $\| a_j (t, \cdot)
      \|_{L^1 (]0, R[)} \leq \esssup_t \| \mu_t \|_{\mathcal M(]0,
        R[)}$ for a.e. $t \in ]0, T[$}.
  \end{split}
\end{equation}
We apply Theorem~\ref{p:classical} and we term $\fhi_j$ the solution of the initial-boundary value problem~\eqref{eq:auxiliary-system} in the case when $a=a_j$.  By applying the stability estimate~\eqref{eq:th2.3} with 
$\hat a = a_j$, $\hat \fhi_0 =0$ we obtain that
\begin{equation}
\label{e:bdunifgei}
  \| \fhi_j (t, \cdot)\|^2_{L^2 (]0, R[)}
  + \int_0^T \int_0^R (\px \fhi_j )^2 dx dt \leq 
  C(\alpha_1, M, T, R, F)  \| \fhi_0 \|^2_{L^2 (]0, R[)}
  \quad \text{for a.e. $t \in ]0, T[$}. 
\end{equation}
\step{we establish compactness of the sequence $\{ \fhi_j \}$} We apply the Aubin-Lions Lemma. First, we point out  that 
we have the following chain of inclusions 
$$
    H^1 (]0, R[) \hookrightarrow C^0 ([0, R]) \hookrightarrow H^\ast (]0, R[), 
$$ 
and the first inclusion is compact and the second is continuous. Owing to the Aubin-Lions Lemma, to establish the compactness of the sequence $\{ \fhi_j \}$ in $L^2 (]0, T[; C^0 ([0, R])$ it suffices to show that
\begin{eqnarray}
     \text{$\{ \fhi_j \}$ is uniformly bounded in $L^2 (]0, T[; H^1(]0, R[))$}, 
     \label{e:unibd12}
     \\
     \text{$\{ \pt \fhi_j \}$ is uniformly bounded in $L^2 (]0, T[; H^\ast (]0, R[))$}.
     \label{e:unibd22}  
\end{eqnarray} 
To establish~\eqref{e:unibd12} we can use~\eqref{e:bdunifgei}. To establish~\eqref{e:unibd22}, we use the equation at the first line 
of~\eqref{eq:auxiliary-system}, repeat the same computations as in {\sc Step 1}  of the proof of Lemma~\ref{l:compactness}, arrive at~\eqref{e:arrivati!} and then use the forth estimate in~\eqref{e:akappa}. 

By relying on the Aubin-Lions Lemma we conclude that, up to subsequences, we have 
\begin{equation}
\label{e:al:convergenza1}
   \fhi_j  \to \fhi  \quad 
   \text{strongly in $L^2 (]0, T[; C^0 ([0, R]))$ }
\end{equation}
for some accumulation point  $\fhi  \in L^2 (]0, T[; C^0 ([0, R]))$
and hence, in particular, 
\begin{equation}
\label{e:al:convergenza2}
   \fhi_j  \to \fhi  \quad 
   \text{strongly in $L^2 (]0, T[ \times ]0, R[)$}.
\end{equation}
By extracting, if needed, a further subsequence, we have 
\begin{equation}
 \fhi_j (t, \cdot) \to \fhi (t, \cdot) 
     \; \text{uniformly in $ C^0 ([0, R])$ for a.e. $t \in [0, T]$}
     \\
     \label{e:convergence2} 
\end{equation}
and, owing to~\eqref{e:maxprin}, this implies, in particular, that $\fhi$ satisfies~\eqref{eq:th2.1}.  
Finally, we recall that $\px \fhi_j$ is uniformly bounded in $L^2 (]0, T[ \times ]0, R[)$ owing to~\eqref{e:bdunifgei} and we infer that, by the Urysohn Lemma,
\begin{equation}
\label{e:convergence3}
     \partial_x \fhi_j \weak \partial_x \fhi \;   
     \text{weakly in $L^2 ( ]0, T[ \times ]0, R[)$}.   
\end{equation}
By combining~\eqref{e:bdunifgei},~\eqref{e:convergence2} and the lower-semicontinuity of the norm with respect two weak convergence we infer that
$\fhi$ satisfies 
\begin{equation}
\label{e:bdunifgei2}
  \| \fhi (t, \cdot)\|^2_{L^2 (]0, R[)}
  + \int_0^T \int_0^R (\px \fhi)^2 dx dt \leq 
  C(\alpha_1, M, T, R, F)  \| \fhi_0 \|^2_{L^2 (]0, R[)}
  \quad \text{for a.e. $t \in ]0, T[$}. 
\end{equation}
\step{we show that the accumulation point $\fhi$ is a weak solution of~\eqref{eq:2}}
First, we point out that $\fhi \in L^2 (]0, T[; H^1 (]0, R[))$ owing to~\eqref{e:bdunifgei2}.
Next, we fix a test function $v \in C^\infty (]- \infty, T[ \times \R)$ and we recall that $\fhi_j$ satisfies ~\eqref{e:distrform}. We now pass to the limit in each of the terms in~\eqref{e:distrform}. 
Owing to~\eqref{e:al:convergenza2} and ~\eqref{e:convergence3}, we have 
\begin{equation}
  \label{e:primotermine}
  \int_0^T \! \!  \int_0^R \!\!
  \pt v \fhi_j dx dt \to \!\! 
  \int_0^T \! \!  \int_0^R \!\!
  \pt v \fhi dx dt, \,\,
  \int_0^T \! \!  \int_0^R \!\!
  \px v  \px \fhi_j dx dt \to \!\! 
  \int_0^T \! \!  \int_0^R \!\!
  \px v  \px \fhi dx dt, 
  \, \text{as $j \to + \infty$. }
\end{equation}
Next, we point out that
$$
 |f(t, x, \fhi_j) - f(t, x, \fhi)| \stackrel{\eqref{eq:2.2}}{\leq} 
    \alpha_1 |\fhi_j - \fhi| 
$$
and hence, owing to~\eqref{e:al:convergenza2}, 
\begin{equation}
\label{e:secondotermine}
   \int_0^T \! \!  \int_0^R
 f(t, x, \fhi_j) \fhi_j (t, x) v(t, x) dx dt \to 
  \int_0^T \! \!  \int_0^R
  f(t, x, \fhi) \fhi (t, x) v(t, x)
   dx dt
  \quad \text{as $j \to + \infty$. }
\end{equation}
Finally, we consider the last term. First, we point out that 
\begin{equation}
\label{e:al:ultimo}
\begin{split}
    \left| \int_0^R 
    a_j (t, x) \fhi_j  (t, x)dx 
    - \int_0^R  \fhi(t, x) d \mu_t (x)
    \right| & \leq 
    \underbrace{\left| \int_0^R 
    a_j (t, x) \fhi (t, x) dx
    - \int_0^R  \fhi(t, x) d \mu_t (x)
    \right|}_{T^j_1(t)} \\
    & +
    \underbrace{\left| \int_0^R 
    a_j \big[ \fhi_j  - \fhi \big] dx 
    \right| }_{T^j_2(t)}.
\end{split}
\end{equation}
We first deal with the term $T^j_1$: we recall that, owing to~\eqref{e:al:convergenza1}, the function $\fhi(t, \cdot)$ is continuous for a.e. $t \in ]0, T[$. We recall the third property in~\eqref{e:akappa} and we conclude that as $j \to + \infty$ $T^j_1 (t) \to 0$ for a.e. $t \in ]0, T[$. Next, we point out that for a.e. $t \in ]0, T[$
\begin{equation*}
\begin{split}
   T^j_1 (t) \leq 
   \big[ \| a_j (t, \cdot) \|_{L^1 (]0, R[)} + \| \mu_t \|_{\mathcal M(]0, R[)}
    \big] \| \fhi \|_{L^\infty (]0, R[)} 
   & \stackrel{\eqref{e:akappa}}{\leq} 
    2 \mathrm{ess \, sup}_{t \in ]0, T[} 
     \| \mu_t \|_{\mathcal M(]0, R[)} \| \fhi \|_{L^\infty (]0, R[)}\\
     &
    \stackrel{\eqref{eq:th2.1}}{\leq} 
    2 \mathrm{ess \, sup}_{t \in ]0, T[} \| \mu_t \|_{\mathcal M(]0, R[)} 
    M.
\end{split}
\end{equation*}
Owing to the Lebesgue Dominated 
Convergence Theorem, this implies that 
\begin{equation}
\label{e:tiuno}
 \int_0^T T^j_1 (t) dt \to 0 \quad \text{as $j \to + \infty$. }
\end{equation} 
Next, we deal with the term $T^j_2$: we first point out that 
\begin{equation}
\label{e:bdtidue}
\begin{split}
   T^j_2 (t) \leq 
   \| a_j (t, \cdot) \|_{L^1 (]0, R[)} \| \fhi_j - \fhi \|_{C^0 (]0, R[)}
   & \stackrel{\eqref{e:akappa}}{\leq} 
     \mathrm{ess \, sup}_{t \in ]0, T[} 
      \| \mu_t \|_{\mathcal M(]0, R[)}
     \| \fhi_j - \fhi \|_{C^0 (]0, R[)} 
     \end{split}
\end{equation}
 and hence, owing to~\eqref{e:convergence2}, for a.e. $t \in ]0, T[$,
 $T^j_2(t) \to 0$ as $j \to + \infty$.  By using again~\eqref{e:bdtidue} we get 
 $$
     T^j_2 (t) \leq 
     \mathrm{ess \, sup}_{t \in ]0, T[} 
      \| \mu_t \|_{\mathcal M(]0, R[)}
     \| \fhi_j - \fhi \|_{C^0 (]0, R[)} \stackrel{\eqref{eq:th2.1}}{\leq}
     2   \mathrm{ess \, sup}_{t \in ]0, T[} 
      \| \mu_t \|_{\mathcal M(]0, R[)} M
 $$
and hence, by applying the Lebesgue Dominated 
Convergence Theorem, 
we get 
\begin{equation}
\label{e:tidue}
   \int_0^T  T^j_2 (t) dt \to 0 
   \quad \text{as $j \to + \infty$}.
\end{equation}
By combining~\eqref{e:al:ultimo},~\eqref{e:tiuno} and~\eqref{e:tidue} we get that 
$$
    \int_0^T     \int_0^R 
    a_j (t, x) \fhi_j  (t, x)dx dt 
    - \int_0^T \int_0^R  \fhi(t, x) d \mu_t (x)
    dt \to 0 \quad \text{as $j \to + \infty$}.
$$
We recall~\eqref{e:primotermine} and~\eqref{e:secondotermine} and we eventually conclude that $\fhi$ satisfies~\eqref{eq:2.9}. This concludes the proof of Theorem~\ref{th:main1}.  
\end{proof}
\section{Necessary conditions for optimality}
\label{sec:4}
This section aims at discussing existence of optimal strategies $\mu$ 
for the payoff functional 
\begin{equation}
  \label{eq:3.1bis}
  \begin{split}
    & \qquad J : L^\infty
    \big(]0, T[; \mathcal M_+ (]0, R[) \big) \to \R
    \\
    & J(\mu):=\int_0^T \! \! \! \int_0^R \vfi(t,x) d\mu_t(x)dt-
    \Psi\left(\int_0^T \! \! \! \int_0^R c(t,x)d\mu_t(x)dt\right).
  \end{split}
\end{equation}
In the above expression, $\fhi$ is the weak solution of the initial-boundary value problem~\eqref{eq:2}. Note that $\fhi$ depends on $\mu$, and hence the functional $J$ is nonlinear. The functions $c$ and $\Psi$ satisfy the following hypotheses. 
\begin{itemize}
\item[({\bf H.3})] \label{h:acca3} The function $c: [0, T] \times [0, R]
\to \R^+ \cup \{+ \infty \}$ is lower semicontinuous. 
\end{itemize}
Note that hypothesis {\bf (H.3)} implies that the function $c$ is $\mu$-integrable (in the sense of~\cite[p.8]{AmbrosioFuscoPallara}) and hence that the second term in~\eqref{eq:3.1bis} is well-defined. 
\begin{itemize}
\item[({\bf H.4})] \label{h:accaquattro} The function
  $\Psi: \R \cup\{+\infty\} \to \R$ is twice
  continuously differentiable, nondecreasing and convex. \\
\end{itemize}
In the following we aim at discussing the existence and uniqueness of an optimal $\mu$ for $J$ under the constraint 
\begin{equation}
\label{eq:3.2}
\int_0^R b(t,x)  d\mu_t(x) \leq  1 \quad \text{for a.e. $t \in ]0, T[$}. 
\end{equation}
The function $b$ satisfies the following hypothesis: 
\begin{itemize}
\item[({\bf H.5})] 
\label{h:acca5}
The function $b: [0, T] \times [0, R]
\to \R \cup \{+ \infty \}$ is lower semicontinuous and bounded away from $0$, namely
\begin{equation}
\label{eq:3.3}
b(t,x)\ge b_0>0, \quad \hbox{for every~} (t,x) \in [0,T]\times[0,R]
\end{equation}
for some suitable constant $b_0 >0$. 
\end{itemize}
Note that $b$ is $\mu_t$-integrable because it is lower semicontinuous and positive, and hence the integral at the left hand side of~\eqref{eq:3.2} is well-defined. 
Note furthermore that by combining~\eqref{eq:3.3} with~\eqref{eq:3.2} we get 
\begin{equation}
\label{e:conscons}
    \mathrm{ess \, sup}_{t \in ]0, T[} \| \mu_t \|_{\mathcal M (]0, R[)} \leq \frac{1}{b_0}. 
\end{equation}
We now focus on the problem
\begin{equation}
  \label{eq:3.1}
  \text{maximize $J(\mu)$, defined as in~\eqref{eq:3.1bis}
    among $\mu \in L^\infty 
    \big(]0, T[; \mathcal M_+ (]0, R[) \big)$ satisfying~\eqref{eq:3.2}}.      
\end{equation}
The following result establishes the existence of an optimal strategy $\mu^\ast$
for problem~\eqref{eq:3.1}. 
\begin{proposition}
\label{th:main2} Assume \emph{{\bf (H.1)}-{\bf (H.5)}}. Then 
the optimization problem \eqref{eq:3.1} admits an optimal solution $\mu^\ast$.
\end{proposition}
The proof of Proposition~\ref{th:main2} is based on the same argument as the proof of    
Theorem 1 in~\cite{BS1} (see also~\cite[Theorem 4.1]{CG}) and so we omit it. 
The main result of the present section establishes Euler-type conditions for optimality. 
\begin{theorem}
  \label{th:Euler}
  Assume  \emph{{\bf (H.1)}-{\bf (H.5)}}. Let $\mu^*\in L^\infty\big(]0,T[;{\mathcal M}_+(]0,R[)\big)$ be an optimal strategy for the problem \eqref{eq:3.1} and $\fhi_\ast$ be the corresponding solution of 
~\eqref{eq:2}.  
For every $\nu\in L^\infty\big(]0,T[;{\mathcal M}_+(]0,R[)\big)$ there is 
$\fhi_1 \in L^2 (]0, T[; H^1 (]0, R[))$ that is a weak solution, in the sense of Definition~\ref{def:sol}, 
of the initial-boundary value problem
\begin{equation}
\label{eq:Euler}
\begin{cases}
\pt \fhi_1=\pxx \fhi_1-\fhi_1 \mu^\ast-\nu \fhi_\ast+ \partial_\fhi g (t, x, \fhi_*) \fhi_1  &
 \text{in $]0, T[ \times ]0, R[$}\\
\px \fhi_1(t,0)=\px \fhi_1(t,R)=0 & t \in ]0, T[  \\
\fhi_1(0,x)=0 & x\in ]0,R[ \\
\end{cases}
\end{equation}
and satisfies 
\begin{equation}
\label{eq:Euler-int}
\int_0^T  \! \!  \int_0^R \fhi_1 (t,x) d\mu^*_t(x)dt+
\int_0^T  \! \! \int_0^R  \fhi_\ast(t,x)
d\nu_t(x)dt-\Psi'
\left(\int_0^T \! \! \int_0^R c(t,x) d\mu^*_t(x)dt\right)
\int_0^T  \! \! \int_0^Rc(t,x) d\nu_t(x)dt \le 0.
\end{equation}
The function $g: \, [0, T] \times [0, R] \times \R 
\to \R$  at the first line of~\eqref{eq:Euler}, is defined by setting 
\begin{equation}
\label{e:cosaeg}
        g (t, x, \fhi) : = f(t, x, \fhi)\fhi. 
\end{equation}
\end{theorem}
\begin{proof}
  Fix $\nu$ as in the statement of the theorem and, for every
  real number $\eps>0$,
  consider the quantity
  \begin{equation*}
    J(\mu^*+\eps \nu)=
    \int_0^T\int_0^R \vfi_\eps (t, x) d(\mu_t^*+\eps\nu_t)(x)dt
    -\Psi\left(\int_0^T\int_0^R c(t,x)d(\mu_t^*+\eps\nu_t)(x)dt\right).
  \end{equation*}
  In the previous expression, we term $\vfi_\eps$ the solution of the initial-boundary problem~\eqref{eq:2} in the case  $\mu = \mu^*+\eps\nu$. The heuristic idea to establish~\eqref{eq:Euler} is to differentiate both $J(\mu^*+\eps \nu)$ and $\fhi_\eps$ with respect to the variable $\eps$. The rigorous proof is organized into the following steps. \\
\firststep
\step{we construct an approximate derivative of $\fhi_*$}
For every $\eps > 0$, define
\begin{equation}
\label{e:cosaepsieps}
\psi_\eps=\frac{\vfi_\eps-\fhi_\ast}{\eps}.
\end{equation}
Note that $\psi_\eps$ is the weak solution, in the sense of Definition~\ref{def:sol}, of the initial-boundary value problem
\begin{equation}
\label{eq:psieps}
\begin{cases}
\pt\psi_\eps=\pxx\psi_\eps-\psi_\eps\mu^\ast-\vfi_\eps\nu+G_\eps(t, x),& \text{in $]0, T[ \times ]0, R[$}\\
\px \psi_\eps (t,0)=\px \psi_\eps (t,R)=0 & t \in ]0, T[  \\
\psi(0,x)=0 & x\in ]0,R[ \\\end{cases}
\end{equation}
provided that 
\begin{equation}
\label{e:cosaeGgrande}
G_\eps(t,x)= \frac{f (t, x, \fhi_\eps ) \fhi_\eps -
f (t, x, \fhi_\ast ) \fhi_\ast  }{\eps}.
\end{equation}
In other words, for every test function $v \in C^\infty_c (]-\infty, T[ \times \R)$ we have 
\begin{equation}
\label{e:distrform6}
    \begin{split}
\int_0^T \! \! \int_0^R 
&
 \left(\pt v \,  \psi_\eps -\px v\, \px \psi_\eps
  \right) dx dt-\int_0^T\! \! \int_0^R v \,\psi_\eps d \mu^*_t (x) dt 
- \int_0^T \! \! \int_0^R v  \, \vfi_\eps  \, d \nu_t(x) dt+
\int_0^T \! \! \int_0^R v \, G_\eps   dx dt =0. \\
\end{split}
\end{equation}
Note that 
\begin{equation}
\label{e:boundsuGgrande}
\begin{split}
  | G_\eps(t,x)| & 
  \stackrel{\eqref{e:cosaeGgrande}}{\leq}
  \frac{1}{\varepsilon}
  |f (t, x, \fhi_\eps ) \fhi_\eps -
f (t, x, \fhi_\eps ) \fhi_\ast | + 
  \frac{1}{\varepsilon} 
   |f (t, x, \fhi_\eps ) \fhi_\ast -
f (t, x, \fhi_\ast ) \fhi_\ast | \\
   & \stackrel{\eqref{eq:2.2}}{\leq}
   | f (t, x, \fhi_\eps )|
   \left| \frac{\fhi_\eps - \fhi_\ast}{\eps}  \right|
   + |\fhi_\ast| \alpha_1 
   \left| \frac{\fhi_\eps - \fhi_\ast}{\eps}  \right| \\
   &  \stackrel{\eqref{e:boundsueffe},~\eqref{e:cosaepsieps}}{\leq}
   \big[ \alpha_1 |\fhi_\eps| + F \big]
   |\psi_\eps| + \alpha_1 |\fhi_\ast| |\psi_\eps| \\
   & \stackrel{\eqref{eq:th2.1}}{\leq}
   \big[ \alpha_1M + F \big]
     |\psi_\eps|+ \alpha_1 M  |\psi_\eps| = C(\alpha_1, M, F)  |\psi_\eps|.  
\end{split}
\end{equation}
We now provide a formal argument, which can be made rigorous by following the same argument as in {\sc Step 2} of~\S~\ref{sss:uni}. 
We multiply the equation at the first line of~\eqref{eq:psieps} times $\psi_\eps$ and integrate by parts to get 
\begin{equation}
\label{e:chainmadrid}
\begin{split}
    \frac{d}{dt} 
    \int_0^R \frac{\psi_\eps^2}{2} dx + &
    \int_0^R (\px \psi_\eps)^2 dx +
    \int_0^R \psi^2_\eps d \mu_t^* (x) = 
    - \int_0^R \psi_\eps \fhi_\eps  d \nu_t (x)
    + \int_0^R G_\eps \psi_\eps  dx \\
    &
    \stackrel{\eqref{e:boundsuGgrande}}{\leq} 
    \| \fhi_\eps (t, \cdot) \|_{L^\infty (]0, R[)}
    \| \psi_\eps (t, \cdot) \|_{L^\infty (]0, R[)}
    \| \nu_t \|_{\mathcal M (]0, R[)} 
    + C(\alpha_1, M, F) 
    \int_0^R \! \! \psi^2_\eps  dx \\
    &
    \stackrel{\eqref{eq:th2.1}}{\leq} 
    M  \| \psi_\eps (t, \cdot) \|_{L^\infty (]0, R[)}
    \| \nu_t \|_{\mathcal M (]0, R[)} 
    + C(\alpha_1, M, F) 
    \int_0^R \! \! \psi^2_\eps  dx \\
    &
     \stackrel{\eqref{e:immersione}}{\leq} 
    C(M, R)  \| \psi_\eps (t, \cdot) \|_{H^1 (]0, R[)}
    \| \nu_t \|_{\mathcal M (]0, R[)} 
    + C(\alpha_1, M, F) 
    \int_0^R \! \! \psi^2_\eps  dx.
    \end{split}
\end{equation} 
Next, we fix a constant $k >0$
(to be determined in the following), we use the Young Inequality and we infer that 
\begin{equation}
\label{e:youngmadrid}
\begin{split}
    \| \psi_\eps (t, \cdot) \|_{H^1 (]0, R[)} & 
    \| \nu_t \|_{\mathcal M (]0, R[)}  \leq 
    \frac{k}{2} \| \psi_\eps (t, \cdot) \|^2_{H^1 (]0, R[)}+
    \frac{1}{2k} \| \nu_t \|^2_{\mathcal M (]0, R[)} \\
    & =  \frac{k}{2} \| \psi_\eps (t, \cdot) \|^2_{L^2 (]0, R[)}
    +  \frac{k}{2} \| \px \psi_\eps (t, \cdot) \|^2_{L^2 (]0, R[)}
    +  \frac{1}{2k} \| \nu_t \|^2_{\mathcal M (]0, R[)} .
    \end{split}
\end{equation} 
Next, we choose the constant $k$ in such a way 
that $C(M, R) k =1$ and by combining~\eqref{e:chainmadrid} and~\eqref{e:youngmadrid} we arrive at 
\begin{equation}
\label{e:chain2madrid}
\begin{split}
    \frac{d}{dt} 
    \int_0^R \frac{\psi_\eps^2}{2} dx + &
    \frac{1}{2}
    \int_0^R (\px \psi_\eps)^2 dx +
    \int_0^R \psi^2_\eps d \mu^\ast_t (x) \\
    & \leq 
    C(M,R)
    \| \nu_t \|^2_{\mathcal M (]0, R[)} 
    + C(\alpha_1, M, F, R) 
    \int_0^R \! \! \psi^2_\eps  dx.
    \end{split}
\end{equation}
Owing to the Gronwall Lemma and to the fact that $\psi_\eps (t=0) \equiv0$, the above inequality implies that 
\begin{equation}
\label{e:conclusione1}
      \int_0^R \psi^2_\eps (t, x)dx 
      \leq 
      C(\alpha_1, M, F, T, R) 
      \| \nu_t \|^2_{\mathcal M (]0, R[)} 
      \quad \text{for  every $t \in ]0, T[$.}
\end{equation}
By integrating~\eqref{e:chain2madrid} over time 
and using~\eqref{e:conclusione1} we eventually arrive at 
\begin{equation}
  \label{e:conclusione2}
  \int_0^R \psi^2_\eps (t, x)dx +
  \int_0^T \int_0^R (\px \psi_\eps)^2 dx ds +
  \int_0^T \int_0^R \psi^2_\eps d \mu_s^* (x) ds 
  \leq 
  C(\alpha_1, M, F, T, R) 
  \| \nu_t \|^2_{\mathcal M (]0, R[)} 
\end{equation}
for  every $t \in ]0, T[$.

\step{we establish compactness of the family $\{ \psi_\eps \}$}
We rely on the Aubin-Lions Lemma. We recall that 
$$
    H^1 (]0, R[) \hookrightarrow C^0 ([0, R]) \hookrightarrow H^\ast (]0, R[), 
$$ 
and the first inclusion is compact and the second is continuous. Owing to the Aubin-Lions Lemma, to establish the compactness of the family $\{ \psi_\eps \}$ in $L^2 (]0, T[; C^0 ([0, R])$ it suffices to show that
\begin{eqnarray}
     \text{$\{ \psi_\eps \}$ is uniformly bounded in $L^2 (]0, T[; H^1(]0, R[))$}, 
     \label{e:unibd12m}
     \\
     \text{$\{ \pt \psi_\eps \}$ is uniformly bounded in $L^2 (]0, T[; H^\ast (]0, R[))$}.
     \label{e:unibd22m}  
\end{eqnarray} 
To establish~\eqref{e:unibd12m} we use~\eqref{e:conclusione2}. To establish~\eqref{e:unibd22m}, we use the equation at the first line 
of~\eqref{eq:psieps}, we argue as in {\sc Step 1}  of the proof of Lemma~\ref{l:compactness} and eventually infer that
\begin{equation*}
  \begin{split}
    \| \pt \psi_\eps \|_{L^2 (]0, T[; H^\ast (]0, R[))} & \leq \|
    \psi_\eps \|_{L^2 (]0, T[; H^1 (]0, R[))} \Big[ K + C(R)
    \esssup_{t \in ]0, T[}
    \| \mu_t^* \|_{\mathcal M(]0, R[)}
    \\
    & \quad + C(\alpha_1, M, F) \Big] + C(M, R, T) \esssup_{t \in ]0,
      T[} \| \nu_t \|_{\mathcal M(]0, R[)}.
  \end{split}
\end{equation*}
Owing to the Aubin-Lions Lemma, we infer that the family $\{ \psi_\eps \}$
is compact in $L^2 (]0, T[; C^0 ([0, R]))$ and hence that there is a sequence $\psi_{\eps_k}$ and a function $\fhi_1 \in L^2 (]0, T[; C^0 (]0, R[))$ such that
\begin{eqnarray}
  \psi_{\eps_k}  \to 
  \fhi_1 \quad 
  \text{in $L^2 (]0, T[; C^0 ([0, R]))$} ,
  \label{e:aubin1} 
  \\
  \psi_{\eps_k} (t, \cdot) \to 
  \fhi_1 (t, \cdot) \quad 
  \text{uniformly in $C^0 ([0, R])$, 
    for a.e. $t \in ]0, T[$}. 
  \label{e:aubin2} 
\end{eqnarray}
By using the bound~\eqref{e:chain2madrid} we infer that the family $\{ \px \psi_\eps \}$ is weakly compact in $L^2 (]0, T[ \times ]0, R[)$ and hence that
\begin{equation}
\label{e:aubin3}
\px \psi_{\eps_k}  \weak 
       \px \fhi_1 \quad 
       \text{weakly in $L^2 (]0, T[\times ]0, R[)$}.
\end{equation}
\step{we show that the accumulation point $\fhi_1$ is a weak solution of~\eqref{eq:Euler}} We argue by passing to the limit in each of the terms in~\eqref{e:distrform6}. First, we point out that by using~\eqref{e:aubin1} and~\eqref{e:aubin3} we get 
\begin{equation}
\label{e:alimite1}
     \int_0^T \int_0^R( \pt v  \, \psi_{\eps_k}
     - \px v \, \px  \psi_{\eps_k} )dtdx\to 
     \int_0^T \int_0^R (\pt v \, \fhi_1
     - \px v \,  \px  \fhi_1  )dtdx.
\end{equation}
Next, we point out that 
\begin{equation}
  \label{e:alimite2}
  \begin{split}
    \left| \int_0^T \int_0^R v \right. & \left. \psi_{\eps_k} d
      \mu^\ast_t (x) dt - \int_0^T \int_0^R v \fhi_1 d \mu^\ast_t (x)
      dt \right|
    \\
    & \leq \int_0^T \| v \|_{C^0 ([0, R] \times [0,
      T])} \| \psi_{\eps_k} (t, \cdot) - \fhi_1(t, \cdot) \|_{C^0 ([0,
      R])}
    \| \mu_t^* \|_{\mathcal M (]0, R[)} dt 
    \\
    & \stackrel{\eqref{e:elledueelleuno}}{\leq} \esssup_{0\le t\le T}
    \| \mu_t \|_{\mathcal M (]0, R[)} \| v \|_{C^0 ([0, R] \times [0,
      T])} C(T) \| \psi_{\eps_k}  - \fhi_1
    \|_{L^2 (]0, T[; C^0 ([0, R]))}
    \\
    & \stackrel{\eqref{e:aubin1}}{\longrightarrow} 0.
  \end{split}
\end{equation}
We now want to show that 
\begin{equation}
\label{e:alimite3}
     \left|  \int_0^T \! \! \int_0^R v  \, \vfi_{\eps_k}  \, d \nu_t(x) dt -
      \int_0^T \! \! \int_0^R v  \, \fhi_\ast  \, d \nu_t(x) dt \right|
     \to 0.
\end{equation}
To this end, we recall that $\fhi_\eps$ is the weak solution of the initial-boundary problem~\eqref{eq:2} in the case when $\mu = \mu^*+\eps\nu$ and we apply the stability estimate~\eqref{e:stability2} with $\fhi= \fhi_\ast$, $\hat \fhi = \fhi_\eps$, $\mu = \mu^\ast$, $\hat \mu = \mu^\ast + \eps \nu$. By using~\eqref{e:immersione}, we infer that 
\begin{equation}
\label{e:alimite5}
    \fhi_{\eps_k} \to \fhi_\ast \quad 
    \text{in $L^2 (]0, T[; C^0 ([0, R]))$}
\end{equation}
and by arguing as in~\eqref{e:alimite2} we arrive 
at~\eqref{e:alimite3}. 
We are left with the last term in~\eqref{e:distrform6}. 
To handle it, we recall that $0 \leq \fhi_\ast, \fhi_\eps \leq M$ owing to~\eqref{eq:th2.1}, we use the Taylor formula with Lagrange remainder and we infer
that 
\begin{equation*}
  g(t, x, \fhi_\eps) =
  g(t, x, \fhi_\ast) + \partial_\fhi g (t, x, \fhi_\ast) [\fhi_\eps - \fhi_\ast]+
  \frac{1}{2} \partial^2_{\fhi \fhi} 
  g (t, x, y) [\fhi_\eps - \fhi_\ast]^2, \quad
  y \in [0, M]. 
\end{equation*}
Owing to~\eqref{e:cosaeg},~\eqref{e:cosaepsieps} and~\eqref{e:cosaeGgrande}, this implies that
\begin{equation}
  \label{e:alimite4}
  | G_\eps(t, x) -  \partial_\fhi g (t, x, \fhi_\ast) \psi_\eps | 
  \leq K \max_{(t, x, y) \in [0, T] \times [0, R] \times [0, M]}
  |\partial^2_{\fhi \fhi} 
  g (t, x, y)| |\psi_\eps| |\fhi_\eps -\fhi_\ast|     
\end{equation}
and hence that 
\begin{equation}
  \label{e:alimite42}
  \begin{split}
    \int_0^T \! \! \int_0^R | G_{\eps_k}(t, x) & - \partial_\fhi g (t,
    x, \fhi_\ast) \fhi_1 | dtdx \leq \int_0^T \! \! \int_0^R | G_{\eps_k}
    (t, x) - \partial_\fhi g (t, x, \fhi_\ast)
    \psi_{\eps_k} | dtdx 
    \\
    & \qquad + \int_0^T \! \! \int_0^R |
    \partial_\fhi g (t, x, \fhi_\ast)
    \psi_{\eps_k} - \partial_\fhi g (t, x, \fhi_\ast)  \fhi_1 | dtdx 
    \\
    & \stackrel{\eqref{e:alimite4}}{\leq} K \max_{(t, x, y) \in [0, T]
      \times [0, R] \times [0, M]} |\partial^2_{\fhi \fhi} g (t, x,
    y)| \int_0^T \! \! \int_0^R |\psi_{\eps_k}| |\fhi_{\eps_k}
    -\fhi_\ast| dtdx
    \\
    & \qquad + \max_{(t, x, y) \in [0, T] \times [0, R] \times [0, M]}
    |\partial_{ \fhi}
    g (t, x, y)| \int_0^T \! \! \int_0^R |\psi_{\eps_k} - \fhi_1| dtdx
    \\
    & \stackrel{\text{H\"older}}{\leq} K \max_{(t, x, y) \in [0, T]
      \times [0, R] \times [0, M]} |\partial^2_{\fhi \fhi} g (t, x,
    y)| \| \psi_{\eps_k} \|_{L^2} \| \fhi_{\eps_k} -
    \fhi_\ast \|_{L^2}   
    \\
    & \qquad + C(T,R) \max_{(t, x, y) \in [0, T] \times [0, R] \times
      [0, M]} |\partial_{ \fhi}
    g (t, x, y)| \| \psi_{\eps_k} - \fhi_1\|_{L^2}
    \\
    & \stackrel{\eqref{e:unibd12m},\eqref{e:aubin1},\eqref{e:alimite5}}
    {\longrightarrow}
    0.
  \end{split}
\end{equation}
This implies that 
$$
    \left|
    \int_0^T \! \! \int_0^R v \, G_{\eps_k}  dx dt  -
    \int_0^T \! \! \int_0^R v \, \partial_\fhi g (t, x, \fhi_\ast) 
    \fhi_1 dx dt 
    \right| \longrightarrow 0
$$
and show that $\fhi_1$ is a weak solution of~\eqref{eq:Euler}. \\
\step{we establish~\eqref{eq:Euler-int}} We recall that $\eps_k >0$,
that $\mu^\ast$ is an optimal strategy and $\mu^\ast + \eps_k \nu$
is a competitor provided that  $\nu\ge0$ is sufficiently small.
We conclude that 
\begin{equation*}
\frac{J(\mu^\ast+\eps_k \nu)-J(\mu^\ast)}{\eps_k}\le 0.
\end{equation*}
By using the explicit expression of $J(\mu^\ast+\eps_k \nu)$ we 
 infer 
\begin{align*}
\frac{J(\mu^\ast+\eps_k \nu)-J(\mu^\ast)}{\eps_k}=&\int_0^T\int_0^R \psi_{\eps_k} d\mu_t^*(x)dt+\int_0^T\int_0^R \vfi_{\eps_k} d\nu_t(x)dt\\
&-\frac{\Psi\left(\displaystyle{
\int_0^T\int_0^R c(t,x)d(\mu_t^*+\eps_k \nu_t)(x)dt} \right)-\Psi\left(
\displaystyle{
\int_0^T\int_0^R c(t,x)d\mu_t^*(x)dt} \right)}{\eps_k}.
\end{align*}
By combining the two above expressions we get 
\begin{align*}
0\ge &\int_0^T\! \! \int_0^R \psi_{\eps_k} d\mu_t^*(x)dt+\int_0^T\! \! \int_0^R \vfi_{\eps_k} d\nu_t(x)dt\\
&\qquad \qquad\qquad -\frac{\Psi\left(\displaystyle{
\int_0^T\! \! \int_0^R c(t,x)d(\mu_t^*+\eps_k \nu_t)(x)dt} \right)-\Psi\left(
\displaystyle{
\int_0^T\! \! \int_0^R c(t,x)d\mu_t^*(x)dt} \right)}{\eps_k} \\
& \stackrel{\eqref{e:alimite2},\eqref{e:alimite3}}{\longrightarrow}
\int_0^T\! \! \int_0^R \fhi_1 d\mu_t^\ast (x)dt+\int_0^T\! \!
\int_0^R \fhi_\ast d\nu_t(x)dt \\
& \qquad \qquad\qquad
- \Psi' \left(\int_0^T\! \! \int_0^R c(t,x)d\mu_t^\ast(x)dt\right)\int_0^T\! \! \int_0^R c(t,x)d\nu_t(x)dt,
\end{align*}
that is \eqref{eq:Euler-int}. This concludes the proof of Theorem~\ref{th:Euler}. 
\end{proof}

\section{Uniqueness of optimal solutions}
\label{sec:6}
In this section we discuss the uniqueness of optimal solutions of the optimization problem~\eqref{eq:3.1}. 
We refine hypothesis {\bf (H.1)} by introducing the following condition. 
\begin{itemize}
\item[({\bf H.6})]
\label{h:accasei}
The function $h$ in~\eqref{eq:2.3} is a constant. In other words, 
\begin{equation}
\label{eq:2.3bis}
    f(t, x, \fhi ) >0 \quad \text{if and only if $\fhi <h$}.
\end{equation}
\end{itemize}
We recall that the  function $\partial_\fhi f$ is continuous and by recalling~\eqref{eq:2.2} we define the constant
 \begin{equation}
 \label{e:accastar}
- h_*: = \max_{(t, x) \in [0, T ] \times [0, R]}
\p_\vfi f(t, x, h) < 0. 
\end{equation}
Also, in the following we term $\alpha_2$ a Lipschitz constant for 
$\p_\fhi f$, namely
\begin{equation}
\label{e:alpha2}
   |\p_\vfi f (t,x,\fhi_1) -\p_\vfi f(t,x,\fhi_2)|
   \leq \alpha_2 |\fhi_1-\fhi_2| \qquad
   \text{for every $(t, x) \in[0, T] \times [0, R]$, 
   $\fhi_1, \fhi_2 \in [0, M]$.}
\end{equation}
The main result of the present section states that, if the initial datum $\fhi_0$ is sufficiently small, then the solution of the optimization problem~\eqref{eq:3.1} is unique within a class of measures with sufficiently small total variation. 
\begin{theorem}
  \label{th:main3}
  Assume {\bf (H.1)}-{\bf (H.6)}. Then there is a constant $0<\delta<1$,
  which only depends on the constants $\alpha_1$, $\alpha_2$, $M$,
  $h$, $T$, $R$ and $h_\ast$ such that, if
  \begin{equation}
    \label{eq:initass}\mathrm{ess \; sup}_{t \in ]0, T[} 
    \|  \mu_t \|_{\mathcal M (]0, R[)} +
    \| \fhi_0 - h \|_{H^1 (]0, R[)} \le \delta \leq 1 \quad 
  \end{equation}
  then the solution $ \mu$ of the constrained optimization problem \eqref{eq:3.1} is unique. More precisely, assume 
  that $\tilde \mu$ and $\bar \mu$ are two points of maximum such that 
  \begin{equation}
    \label{e:muastast}
    \mathrm{ess \; sup}_{t \in ]0, T[}   
    \| \tilde \mu_t \|_{\mathcal M(]0, R[)}, 
    \mathrm{ess \; sup}_{t \in ]0, T[}   
    \| \bar \mu_t \|_{\mathcal M(]0, R[)}  \leq \delta. 
  \end{equation}
  Then $\tilde \mu = \bar \mu$.
\end{theorem}
Note that, owing to~\eqref{e:conscons}, to achieve the bound 
$$
    \sup_{t \in ]0, T[} 
     \| \tilde \mu_t \|_{\mathcal M (]0, R[)} \leq \delta
$$
it suffices to have $1/b_0 \leq \delta$, where $b_0$ is the bound from below on the function $b$, see~\eqref{eq:3.3}, and is therefore a datum of the problem. 
\begin{proof}[Proof of Theorem~\ref{th:main3}]
We fix two points of maximum $\tilde \mu$ and $\bar \mu$ satisfying~\eqref{e:muastast} and we argue by contradiction: we assume that $\tilde \mu \neq \bar \mu$. We set 
\begin{equation}
\label{e:nueps}
   \varepsilon:= \mathrm{ess \; sup}_{t \in ]0, T[} \|
      \bar \mu_t - \tilde \mu_t \|_{\mathcal M(]0, R[)} \leq 2 \delta .
\end{equation}
We define the (signed) measure $\nu_t \in L^\infty (]0, T[; \mathcal M(]0, R[))$ by setting 
 \begin{equation}
 \label{e:nu}
     \nu_t : =
      \frac{\bar \mu_t - \tilde \mu_t}{\varepsilon}  .
  \end{equation}
We define the map $j: [0, \varepsilon] \to \R$ by setting 
\begin{equation}
\label{e:geipiccolo}
    j(\zeta) : = J(\tilde \mu + \zeta \nu) 
\end{equation}
and we point out that by construction $j$ attains its maximum at both $\zeta =0$ and $\zeta= \varepsilon$.   

Next, we use Lemma~\ref{l:concave} below and we conclude that the map $j$ is continuous and concave on $[0, \varepsilon]$. This contradicts the fact that $j$ attains its maximum at $\zeta=\varepsilon$ and hence concludes the proof of Theorem~\ref{th:main3}. 
\end{proof}
\begin{lemma}
\label{l:concave}
Let $j$ be the same function as in~\eqref{e:geipiccolo}. Then $j$ is continuous on $[0, \varepsilon]$ and twice differentiable in $]0, \varepsilon[$. Also, if the constant $\delta$ in~\eqref{e:muastast} is sufficiently small, then 
\begin{equation}
   j''(\zeta) <0 \quad \text{for every $\zeta \in ]0, \varepsilon[$.} 
\end{equation} 
\end{lemma}
\begin{proof}
First, we point out that the map $j$ is continuous: this can be seen by using the stability estimate~\eqref{e:stability2}. Also, by arguing as in the proof of Theorem~\ref{th:Euler} we infer that $j$ is twice differentiable and that, 
for every $\zeta \in ]0, \varepsilon[$, we have 
\begin{equation}
    \begin{split}
    \label{e:jtwice}
             j''(\zeta) =          
            &
\int_0^T \! \! \int_0^R \vfi_2(t,x) d  \mu^\ast_t(x)dt+
2\int_0^T \! \!  \int_0^R \vfi_1(t,x) d\nu_t(x)dt\\
&-\Psi''\left(\int_0^T\int_0^R c(t,x)d\mu^\ast_t(x)dt\right)\left(\int_0^T \! \!  \int_0^R c(t,x)d\nu_t(x)dt\right)^2\\
\end{split}
    \end{equation}
provided that the measure $\mu^\ast$ is given by 
\begin{equation}
\label{e:cosaemuast}
    \mu^\ast : = \tilde \mu+ \zeta \nu 
\end{equation}
and the functions $\fhi_1$ and $\fhi_2$ are defined as follows. The function $\fhi_1$ 
is the solution of
\begin{equation}
\label{e:defhiuno}
\begin{cases}
\pt \fhi_1=\pxx \fhi_1-\fhi_1 \mu^\ast_t- \fhi_\ast \nu_t+ \partial_\fhi g ( \fhi_\ast) \fhi_1  &
 \text{in $]0, T[ \times ]0, R[$}\\
\px \fhi_1(t,0)=\px \fhi_1(t,R)=0 & t \in ]0, T[  \\
\fhi_1(0,x)=0 & x\in ]0,R[ \\
\end{cases}
\end{equation} and $\fhi_2$ is the weak solution of the initial-boundary value problem   
 \begin{equation}
\label{eq:Euler2}
\begin{cases}
\pt\fhi_2=\pxx \fhi_2
-\fhi_2 \mu^\ast_t - 2 \nu_t \fhi_1 + \partial_\fhi g ( \fhi_\ast) \fhi_2 + \partial^2_{\fhi \fhi} g (\fhi_\ast) \fhi_1^2 
 &
 \text{in $]0, T[ \times ]0, R[$}\\
\px \fhi_2(t,0)=\px \fhi_2(t,R)=0 & t \in ]0, T[  \\
\fhi_2(0,x)=0 & x\in ]0,R[. \\
\end{cases}
\end{equation}
In the above equation, $\fhi_\ast$ is the weak solution of
\begin{equation}
  \label{e:fhiast}
  \begin{cases}
    \pt\vfi_\ast = \pxx\vfi_\ast - \vfi_\ast \mu^\ast_t + g(t,x,\vfi_\ast),
    & \text{in $]0, T[ \times ]0, R[$},
    \\
      \px \fhi_\ast(t,0)=\px\vfi_\ast(t,R) = 0, & t \in \, ]0, T[,
     \\
    \vfi_\ast(0,x)=\vfi_0(x),
    & x\in{]0,R[}.
  \end{cases}
\end{equation}
By using the convexity of the function $\Psi$, we infer from~\eqref{e:jtwice} that 
\begin{equation*}
  j''(\zeta) \leq 
  \int_0^T\int_0^R \vfi_2(t,x) d  \mu_t^*(x)dt+
  2\int_0^T\int_0^R \vfi_1(t,x) d\nu_t(x)dt
\end{equation*}
and hence the proof of Lemma~\ref{l:concave} is an easy consequence of Lemma~\ref{l:uni1} below.
\end{proof}
\begin{lemma}
\label{l:uni1} Assume $\mathrm{{\bf (H.1)}}$-$\mathrm{{\bf (H.6)}}$.
Let $\mu^\ast$, $\fhi_1$ and $\fhi_2$ be the same as in~\eqref{e:cosaemuast},~\eqref{e:defhiuno} and~\eqref{eq:Euler2}, respectively.  Then  
 there is a sufficiently small constant $\delta$ such that, if~\eqref{eq:initass} holds, then 
\begin{equation}
\label{e:keypoint}
         2\int_0^T \! \! \int_0^R \vfi_1(t,x) d\nu_t(x)dt +
          \int_0^T \! \! \int_0^R \vfi_2(t,x) d\mu^\ast_t(x)dt< 0. 
\end{equation}
\end{lemma}
The proof of Lemma~\ref{l:uni1} is rather long and technical and it is established in~\S~\ref{s:pkey}. 
\section{Proof of Lemma~\ref{l:uni1}}
\label{s:pkey}
To establish the proof of Lemma~\ref{l:uni1} we proceed as follows: in \S\ref{ss:fhiast}, \S\ref{ss:fhiuno} and \S\ref{ss:fhidue} we provide a formal argument which is completely justified only in the case when all the functions are sufficiently regular. In \S\ref{ss:tuttorigoroso} we first conclude this formal argument by (formally) establishing~\eqref{e:keypoint}  and then we explain how the formal argument can be made rigorous by relying on an approximation argument. To simplify notation, in the formal argument given in \S\ref{ss:fhiast}, \S\ref{ss:fhiuno} and \S\ref{ss:fhidue} we write $\sup_{t \in ]0, T[} \| \mu_t \|_{\mathcal M(]0, R[)}$
to denote $\mathrm{ess \; sup}_{t \in ]0, T[} \| \mu_t \|_{\mathcal M(]0, R[)}$. 
\subsection{Proof of Lemma~\ref{l:uni1}: estimates on $\fhi_\ast$}
\label{ss:fhiast}
We first control the distance of the function $\fhi_\ast$ from the constant $h$. 
\begin{lemma}
\label{l:fiast}
Let $\fhi_\ast$ be the weak solution of the initial-boundary 
value problem~\eqref{eq:2}, then \begin{equation}
\label{e:inftyast}
     \| \fhi_\ast (t, \cdot) - h \|_{L^\infty (]0, R[)} \leq 
     C(\alpha_1, M, T, R ) \left( \| \fhi_0 - h \|_{H^1 (]0, R[)}  + 
     \sup_{t \in ]0, T[} \| \mu_t \|_{\mathcal M (]0, R[)} \right) 
     \quad \forall\text{ $t \in ]0, T[$.}
\end{equation}
In particular, there is a threshold $\delta$, which only depends on $\alpha_1, M, T, R$ and $h$ such that, if~\eqref{eq:initass}  holds, then 
\begin{equation}
\label{e:grande}
    \fhi_\ast (t, x) \ge \frac{h}{2} >0 \quad 
    \text{for a.e. $(t, x) \in \, ]0, T[ \times ]0, R[$}. 
\end{equation}
\end{lemma}
\begin{proof}
Owing to the Duhamel Representation Formula (see the Appendix) we have 
\begin{equation}
\label{e:duha0}
\begin{split}
    \fhi_\ast (t, x)  = &
    \int_0^R D(t, x, y) \fhi_0(y)  dy -
    \int_0^t \! \! \int_0^R D(t-s, x, y) \fhi_\ast (s, y) d \mu^\ast_s (y) ds \\
    & +
    \int_0^t \! \! \int_0^R D(t-s, x, y) f \big (s, y, \fhi_\ast (s, y)\big)   
    \fhi_\ast(s, y) dy ds. \\
\end{split}
\end{equation}
We recall that $f(t, x, h) \equiv 0$, which implies that $\fhi \equiv h$ is a weak solution of the initial-boundary value problem~\eqref{eq:2} in the case when $\mu^\ast \equiv 0$. We deduce the following representation formula:
\begin{equation}
\label{e:duha}
\begin{split}
    \fhi_\ast (t, x) - h = &
    \int_0^R D(t, x, y) \big[ \fhi_0(y) -h \big] dy -
    \int_0^t \! \! \int_0^R D(t-s, x, y) \fhi_\ast (s, y) d \mu^\ast_s (y) ds \\
    & +
    \int_0^t \! \! \int_0^R D(t-s, x, y) f \big (s, y, \fhi_\ast (s, y)\big)   
    \fhi_\ast(s, y) dy ds. \\
\end{split}
\end{equation}
In the previous expressions, $D$ is the same kernel as in~\eqref{e:Di}.  
Since $f(\cdot, \cdot, h) \equiv 0$, then  
\begin{equation}
\label{e:effestar}
\begin{split}
    | f (t, x, \fhi_\ast )   
    \fhi_\ast   | & =
    | f  (t, x, \fhi_\ast )   
    \fhi_\ast - 
    f  (t, x, h )   
    \fhi_\ast     | \stackrel{\eqref{eq:2.2}}{\leq}
     \alpha_1 | \fhi_\ast || \fhi_\ast -h| \\
     & \stackrel{\eqref{eq:th2.1}}{\leq}  
     C(\alpha_1, M) | \fhi_\ast (t, x) -h| \quad
    \text{for every $(t, x) \in [0, T]\times [0, R]$}. 
    \end{split} 
\end{equation}
By plugging the above inequality into~\eqref{e:duha} we infer 
\begin{equation*}
  \begin{split}
    \| \fhi_\ast (t, \cdot) - h \|_{L^\infty (]0, R[)} &
    \stackrel{\eqref{eq:th2.1}}{\leq} \| D(t, x, \cdot)\|_{L^1 (]0,
      R[)} \| \fhi_0 -h \|_{L^\infty (]0, R[)}
    \\ 
    & \quad + M \sup_t
    \| \mu^\ast \|_{\mathcal M (]0, R[ )}
    \int_0^t \| D(t-s, x, \cdot) \|_{L^\infty (]0, R[)}ds 
    \\
    & \quad + \int_0^t \! \!  \| D(t-s, x, \cdot) \|_{L^1 (]0, R[)} \|
    f(s, \cdot, \fhi_\ast) \fhi_\ast (s, \cdot) \|_{L^\infty (]0, R[)}
    ds
    \\
    & \stackrel{\eqref{e:Duno},\eqref{e:Dlinfty}}{\leq} K \| \fhi_0 -h
    \|_{L^\infty (]0, R[)} + C(M, R) \sup_t \| \mu_t^\ast \|_{\mathcal M
      (]0, R[ )}
    \int_0^t \frac{1}{\sqrt{t-s}} ds 
    \\
    & \quad + K \int_0^t \! \!
    \| f(s, \cdot, \fhi_\ast) \fhi_\ast (s, \cdot) \|_{L^\infty (]0, R[)} ds
    \\
    & \stackrel{\eqref{e:immersione},\eqref{e:effestar}}{\leq} C(R) \|
    \fhi_0 -h \|_{H^1 (]0, R[)} +
    C(M, T, R) \sup_t \| \mu_t^\ast \|_{\mathcal M (]0, R[ )}  
    \\
    & \quad + C(\alpha_1, M) \int_0^t \! \!  \| \fhi_\ast(s, \cdot) -h
    \|_{L^\infty (]0, R[)} ds.
  \end{split}
\end{equation*}
Owing to the Gronwall Lemma, the above inequality 
implies~\eqref{e:inftyast}. 
\end{proof}
Next, we control the derivative $\px \fhi_\ast$ in $L^\infty (]0, T[, L^2 (]0, R[))$.
\begin{lemma}
  \label{l:linftyelledue}
  Under the same assumptions as in the statement of Lemma~\ref{l:fiast},
  we have 
\begin{equation}
\label{e:linftyelledue}
    \| \partial_x \fhi_\ast (t, \cdot) \|_{L^2 (]0, R[)}
    \leq C(\alpha_1, M, T, R) 
    \left[ 
     \| \fhi_0 - h \|_{H^1 (]0, R[)}+
     \sup_{t \in ]0, T[} 
     \| \mu_t^\ast \|_{\mathcal M (]0, R[)}  \right] \quad 
     \forall \text{ $t \in ]0, T[$.}
\end{equation}
\end{lemma}
\begin{proof}
By differentiating the representation formula~\eqref{e:duha0} with respect to $x$ and using~\eqref{e:dualita} we get 
\begin{equation}
  \label{e:duha1}
  \begin{split}
    \partial_x \fhi_\ast (t, x) = & \underbrace{\int_0^R \tilde D(t,
      x, y) \fhi'_0(y) dy}_{I_1(t, x)} - \underbrace{\int_0^t \! \!
      \int_0^R \partial_x D(t-s, x, y) \fhi_\ast (s, y) d \mu^\ast_s
      (y) ds}_{I_2 (t, x)}
    \\
    & + \underbrace{\int_0^t \! \! \int_0^R \partial_x D(t-s, x, y) f \big
      (s, y, \fhi_\ast \big)
      \fhi_\ast(s, y) dy ds.}_{I_3(t, x)} 
  \end{split}
\end{equation}
We control the first term by arguing as follows: 
\begin{equation}
\label{e:izero}
\begin{split}
  \| I_1 (t, \cdot) \|_{L^2 (]0, R[)} & = 
  \left[ \int_0^R  \left(
  \int_0^R \tilde D(t, x, y) \fhi'_0(y)  dy \right)^2 dx \right]^{1/2}
  \\
  & \stackrel{\text{H{\"o}lder}}{\leq} 
  \left[
  \int_0^R \left(  \int_0^R 
  \tilde D(t, x, y) (\fhi'_0)^2 (y)  dy  
  \right) 
  \left(
  \int_0^R
  \tilde D(t, x, y)   dy
  \right) dx \right]^{1/2} \\ &
  \stackrel{\eqref{e:tDuno}}{\leq} 
  K 
  \left[
  \int_0^R   \int_0^R 
  \tilde D(t, x, y) (\fhi'_0)^2 (y)  dy  dx 
  \right]^{1/2} \\
  &  \leq 
  K 
  \left(
  \int_0^R (\fhi'_0)^2 (y)   \int_0^R 
  \tilde D(t, x, y)   dx \, dy 
  \right)^{1/2}
  \\ & 
  \stackrel{\eqref{e:tDunox}}{\leq} K   \left[
  \int_0^R (\fhi'_0)^2 (y)  
  \right]^{1/2} 
   \leq K \| \fhi_0 -h \|_{H^1 (]0, R[).} \phantom{\int} \\
\end{split}
\end{equation}
To establish the last inequality, we have used the fact that $h$ is a constant. 
Owing to the Bochner Theorem~\cite[p.473]{Salsa}, we have 
\begin{equation}
\label{e:iuno}
\begin{split}
    \| I_2 (t, \cdot) \|_{L^2 (]0, R[)} & 
    \leq \int_0^t \left\|  
     \int_0^R \partial_x 
    D(t-s, \cdot, y) \fhi_\ast (s, y) d \mu^\ast_s (y)  \right\|_{L^2(]0, R[)} ds \\
    & 
     \stackrel{\eqref{eq:th2.1}}{\leq} 
    M 
    \int_0^t 
    \left[
    \int_0^R \left(
    \int_0^R  |\partial_x D|
    (t-s, x, y) d \mu^\ast_s (y) \right)^2 dx \right]^{1/2}   ds 
    \\
    & \stackrel{\text{H{\"o}lder}}{\leq}  
    M \int_0^t
    \left[
    \int_0^R
    \left(
    \int_0^R 
    (\partial_x 
    D)^2
     (t-s, x, y)  d 
     \mu^\ast_s (y) \right)
     \| \mu^\ast_s \|_{\mathcal M (]0, R[)} dx \right]^{1/2}   ds  \\ 
    & \leq 
    M 
    \left( \sup_{s \in ]0, T[}   \| \mu^\ast_s \|_{\mathcal M (]0, R[)} 
    \right)^{1/2}
    \int_0^t
    \left(  
    \int_0^R
    \int_0^R 
    (\partial_x 
    D)^2
     (t-s, x, y) dx   d 
     \mu^\ast_s (y) 
      \right)^{1/2}   ds  \\ 
     & 
     \stackrel{\eqref{e:Dxduex}}{\leq} 
     C(M) \left( \sup_{s \in ]0, T[}
      \| \mu^\ast_s \|_{\mathcal M (]0, R[)}  \right)^{1/2}
     \int_0^t \left[ \int_0^R \frac{1}{(t-s)^{3/2}} d \mu^\ast_s(y)  \right]^{1/2} ds \\ &
     \leq C(M) \sup_{s \in ]0, T[}
      \| \mu^\ast_s \|_{\mathcal M (]0, R[)}  
     \int_0^t  \frac{1}{(t-s)^{3/4}}  ds \\ &
     \leq  C(M, T) 
    \sup_{s \in ]0, T[}  \| \mu^\ast_s \|_{\mathcal M (]0, R[)}.   \phantom{\int}
\end{split}
\end{equation}
By using again the Bochner Theorem~\cite[p.473]{Salsa} and arguing as before 
we get  
\begin{equation}
  \label{e:idueelleue}
  \begin{split}
    \| I_3 (t, \cdot) \|_{L^2 (]0, R[)} & \leq \int_0^t \left\|
      \int_0^R \partial_x D(t-s, \cdot, y) f \big (s, y, \fhi_\ast
      \big)
      \fhi_\ast(s, y) dy \right\|_{L^2(]0, R[)} ds
    \\
    & \stackrel{\eqref{e:maxprin},\eqref{e:effestar}}{\leq} C(M,
    \alpha_1) \int_0^t \left\| \int_0^R | \partial_x D(t-s, \cdot, y)|
      | \fhi_\ast -h |(s, y) dy \right\|_{L^2(]0, R[)} ds
    \\
    & \leq
    C(M, \alpha_1) \int_0^t \int_0^R \| \partial_x D (t-s, \cdot, y)
    \|_{L^2 (]0, R[)} | \fhi_\ast -h |(s, y) dy ds
    \\
    & \stackrel{\eqref{e:Dxduex}}{\leq} C(\alpha_1, M) 
    \int_0^t \| \fhi_\ast(s,
    \cdot) -h \|_{L^\infty (]0, R[)}\frac{1}{(t-s)^{3/4}}
    ds
    \\
    & \stackrel{\eqref{e:inftyast}}{\leq} C(\alpha_1, M, T, R) \Big[
    \sup_{t \in ]0, T[} \| \mu_t^\ast \|_{\mathcal M (]0, R[)} + \|
    \fhi_0 - h \|_{H^1(]0, R[)} \Big] \int_0^t \frac{1}{(t-s)^{3/4}}
    ds
    \\
    & \leq C(\alpha_1, M, T, R) \Big[ \sup_{t \in ]0, T[} \|
    \mu_t^\ast \|_{\mathcal M (]0, R[)} + \| \fhi_0 - h \|_{H^1 (]0,
      R[)} \Big].
  \end{split}
\end{equation}
By plugging~\eqref{e:izero},~\eqref{e:iuno} and~\eqref{e:idueelleue} into~\eqref{e:duha1} we establish~\eqref{e:linftyelledue}. 
\end{proof}
\subsection{Proof of Lemma~\ref{l:uni1}: estimates on $\fhi_1$}
\label{ss:fhiuno}
We can now control the first term in~\eqref{e:keypoint}. 
\begin{lemma}
\label{le:6.3}
  Under the same assumptions as in the statement of Lemma~\ref{l:fiast}, we have 
\begin{equation}
\label{e:integral}
          \int_0^T  \! \! \int_0^R \vfi_1(t,x) d\nu_t(x)dt 
\leq  \| \fhi_1 \|^2_{L^2 (  ]0, T[; H^1 (]0, R[))} 
\Big[ 
- C(M, h_\ast) +  C(\alpha_1, \alpha_2, M, h, T, R)  \delta
\Big].
\end{equation}
\end{lemma}
\begin{proof}
We argue according to the following steps. \\
\firststep
\step{we find a more convenient expression for the left hand side
  of~\eqref{e:integral}}
Owing to~\eqref{e:grande}, the function $\fhi_\ast$ is bounded away from $0$. We recall that by definition $g(\cdot, \cdot, \fhi) = f(\cdot, \cdot, \fhi)\fhi$, 
which implies that 
$$
    \partial_\fhi g (\cdot, \cdot, \fhi_\ast) =
    f(\cdot, \cdot, \fhi_\ast) + \partial_\fhi f(\cdot, \cdot, \fhi_\ast )\fhi_\ast.  
$$
We divide  the equation at the first line of~\eqref{eq:Euler} times $\fhi_\ast$ and we use  the above expression for $\partial_\fhi g$. We eventually get
\begin{equation*}
\label{eq:4.2.3}
\begin{split}
\nu_t=&-\frac{\pt\vfi_1}{\fhi_\ast}+\frac{\pxx\vfi_1}{\fhi_\ast}-\frac{\vfi_1}{\fhi_\ast}\mu^\ast_t+
f(\cdot,\cdot,\fhi_\ast)\frac{\vfi_1}{\fhi_\ast}+\p_\vfi f(t,x,\fhi_\ast)\vfi_1. \\
\end{split}
\end{equation*} 
By using the above expression for $\nu_t$, we can formally rewrite
the left hand side of~\eqref{e:integral} as  
\begin{equation}
  \label{e:integrake}
  \begin{split}
    \int_0^T \! \int_0^R\! \vfi_1(t,x) d\nu_t(x)dt = & - \int_0^T \!
    \int_0^R \!\frac{\vfi_1 }{\fhi_\ast} \pt\vfi_1 dx dt + \int_0^T \!
    \int_0^R \!\frac{\vfi_1 }{\fhi_\ast} \pxx\vfi_1 dx dt - \int_0^T \!
    \int_0^R \!
    \frac{\vfi_1^2}{\fhi_\ast} d \mu^\ast (x ) dt
    \\
    & + \int_0^T \!  \int_0^R \left(
      f(t,x,\fhi_\ast)\frac{\vfi_1^2}{\fhi_\ast} +\p_\vfi
      f(t,x,\fhi_\ast)\vfi_1^2
    \right) dx dt.
  \end{split}
\end{equation}
\step{we separately control each of the terms in~\eqref{e:integrake}}
Owing to Cauchy condition $\fhi_1 (0, \cdot) \equiv 0$ in~\eqref{eq:Euler} and to the inequality~\eqref{e:grande} we have 
\begin{equation}
\label{e:inte1}
\begin{split}
     - \int_0^T \!
\int_0^R   \frac{\vfi_1}{\fhi_\ast}\pt\vfi_1 \, dx dt = & 
- \frac{1}{2} \int_0^T \!
\int_0^R    \frac{\pt \fhi_1^2}{\fhi_\ast} \, dx dt =
- \frac{1}{2} \int_0^R  \frac{\fhi_1^2}{\fhi_\ast} (T, x) dx -
 \frac{1}{2} \int_0^T \!
\int_0^R  \frac{\fhi_1^2 }{\fhi_\ast^2} \, \pt \fhi_\ast \, dx dt \\
& \stackrel{\eqref{e:grande}}{\leq} - \frac{1}{2} \int_0^T \!
\int_0^R  \frac{\fhi_1^2 }{\fhi_\ast^2} \, \pt \fhi_\ast \, dx dt  \\
& \stackrel{\eqref{e:fhiast}}{=}
- \frac{1}{2} \int_0^T \!
\int_0^R  \frac{\fhi_1^2 }{\fhi_\ast^2}
\, \partial_{xx} \fhi_\ast \, dx dt +
\frac{1}{2} \int_0^T \!
\int_0^R  \frac{\fhi_1^2 }{\fhi_\ast}
 d \mu^\ast (x) dt \\ & \quad -
 \frac{1}{2}
  \int_0^T \!
\int_0^R  \frac{\fhi_1^2 }{\fhi_\ast}
\, f(t, x, \fhi_\ast)  \, dx dt.
\end{split}
\end{equation}
By using the fact that $\fhi_\ast$ satisfies homogeneous Neumann boundary conditions~\eqref{eq:2} we obtain
\begin{equation}
\label{e:inte3}
\begin{split}
     - \frac{1}{2} \int_0^T \!
\int_0^R  \frac{\vfi_1^2}{\fhi_\ast^2}\pxx\fhi_\ast \, dx dt = & 
 \int_0^T \!
\int_0^R   \px \fhi_\ast \, \px \fhi_1  \frac{\fhi_1}{\fhi^2_\ast} \, dx dt 
-  \int_0^T \! \int_0^R ( \px \fhi_\ast)^2  \frac{\fhi_1^2}{\fhi^3_\ast}  \, dx dt . \\ 
\end{split}
\end{equation}
By using the fact that $\fhi_1$ satisfies homogeneous Neumann boundary conditions~\eqref{eq:Euler} we get  \begin{equation}
\label{e:inte2}
\begin{split}
      \int_0^T \!
\int_0^R   \frac{\vfi_1}{\fhi_\ast}\pxx \vfi_1 \, dx dt = & 
-  \int_0^T \!
\int_0^R    \frac{( \px \fhi_1)^2}{\fhi_\ast} \, dx dt +
\int_0^T \!
\int_0^R    \frac{\fhi_1}{\fhi^2_\ast} \, \px \fhi_1 \, \px \fhi_\ast \, dx dt.   
\end{split}
\end{equation}
By  plugging~\eqref{e:inte1},~\eqref{e:inte2} and~\eqref{e:inte3} into~\eqref{e:integrake} we arrive at 
\begin{equation}
  \label{e:integrake2}
  \begin{split}
    \int_0^T \! \int_0^R \vfi_1(t,x) d\nu_t(x)dt &
    \stackrel{\eqref{e:integrake},~\eqref{e:inte1}}{\leq} -
    \frac{1}{2} \int_0^T \!  \int_0^R \frac{\fhi_1^2 }{\fhi_\ast^2}
    \, \partial_{xx} \fhi_\ast \, dx dt + \frac{1}{2} \int_0^T \!
    \int_0^R \frac{\fhi_1^2 }{\fhi_\ast} d \mu^\ast (x) dt 
    \\
    & \quad
    - \frac{1}{2} \int_0^T \!  \int_0^R \frac{\fhi_1^2 }{\fhi_\ast} \,
    f(t, x, \fhi_\ast) \, dx dt + \int_0^T \!
    \int_0^R \frac{\vfi_1}{\fhi_\ast} \pxx\vfi_1 dx dt 
    \\
    & - \int_0^T \!  \int_0^R
    \frac{\vfi_1^2}{\fhi_\ast} d \mu^\ast (x ) dt
    \\
    & + \int_0^T \!  \int_0^R \left(
      f(t,x,\fhi_\ast)\frac{\vfi_1^2}{\fhi_\ast} +\p_\vfi
      f(t,x,\fhi_\ast)\vfi_1^2
    \right) dx dt 
    \\
    & \stackrel{\eqref{e:inte3}}{\leq} \int_0^T \!  \int_0^R \px
    \fhi_\ast \, \px \fhi_1 \frac{\fhi_1}{\fhi^2_\ast} \, dx dt -
    \int_0^T \! \int_0^R ( \px \fhi_\ast)^2
    \frac{\fhi_1^2}{\fhi^3_\ast} \, dx dt 
    \\
    & - \frac{1}{2} \int_0^T
    \!  \int_0^R \frac{\fhi_1^2 }{\fhi_\ast} d \mu^\ast (x) dt 
    \\
    &
    \quad + \frac{1}{2} \int_0^T \!  \int_0^R \frac{\fhi_1^2
    }{\fhi_\ast} \, f(t, x, \fhi_\ast) \, dx dt + \int_0^T \!
    \int_0^R \frac{\vfi_1 }{\fhi_\ast} \pxx\vfi_1 dx dt  
    \\
    & + \int_0^T \!  \int_0^R \p_\vfi f(t,x,\fhi_\ast)\vfi_1^2
    dx dt
    \\
    & \stackrel{\eqref{e:inte2}}{\leq} 2 \int_0^T \!  \int_0^R \px
    \fhi_\ast \, \px \fhi_1 \frac{\fhi_1}{\fhi^2_\ast} \, dx dt -
    \int_0^T \! \int_0^R ( \px \fhi_\ast)^2
    \frac{\fhi_1^2}{\fhi^3_\ast} \, dx dt
    \\
    & - \frac{1}{2} \int_0^T
    \!  \int_0^R \frac{\fhi_1^2 }{\fhi_\ast} d \mu^\ast (x) dt 
    \\
    &
    \quad + \frac{1}{2} \int_0^T \!  \int_0^R \frac{\fhi_1^2
    }{\fhi_\ast} \, f(t, x, \fhi_\ast) \, dx dt - \int_0^T \!
    \int_0^R \frac{( \px \vfi_1)^2 }{\fhi_\ast}dx dt 
    \\
    & + \int_0^T \!  \int_0^R \p_\vfi f(t,x,\fhi_\ast)\vfi_1^2 dx dt .
  \end{split}
\end{equation}
Next, we use the Young Inequality and we get 
$$ 
2 \int_0^T \!
\int_0^R    \frac{\fhi_1}{\fhi^2_\ast} \, \px \fhi_1 \, \px \fhi_\ast \, dx dt 
\leq \frac{1}{a} \int_0^T \!
\int_0^R   \frac{( \px \fhi_1)^2 }{\fhi_\ast} \, dx dt +
a \int_0^T \!
\int_0^R   \frac{( \px \fhi_\ast)^2 }{\fhi^3_\ast} \, \fhi_1^2 \, dx dt 
$$
for some $a>0$ to be determined in the following. 
We plug the above inequality into~\eqref{e:integrake2} and we get 
\begin{equation}
  \label{e:integrake3}
  \begin{split}
    \int_0^T \! \! \int_0^R \vfi_1(t,x) d\nu_t(x)dt \leq & \left[
      \frac{1}{a} -1 \right] \int_0^T \! \! \int_0^R \frac{( \px
      \fhi_1)^2}{\fhi_\ast} \, dx dt + [a-1] \int_0^T \!  \int_0^R
    \frac{( \px \fhi_\ast)^2 }{\fhi^3_\ast} \, \fhi_1^2 \, dx dt
    \\
    &-
    \frac{1}{2} \int_0^T \!  \int_0^R \frac{\fhi_1^2 }{\fhi_\ast} d
    \mu^\ast (x) dt + \frac{1}{2} \int_0^T \!  \int_0^R \frac{\fhi_1^2
    }{\fhi_\ast}
    \, f(t, x, \fhi_\ast)  \, dx dt  
    \\
    & + \int_0^T \!  \int_0^R \p_\vfi f(t, x,\fhi_\ast)\vfi_1^2
    dx dt
    \\
    & \stackrel{\eqref{e:grande}}{\leq} \left[ \frac{1}{a} -1 \right]
    \int_0^T \! \! \int_0^R \frac{( \px \fhi_1)^2}{\fhi_\ast} \, dx dt
    + [a-1] \int_0^T \!  \int_0^R \frac{( \px \fhi_\ast)^2
    }{\fhi^3_\ast} \, \fhi_1^2 \, dx dt
    \\
    & + \frac{1}{2} \int_0^T \!
    \int_0^R \frac{\fhi_1^2 }{\fhi_\ast} \, f(t, x, \fhi_\ast)
    \, dx dt + \int_0^T \!  \int_0^R \p_\vfi
    f(t,x,\fhi_\ast)\vfi_1^2
    dx dt .
  \end{split}
\end{equation}
Owing to~\eqref{e:accastar}, we get
\begin{equation}
\label{e:inte4}
\begin{split}
 \int_0^T \! \!
\int_0^R 
\p_\vfi f(t,x,\fhi_\ast)\vfi_1^2 \, dx dt  & \leq  \int_0^T \! \!
\int_0^R
\big[ - h_\ast + [\p_\vfi f (t,x,\fhi_\ast) -\p_\vfi f(t,x,h) ]
\big]
 \vfi_1^2 (t, x)
 \, dx dt \\
 & \stackrel{\eqref{e:alpha2}}{\leq} \int_0^T \! \!
\int_0^R
\big[ - h_\ast + \alpha_2 |\fhi_\ast - h|
\big]
 \vfi_1^2
 \, dx  dt. \\
\end{split}
\end{equation}
Next, we choose $a=2$ and by recalling~\eqref{eq:th2.1} we infer 
$$
    \left[ \frac{1}{a} -1 \right] \int_0^T  \! \! \int_0^R    \frac{( \px \fhi_1)^2}{\fhi_\ast} \, dx dt
    = -\frac{1}{2}
   \int_0^T  \! \! \int_0^R    \frac{( \px \fhi_1)^2}{\fhi_\ast} \, dx dt 
        \stackrel{\eqref{eq:th2.1}}{\leq}
        - C(M)
          \int_0^T  \! \! \int_0^R 
        ( \px \fhi_1)^2 dx dt. 
$$
By combining the above equation with~\eqref{e:inte4} and plugging the result in~\eqref{e:integrake3} we arrive at 
\begin{equation}
  \label{e:coercivita}
  \begin{split}
    \int_0^T \! \! \int_0^R \!\!\vfi_1(t,x) d\nu_t(x)dt \leq & - C(M)\!
    \int_0^T \! \! \int_0^R \!\!( \px \fhi_1)^2 \, dx dt - h_\ast\! \int_0^T
    \! \! \int_0^R \!\!\fhi_1^2 \, dx dt +\! \int_0^T \!  \int_0^R \!\!\frac{(
      \px \fhi_\ast)^2 }{\fhi^3_\ast} \, \fhi_1^2 \, dx dt
    \\
    &+
    \frac{1}{2} \int_0^T \!  \int_0^R \frac{\fhi_1^2 }{\fhi_\ast} \,
    f(t, x, \fhi_\ast) \, dx dt + \alpha_2 \int_0^T \!  \int_0^R
    |\vfi_\ast - h |\vfi_1^2 dx dt 
    \\
    &
    \stackrel{\eqref{eq:2.2}}{\leq} - C(M, h_\ast) \| \fhi_1 \|^2_{L^2
      ( ]0, T[; H^1 (]0, R[))} + \int_0^T \!  \int_0^R \frac{( \px
      \fhi_\ast)^2 }{\fhi^3_\ast} \, \fhi_1^2 \, dx dt 
    \\
    &+ C(
    \alpha_1) \int_0^T \!  \int_0^R \frac{\fhi_1^2}{\fhi_\ast}
    |\fhi_\ast - h| \, dx dt + \alpha_2 \int_0^T \!  \int_0^R
    |\vfi_\ast - h |\vfi_1^2
    dx dt 
    \\
    & \stackrel{\eqref{e:grande}}{\leq} - C(M, h_\ast) \| \fhi_1
    \|^2_{L^2 ( ]0, T[; H^1 (]0, R[))} + \underbrace{C(h) \int_0^T \!
      \int_0^R ( \px \fhi_\ast)^2 \, \fhi_1^2 \, dx dt}_{A_1}
    \\
    &+
    \underbrace{C(\alpha_1, \alpha_2, h ) \int_0^T \!  \int_0^R
      \fhi_1^2 |\fhi_\ast - h| \, dx dt}_{A_2}.
  \end{split}
\end{equation}
To control $A_1$ we use~\eqref{e:linftyelledue} and argue as follows:
\begin{equation}
\label{e:auno}
\begin{split}
 A_1 & = C(h) \int_0^T \!
\int_0^R   ( \px \fhi_\ast)^2  \, \fhi_1^2 \, dx dt
\leq C(h)  \int_0^T \! \| \fhi_1 (t, \cdot) \|^2_{L^\infty (]0, R[)}
\int_0^R   ( \px \fhi_\ast)^2  dx dt \\ &
\stackrel{\eqref{e:immersione}}{\leq}
C(h, R)  \int_0^T \! \| \fhi_1 (t, \cdot) \|^2_{H^1 (]0, R[)}
\int_0^R   ( \px \fhi_\ast)^2  dx dt \\ &
 \stackrel{\eqref{e:linftyelledue}}{\leq}
 C(M,h, T, R) \| \fhi_1 \|^2_{L^2 (  ]0, T[; H^1 (]0, R[))}
 \left[ 
     \| \fhi_0 - h \|_{H^1 (]0, R[)}+
     \sup_{t \in ]0, T[} 
     \| \mu_t^\ast \|_{\mathcal M (]0, R[)}  \right]  \\
     & \stackrel{\eqref{eq:initass}}{\leq}
     C(M,h, T, R) \| \fhi_1 \|^2_{L^2 (  ]0, T[; H^1 (]0, R[))} \delta. 
\end{split}
\end{equation}
To control $A_2$ we use~\eqref{e:inftyast} and we get 
\begin{equation}
\label{e:adue}
\begin{split}
  A_2 & =        C(\alpha_1, \alpha_2, h  )
  \int_0^T \!
\int_0^R  \fhi_1^2 |\fhi_\ast - h|  \, dx dt
  \\
  & \stackrel{\eqref{e:inftyast}}{\leq} 
  C(\alpha_1, \alpha_2, M, h, T, R) 
  \| \fhi_1 \|^2_{L^2 (  ]0, T[; H^1 (]0, R[))}
  \left[ 
     \| \fhi_0 - h \|_{H^1 (]0, R[)}+
     \sup_{t \in ]0, T[} 
     \| \mu_t^\ast \|_{\mathcal M (]0, R[)}  \right] \\
     & \stackrel{\eqref{eq:initass}}{\leq}
   C(\alpha_1, \alpha_2, M, h, T, R) 
  \| \fhi_1 \|^2_{L^2 (  ]0, T[; H^1 (]0, R[))} \delta.     
\end{split}
\end{equation}
By plugging~\eqref{e:auno} and~\eqref{e:adue} into~\eqref{e:coercivita} we eventually arrive at 
$$
   \int_0^T  \! \! \int_0^R \vfi_1(t,x) d\nu_t(x)dt 
\leq  \| \fhi_1 \|^2_{L^2 (  ]0, T[; H^1 (]0, R[))} 
\Big[ 
- C(M, h_\ast) +  C(\alpha_1, \alpha_2, M, h, T, R)  \delta
\Big],
$$ 
that is~\eqref{e:integral}. \end{proof}
\subsection{Proof of Lemma~\ref{l:uni1}: estimates on $\fhi_2$}
\label{ss:fhidue}
We now establish a control on the $L^\infty$ norm of $\fhi_2$.
\begin{lemma}
\label{l:intfhi2}
Under the same assumptions as in the statement of Lemma~\ref{l:fiast}, we have 
\begin{equation}
  \label{e:intfhi2}
  \int_0^T \| \fhi_2(t, \cdot) \|_{L^\infty(]0, R[)} dt \leq
  C (\alpha_1, \alpha_2, M, h, T, R, F)
  \| \fhi_1 \|^2_{L^2(]0, T[ ; H^1(]0, R[ ))}.
\end{equation}
\end{lemma}
\begin{proof}
First, we recall that, since $g (t, x, \fhi) = f(t, x, \fhi) \fhi$, then 
$$
    \p_\fhi g = \p_\fhi f \cdot \fhi + f, \qquad 
    \p^2_{\fhi \fhi} g =  \p^2_{\fhi \fhi} f \cdot \fhi+ 2 \p_\fhi f. 
$$
This implies  
\begin{equation}
\label{e:bounduniformi1}
\begin{split}
    \|  \p_\fhi g (\cdot, \cdot, \fhi_\ast) \|_{L^\infty (]0, T[ \times ]0, R[)}
    & \stackrel{\eqref{eq:2.2},\eqref{eq:th2.1}}{\leq}
    \alpha_1 M + \| f (\cdot, \cdot, \fhi_\ast) - 
    f (\cdot, \cdot, h) \|_{L^\infty (]0, T[ \times ]0, R[)} \\
    & 
    \stackrel{\eqref{eq:initass},\eqref{e:inftyast}}{\leq}
    C(\alpha_1, M) + C(\alpha_1, M, T, R) \delta \\
    &
    \stackrel{\delta \leq 1}{\leq} C(\alpha_1, M, T, R)
     \end{split} 
\end{equation}
and, by combining~\eqref{eq:2.2},~\eqref{eq:th2.1} and~\eqref{e:alpha2},  
\begin{equation}
\label{e:bounduniformi2}
 \|  \p^2_{\fhi \fhi} g (\cdot, \cdot, \fhi_\ast) \|_{L^\infty (]0, T[ \times ]0, R[)}
 \leq C(\alpha_1, \alpha_2, M). 
\end{equation}
 By applying the Duhamel Representation Formula (see the Appendix) to the linear equation~\eqref{eq:Euler2} we arrive at  
\begin{equation*}
\begin{split}   
    \fhi_2 (t, x)  = & - 
    \int_0^t  \! \! \int_0^R \! \! \! D(t-s, x, y) \fhi_2 (s, y) d \mu^\ast_s (y) ds -
    2 \int_0^t  \! \!  \int_0^R \! \! \! D(t-s, x, y) \fhi_1 (s, y) d \nu_s (y) ds \\
    &  + \int_0^t  \! \! \int_0^R \! \! \! D(t-s, x, y)  \p_\fhi g (s, y, \fhi_\ast) \fhi_2 (s, y) dy ds +
     \int_0^t  \! \! 
     \int_0^R \! \! \! D(t-s, x, y)  \p^2_{\fhi \fhi} g (s, y, \fhi_\ast) \fhi^2_1 (s, y) dy ds.
     \\ 
\end{split}
\end{equation*}
The above representation formula implies 
\begin{equation}
  \label{e:duhamelfhidue}
  \begin{split}
    | \fhi_2 (t, x) |
    \stackrel{\eqref{e:bounduniformi1},\eqref{e:bounduniformi2}}{\leq}&
    \int_0^t \| D(t-s, x, \cdot) \|_{L^\infty (]0, R[)} \| \fhi_2 (s,
    \cdot) \|_{L^\infty (]0, R[)} \| \mu^\ast_s \|_{\mathcal M(]0,R[
      )}
    ds \\
    & + I(t) + C(\alpha_1, M, T, R) \int_0^t \| D(t-s, x, \cdot)
    \|_{L^1 (]0, R[)}
    \| \fhi_2 (s, \cdot) \|_{L^\infty (]0, R[)} ds \\
    & + C(\alpha_1, \alpha_2, M) \int_0^t \| D(t-s, x, \cdot) \|_{L^1
      (]0, R[)}
    \| \fhi_1 (s, \cdot) \|^2_{L^\infty (]0, R[)} ds, \\
  \end{split}
\end{equation}
where we have defined the function $I(t)$ by setting 
\begin{equation}
\label{e:iditi}
  I(t) : = 2  
  \sup_{x \in ]0, R[}\left|
  \int_0^t  \! \! 
   \int_0^R \! \! \! D(t-s, x, y) \fhi_1 (s, y) d \nu_s (y) ds
  \right|.
\end{equation}
From~\eqref{e:duhamelfhidue} we get
\begin{equation}
\label{e:prosegueduhamel}
\begin{split}
   \| \fhi_2 (t, \cdot) \|_{L^\infty (]0, R[)} &
   \stackrel{\eqref{e:Dlinfty},\eqref{e:Duno}}{\leq} 
   \sup_t \| \mu^\ast_t \|_{\mathcal M(]0, T[)}
   \int_0^t \frac{C(R)}{\sqrt{t-s}}  \| \fhi_2 (s, \cdot) \|_{L^\infty (]0, R[)} ds \\
   & \qquad + I(t) + C(\alpha_1, M, T, R) \int_0^t  
   \| \fhi_2 (s, \cdot) \|_{L^\infty (]0, R[)} ds   \\
   &  \qquad + C(\alpha_1, \alpha_2, M) \int_0^t  
   \| \fhi_1 (s, \cdot) \|^2_{L^\infty (]0, R[)} ds\\
   & \stackrel{\eqref{eq:initass},\eqref{e:immersione}}{\leq}
  \delta 
   \int_0^t \frac{C(R)}{\sqrt{t-s}}  \| \fhi_2 (s, \cdot) \|_{L^\infty (]0, R[)} ds +
   I(t)\\
   & \qquad 
   + C(\alpha_1, M, T, R) \int_0^t  
   \| \fhi_2 (s, \cdot) \|_{L^\infty (]0, R[)} ds    \\
   & \qquad 
   + C(\alpha_1, \alpha_2, M, R)  
   \| \fhi_1  \|^2_{L^2 (]0, T[; H^1 (]0, R[))} \\
   & \stackrel{\delta \leq 1}{\leq}
    C(\alpha_1, M, T, R)
   \int_0^t \left( \frac{1}{\sqrt{t-s}}+1 \right)
     \| \fhi_2 (s, \cdot) \|_{L^\infty (]0, R[)} ds \\
     & \qquad +
   I(t)   + C(\alpha_1, \alpha_2, M, R)  
   \| \fhi_1  \|^2_{L^2 (]0, T[; H^1 (]0, R[))} 
   . \phantom{\int}
\end{split}
\end{equation} 
Next, we point out that by definition~\eqref{e:iditi} the function $I$ is nondecreasing and we apply the Gronwall Lemma.
We get 
\begin{equation*}
\begin{split}
        \| \fhi_2 (t, \cdot) \|_{L^\infty (]0, R[} \stackrel{\eqref{e:prosegueduhamel}, \text{Gronwall}}{\leq} & 
         C(\alpha_1, \alpha_2, M, T, R)  \left[  I (t) + 
         \| \fhi_1  \|^2_{L^2 (]0, T[; H^1 (]0, R[))} \right]. 
\end{split}  
\end{equation*}
Finally, we conclude the proof of Lemma~\ref{l:intfhi2} by time integrating the above inequality and  
using Lemma~\ref{l:iditi} below. 
\end{proof}
To conclude the proof of Lemma~\ref{l:intfhi2} we are left to establish the following result. 
\begin{lemma}
\label{l:iditi}
Let $I$ be the same function as in~\eqref{e:iditi}, then we have 
\begin{equation}
\label{e:controli}
   \int_0^T  I (t) 
   dt  \leq C(\alpha_1, M, h, T, R, h, F)
   \| \fhi_1 \|^2_{L^2 (]0, T[; H^1 (]0, R[)}. 
\end{equation}
\end{lemma}
\begin{proof}
  First, we use the formula in the proof of Lemma~\ref{le:6.3}
  for $\nu_t$, we recall that $I$ is defined as in~\eqref{e:iditi} and we get 
  \begin{equation}
    \label{e:jeis}
    \begin{split}
      \int_0^T I(t) dt & \stackrel{\eqref{eq:4.2.3}}{=} \underbrace{ 2
        \int_0^T \sup_{x \in ]0, R[} \left| \int_0^t \int_0^R \! \! \!
          D(t-s, x, y) \frac{\pt \fhi_1}{\fhi_\ast} \fhi_1 (s, y) dy
          ds \right|
        dt}_{J_1}
      \\
      & + \underbrace{ 2\int_0^T \sup_{x \in ]0, R[} \left| \int_0^t
          \int_0^R \! \! \! D(t-s, x, y) \frac{\pxx \fhi_1}{\fhi_\ast}
          \fhi_1 (s, y) dy ds \right| dt}_{J_2}
      \\
      & + \underbrace{ 2 \int_0^T \sup_{x \in ]0, R[} \left| \int_0^t
          \int_0^R \! \! \! D(t-s, x, y) \frac{\fhi^2_1}{\fhi_\ast}
          (s,y ) d \mu^\ast_s (y) ds \right|
        dt}_{J_3} 
      \\
      & + \underbrace{ 2 \int_0^T \sup_{x \in ]0, R[} \left| \int_0^t
          \int_0^R \! \! \! D(t-s, x, y) \p_\fhi f (s, y, \fhi_\ast)
          \fhi^2_1 (s, y) dy ds \right| dt}_{J_4}
      \\
      & + \underbrace{ 2 \int_0^T \sup_{x \in ]0, R[} \left| \int_0^t
          \int_0^R \! \! \! D(t-s, x, y) \frac{ f (s, y, \fhi_\ast)}{\fhi_\ast}
          \fhi^2_1 (s, y) dy ds \right| dt}_{J_5}
    \end{split}
  \end{equation}
We now separately control the terms $J_1, \dots, J_5$. First, we consider the term $J_1$: we point out that
 \begin{equation}
 \label{e:completamento}
2\frac{\pt \fhi_1}{\fhi_\ast} \fhi_1 = 
\frac{\pt \fhi_1^2}{\fhi_\ast} = 
\pt \left[ \frac{\fhi_1^2}{\fhi_\ast}  \right]
+ \frac{\fhi_1^2}{\fhi_\ast^2} \pt \fhi_\ast. 
 \end{equation}
and we get 
\begin{equation}
\label{e:geiuno}
\begin{split}
  J_1  \stackrel{\eqref{e:completamento}}{\leq}&
    \int_0^T \sup_{x \in ]0, R[} 
    \left|
  \int_0^t    \int_0^R \! \! \! D(t-s, x, y) \partial_s 
  \left[ \frac{\fhi^2_1}{\fhi_\ast}  \right] (s, y) dy
   ds \right|
  dt \\ & 
  +  \underbrace{  \int_0^T 
  \sup_{x \in ]0, R[} \left|
  \int_0^t   \int_0^R \! \! \! D(t-s, x, y) \frac{\fhi_1^2}{\fhi_\ast^2} 
  \partial_s
   \fhi_\ast (s, y) dy   ds  \right|
  dt.}_{J_{12}}  
\end{split}
\end{equation}
By using the Integration by Parts Formula and the initial condition $\fhi_1 (t=0)\equiv 0$ we get
\begin{equation}
\label{e:geiuno1}
\begin{split}  
   \int_0^T  
  \sup_{x \in ]0, R[} &
    \left|
 \int_0^t \int_0^R \! \! \! D(t-s, x, y) \partial_s 
  \left[ \frac{\fhi^2_1}{\fhi_\ast}  \right] (s, y) dy
   ds  \right|
  dt \\ & 
  \leq 
   \int_0^T
    \sup_{x \in ]0, R[} 
    \left|
  \int_0^t  \int_0^R \! \! \! \partial_s D(t-s, x, y) 
  \frac{\fhi^2_1}{\fhi_\ast}   (s, y) dy 
   ds  \right|
  dt \\ & \qquad \qquad 
  +
   \int_0^T  \lim_{s\to t^-}
  \left| 
     \sup_{x \in ]0, R[} 
    \int_0^R  D(t-s, x, y) 
     \frac{\fhi^2_1}{\fhi_\ast}  (s, y) dy
   \right| dt \\
  & 
  \stackrel{\eqref{e:grande},\eqref{e:Duno}}{\leq}
  \underbrace{ \int_0^T
    \int_0^t  \sup_{x \in ]0, R[} 
    \left|
     \int_0^R \! \! \! \partial_{xx} D(t-s, x, y) 
  \frac{\fhi^2_1}{\fhi_\ast}   (s, y) dy
   \right|
   ds 
  dt}_{J_{11}} \\ &
   \qquad \qquad + C(h) \int_0^T \| \fhi_1 (t, \cdot) \|^2_{L^\infty (]0, R[)} dt \\
   & \stackrel{\eqref{e:immersione}}{\leq} J_{11} + 
   C(h, R) \int_0^T \| \fhi_1 (t, \cdot) \|^2_{H^1 (]0, R[)} dt
   \phantom{\int} \\
   & 
   \leq J_{11} + 
   C(h, R)  \| \fhi_1 \|^2_{L^2 (]0, T[ ; H^1 (]0, R[))}. 
    \phantom{\int} 
\end{split}
\end{equation}
By plugging the above formula into~\eqref{e:geiuno} we get 
\begin{equation}
\label{e:geiuno2}
  J_1 \leq J_{11} + J_{12} +  C(h, R)  \| \fhi_1 \|^2_{L^2 (]0, T[ ; H^1 (]0, R[))}.
\end{equation}
We now focus on the term $J_{11}$ defined in~\eqref{e:geiuno1}. First, we point out that 
\begin{equation}
\label{e:derivatax}
   \partial_y \left[ \frac{\fhi_1^2}{\fhi_\ast} \right] =
   2 \frac{\fhi_1}{\fhi_\ast} \partial_y \fhi_1 - 
    \frac{\fhi_1^2}{\fhi_\ast^2} \partial_y \fhi_\ast. 
\end{equation}
We infer that 
\begin{equation}
\label{e:geiunouno}
\begin{split}
 J_{11} \stackrel{\eqref{e:dualita}}{\leq} &
  \int_0^T   \int_0^t \sup_{x \in ]0, R[} 
    \left|
 \int_0^R \! \! \! \partial_{x} \tilde D(t-s, x, y) 
  \partial_y \left[ \frac{\fhi^2_1}{\fhi_\ast}  
  \right] (s, y) dy
  \right|  ds 
  dt  \\
  & \stackrel{\eqref{e:derivatax}}{\leq} 
   K \underbrace{\int_0^T   \int_0^t  
   \sup_{x \in ]0, R[} 
    \left|
    \int_0^R \! \! \! \partial_{x} \tilde D(t-s, x, y) 
   \frac{\fhi_1}{\fhi_\ast} \partial_y \fhi_1 
   (s, y) dy \right|
   ds 
   dt }_{J_{111}} \\  & \qquad + K \underbrace{
   \int_0^T   \int_0^t \sup_{x \in ]0, R[} 
    \left|
  \int_0^R \! \! \! \partial_{x} \tilde D(t-s, x, y) 
  \frac{\fhi_1^2}{\fhi_\ast^2} \partial_y \fhi_\ast
  (s, y) dy  \right|
   ds 
  dt.}_{J_{112}}
\end{split}
\end{equation}
We now control $J_{111}$:
\begin{equation}
\label{e:geiunounouno}
 \begin{split}
 J_{111} & \leq 
 \int_0^T  
  \int_0^t \sup_{x \in ]0, R[}  \int_0^R
  \left|   \partial_{x} \tilde D(t-s, x, y) 
   \frac{\fhi_1}{\fhi_\ast} \partial_y \fhi_1 
   (s, y)  \right|  dy
   ds 
  dt \\
  & \stackrel{\text{H\"older}}{\leq}
   \int_0^T  
  \int_0^t  \sup_{x \in ]0, R[}  
  \left\| \partial_{x} \tilde D(t-s, x, \cdot) 
  \right\|_{L^2 (]0, R[)}
  \left\|  \frac{\fhi_1}{\fhi_\ast} \partial_y \fhi_1 
   (s, \cdot) \right\|_{L^2 (]0, R[)} 
   ds 
  dt \\ &
  \stackrel{\eqref{e:grande},~\eqref{e:tDxdue}}{\leq}
  C(h) \int_0^T \int_0^t  \frac{1}{(t-s)^{3/4}}
  \| \fhi_1(s, \cdot)  \|_{L^\infty(]0, R[)}
  \| \partial_y \fhi_1 \|_{L^2 (]0, R[)} ds dt \\ &
  \stackrel{\text{Young}}{\leq}
  C(h) \int_0^T \int_0^t  \frac{1}{(t-s)^{3/4}}
  \Big[ \| \fhi_1(s, \cdot)  \|^2_{L^\infty(]0, R[)} + 
  \| \partial_y \fhi_1 \|^2_{L^2 (]0, R[)} \Big] ds dt  \\ & 
  \stackrel{\eqref{e:immersione}}{\leq}
   C(h, R) \int_0^T \int_0^t  \frac{1}{(t-s)^{3/4}}
   \| \fhi_1(s, \cdot)  \|^2_{H^1(]0, R[)} ds \, dt \\
  & =  C(h, R) 
  \int_0^T \| \fhi_1(s, \cdot)  \|^2_{H^1(]0, R[)} 
  \int_s^T  \frac{1}{(t-s)^{3/4}} dt \, ds  \\
  & \leq  C(h, T, R) 
  \| \fhi_1 \|^2_{L^2 (]0, T[; H^1 (]0, R[))}. \phantom{\int}
 \end{split}
\end{equation}
Next, we control $J_{112}$:
\begin{equation}
  \label{e:geiunounodue}
  \begin{split}
    J_{112} & \leq \int_0^T \int_0^t \sup_{x \in ]0, R[} \int_0^R
    \left| \partial_{x} \tilde D(t-s, x, y)
      \frac{\fhi_1^2}{\fhi_\ast^2} \partial_y \fhi_\ast (s, y) \right|
    dy ds  dt
    \\
    & \stackrel{\text{H\"older}}{\leq} \int_0^T \sup_{x \in ]0, R[}
    \int_0^t \left\| \partial_{x} \tilde D(t-s, x, \cdot)
    \right\|_{L^2(]0, R[)} \left\| \frac{\fhi_1^2}{\fhi_\ast^2} (s,
      \cdot) \right\|_{L^\infty(]0, R[)} \| \partial_y \fhi_\ast (s,
    \cdot) \|_{L^2 (]0, R[)} ds dt 
    \\
    &
    \stackrel{\eqref{eq:initass},\eqref{e:linftyelledue}}{\leq} C(\alpha_1,M,
    T, R) \delta \int_0^T \sup_{x \in ]0, R[} \int_0^t
    \left\| \partial_{x} \tilde D(t-s, x, \cdot) \right\|_{L^2(]0,
      R[)} \left\| \frac{\fhi_1^2}{\fhi_\ast^2} (s, \cdot)
    \right\|_{L^\infty(]0, R[)} dsdt
    \\
    &
    \stackrel{\eqref{e:grande},\eqref{e:tDxdue}}{\leq} C(\alpha_1, M, h, T, R)
    \delta \int_0^T \int_0^t \frac{1}{(t-s)^{3/4}} \| \fhi_1 (s,
    \cdot) \|_{L^\infty(]0, R[)}^2 ds dt
    \\
    &
    \stackrel{\eqref{e:immersione}}{\leq} C(\alpha_1, M, h, T, R) \delta
    \int_0^T \int_0^t \frac{1}{(t-s)^{3/4}} \| \fhi_1 (s, \cdot)
    \|_{H^1(]0, R[)}^2 ds
    dt  
    \\
    & \leq C(\alpha_1, M, h, T, R) \delta \| \fhi_1 \|^2_{L^2 (]0, T[; H^1 (]0,
      R[))}. \phantom{\int}
  \end{split}
\end{equation}
By combining~\eqref{e:geiunounouno} and~\eqref{e:geiunounodue} and recalling that $\delta \leq 1$ we get 
\begin{equation}
\label{e:geiunounofinal}
   J_{11} \le K \left(J_{111} + J_{112}\right)  \leq 
   C(\alpha_1, M, h, T, R)
      \| \fhi_1 \|^2_{L^2 (]0, T[; H^1 (]0, R[))}.
\end{equation}
We now control $J_{12}$: we recall that $J_{12}$ is defined as in~\eqref{e:geiuno} and that $\fhi_\ast$ satisfies~\eqref{e:fhiast}. We 
get 
\begin{equation}
\label{e:geiunodue}
\begin{split}
   J_{12} & \stackrel{\eqref{e:fhiast}}{\leq}
     \underbrace{2 \int_0^T 
  \int_0^t  \sup_{x \in ]0, R[} \left|  \int_0^R \! \! \! D(t-s, x, y) \frac{\fhi_1^2}{\fhi_\ast^2} 
  \partial^2_{yy}
   \fhi_\ast (s, y) dy
  \right| ds 
   dt}_{J_{121}} \\ & +
    \underbrace{2 \int_0^T
  \int_0^t  \sup_{x \in ]0, R[} \left|  \int_0^R \! \! \! D(t-s, x, y) \frac{\fhi_1^2}{\fhi_\ast} 
  d \mu^\ast_s (y)  
  \right|  ds 
  dt}_{J_{122}} \\ & + 
   \underbrace{2 \int_0^T
  \int_0^t   \sup_{x \in ]0, R[} \left|  \int_0^R \! \! \! D(t-s, x, y) \frac{\fhi_1^2}{\fhi_\ast} 
  f(s, y, \fhi_\ast) dy  
  \right|  ds 
  dt. }_{J_{123}}
\end{split}
\end{equation}
To control $J_{121}$, we first point out that 
\begin{equation}
\label{e:scarico2}
    \partial_y \left[ \frac{\fhi_1^2}{\fhi_\ast^2} 
  \partial_y
   \fhi_\ast \right] =
  \frac{\fhi_1^2}{\fhi_\ast^2}  \partial^2_{yy} \fhi_\ast +
  2 \frac{\fhi_1}{\fhi_\ast^2}  \partial_y \fhi_1 \partial_y \fhi_\ast 
  - 2 \frac{\fhi_1^2}{\fhi_\ast^3} (\partial_y \fhi_\ast )^2.  
\end{equation}
This implies 
\begin{equation}
\label{e:geiunodueuno}
\begin{split}
    J_{121} & \stackrel{\eqref{e:scarico2}}{\leq}
   K \underbrace{\int_0^T 
  \int_0^t \sup_{x \in ]0, R[} \left|\int_0^R \! \! \! 
  \partial_y D(t-s, x, y) \frac{\fhi_1^2}{\fhi_\ast^2} 
  \partial_y
   \fhi_\ast (s, y) dy
  \right|   ds 
  dt}_{J_{1211}} \\ & + K
  \underbrace{\int_0^T
  \int_0^t  \sup_{x \in ]0, R[} \left| \int_0^R \! \! \! 
   D(t-s, x, y)  \frac{\fhi_1}{\fhi_\ast^2}  \partial_y \fhi_1 \partial_y \fhi_\ast 
   (s, y) dy
  \right|  ds 
   dt}_{J_{1212}} \\ & +
  K  \underbrace{\int_0^T 
  \int_0^t  \sup_{x \in ]0, R[} \left| \int_0^R \! \! \! 
   D(t-s, x, y)  \frac{\fhi_1^2}{\fhi_\ast^3} (\partial_y \fhi_\ast )^2 
   (s, y) dy
   \right|  ds 
  dt.}_{J_{1213}} \\
\end{split} 
\end{equation}
We have 
\begin{equation}
  \label{e:geiunodueunouno}
  \begin{split}
    J_{1211} & \stackrel{\text{H\"older}}{\leq} K \int_0^T \!\!\int_0^t
    \left\| \frac{\fhi_1^2}{\fhi_\ast^2} (s, \cdot)
    \right\|_{L^\infty(]0, R[)} \| \partial_y \fhi_\ast (s, \cdot)
    \|_{L^2 (]0, R[)} \sup_{x \in ]0, R[} \|
    \partial_y D(t-s, x, \cdot)\|_{L^2 (]0, R[ )} ds
    dt 
    \\
    & \stackrel{\eqref{eq:initass},~\eqref{e:linftyelledue}}{\leq}
    C(\alpha_1, M, T, R) \delta \int_0^T \int_0^t \left\|
      \frac{\fhi_1^2}{\fhi_\ast^2} (s, \cdot) \right\|_{L^\infty(]0,
      R[)} \sup_{x \in ]0, R[} \|
    \partial_y D(t-s, x, \cdot)\|_{L^2 (]0, R[ )} ds
    dt 
    \\
    & \stackrel{\eqref{e:grande},\eqref{e:Dxdue}}{\leq} C(\alpha_1, M, h, T, R)
    \delta
    \int_0^T \int_0^t \frac{1}{(t-s)^{\frac{3}{4}}} \| \fhi_1 (s, \cdot) \|^2_{L^\infty (]0, R[)} ds dt 
    \\
    & \leq C(\alpha_1, M, h, T, R) \delta \int_0^T \| \fhi_1 (s, \cdot)
    \|^2_{H^1 (]0, R[)}
    \int_s^T \frac{1}{(t-s)^{\frac{3}{4}}} dt ds
    \\
    & \leq C(\alpha_1, M, h, T, R) \delta \| \fhi_1 \|^2_{L^2 (]0, T[; H^1 (]0,
      R[))}.  \phantom{\int}
  \end{split}
\end{equation}
We also have 
\begin{equation}
  \label{e:geiunodueunodue}
  \begin{split}
    J_{1212} & \stackrel{\eqref{e:grande}}{\leq} \!\!C(h)\!\! \int_0^T
    \!\!\!\int_0^t \!\!\| \fhi_1(s, \cdot) \|_{L^\infty (]0, R[)} \| \partial_y
    \fhi_1 \partial_y \fhi_\ast (s, \cdot) \|_{L^1 (]0, R[)}
    \!\!\sup_{x \in ]0, R[} \!\! \| D(t-s, x, \cdot ) \|_{L^\infty (]0, R[) } ds dt
    \\
    & \stackrel{\text{H\"older}, \eqref{e:Dlinfty}}{\leq} \!\!
    \!\!\!\!\!\!C(h, R)
    \!\int_0^T \!\!\!\int_0^t \!\!\frac{1}{\sqrt{t-s}} \| \fhi_1 (s, \cdot)
    \|_{L^\infty (]0, R[)} \| \partial_y \fhi_1 (s, \cdot) \|_{L^2
      (]0, R[)} \| \partial_y \fhi_\ast (s, \cdot) \|_{L^2 (]0, R[)}
    ds dt
    \\
    & \stackrel{\eqref{e:linftyelledue},\eqref{e:immersione}}{\leq}
    C(\alpha_1, M, h, T, R) \delta \int_0^T \int_0^t \frac{1}{\sqrt{t-s}} \|
    \fhi_1 (s, \cdot) \|_{H^1 (]0, R[)} \| \partial_y \fhi_1 (s,
    \cdot) \|_{L^2 (]0, R[)} ds dt
    \\
    & \leq C(\alpha_1, M, h, T, R) \delta \int_0^T \| \fhi_1 (s, \cdot)
    \|^2_{H^1 (]0, R[)} \int_s^t \frac{1}{\sqrt{t-s}} dt ds
    \\
    & \leq C(\alpha_1, M, h, T, R) \delta \| \fhi_1 \|^2_{L^2 (]0, T[; H^1 (]0,
      R[)).}  
  \end{split}
\end{equation}
Finally, we have 
$$
    J_{1213} \leq 
   \int_0^T 
  \int_0^t 
  \left\| \frac{\fhi_1^2}{\fhi_\ast^3}
  (s, \cdot) \right\|_{L^\infty (]0, R[)}
   \|  \partial_y \fhi_\ast (s, \cdot) \|^2_{L^2 (]0, R[)}
   \sup_{x \in ]0, R[}  \|D(t-s, x, \cdot ) \|_{L^\infty (]0, R[) }
  ds dt 
  $$
  and by arguing as in~\eqref{e:geiunodueunouno} and \eqref{e:geiunodueunodue}
  we eventually arrive at 
\begin{equation}
\label{e:geiunodueunotre}
      J_{1213}
      \leq 
    C(\alpha_1, M, h, T, R) \delta^2 
    \| \fhi_1  \|^2_{L^2 (]0, T[; H^1 (]0, R[))}. 
\end{equation}
By combining~\eqref{e:geiunodueunouno},~\eqref{e:geiunodueunodue} and \eqref{e:geiunodueunotre} and recalling that $\delta \leq 1$ we obtain 
\begin{equation}
\label{e:geiunodueunofinal}
      J_{121} =  J_{1211}+  J_{1212} +  J_{1213} \leq 
     C(\alpha_1, M, h, T, R) \delta 
    \| \fhi_1  \|^2_{L^2 (]0, T[; H^1 (]0, R[)).} 
\end{equation} 
We now focus on $J_{122}$, which is defined in~\eqref{e:geiunodue}. We control 
 it by arguing as follows:
\begin{equation}
\label{e:geiunoduedue}
\begin{split}
  J_{122} & \stackrel{\eqref{e:grande}}{\leq}  
  C(h) \int_0^T 
  \int_0^t 
  \| \mu_s^\ast \|_{\mathcal M(]0, R[)} 
   \|  \fhi_1 (s, \cdot) \|^2_{L^\infty (]0, R[)}
    \sup_{x \in ]0, R[}  \| D(t-s, x, \cdot ) \|_{L^\infty (]0, R[) }
  ds dt \\
  &  \stackrel{\eqref{e:Dlinfty},\eqref{e:immersione}}{\leq}
  C(h, R) \sup_{t \in ]0, T[}  \| \mu_t^\ast \|_{\mathcal M(]0, R[)} 
  \int_0^T \int_0^t \frac{1}{\sqrt{t-s}} 
    \|  \fhi_1 (s, \cdot) \|^2_{H^1 (]0, R[)} ds dt  \\
    &  \stackrel{\eqref{eq:initass}}{\leq}
    C(h, R) \delta \int_0^T   
     \|  \fhi_1 (s, \cdot) \|^2_{H^1 (]0, R[)} \int_s^T 
     \frac{1}{\sqrt{t-s}} 
      dt ds \\ & \leq C(h, T, R) \delta 
      \|  \fhi_1  \|^2_{L^2 (]0, T[; H^1 (]0, R[))}. \phantom{\int}
\end{split}
\end{equation}
To control $J_{123}$, we first point out that
\begin{equation}
\label{e:effeast}
   \| f(\cdot, \cdot, \fhi_\ast) \|_{L^\infty (]0, T[ \times ]0, R[)}
   \stackrel{\eqref{eq:2.2}}{\leq}
   \alpha_1 \| \fhi_\ast - h \|_{L^\infty (]0, T[ \times ]0, R[)}
   \stackrel{\eqref{eq:initass},\eqref{e:inftyast}}{\leq}
   C(\alpha_1,M, T, R) \delta. 
\end{equation}
Next, we recall that $J_{123}$ is defined in~\eqref{e:geiunodue} and 
we control it by arguing as follows:
\begin{equation}
\label{e:geiunoduetre}
\begin{split}
     J_{123} & \stackrel{\eqref{e:grande}}{\leq} 
     C(h) \int_0^T  \! \!
  \int_0^t   
   \|  \fhi_1 (s, \cdot) \|^2_{L^\infty (]0, R[)}
   \| f(s, \cdot, \fhi_\ast) \|_{L^\infty ( ]0, R[)}
   \! \sup_{x \in ]0, R[}  \| D(t-s, x, \cdot ) \|_{L^1 (]0, R[) }
  ds dt \\
  & \stackrel{\eqref{e:effeast},\eqref{e:immersione}}{\leq}
    C(\alpha_1,M, h, T, R) \delta
   \int_0^T \! \!
  \int_0^t  
   \|  \fhi_1 (s, \cdot) \|^2_{H^1 (]0, R[)}
   \sup_{x \in ]0, R[}  \| D(t-s, x, \cdot ) \|_{L^1 (]0, R[) } 
   ds dt  \\
   &  \stackrel{\eqref{e:Duno}}{\leq}
   C(\alpha_1,M, h, T, R) \delta 
   \int_0^T 
  \int_0^t  
   \|  \fhi_1 (s, \cdot) \|^2_{H^1 (]0, R[)}ds dt \\
   & \leq C(\alpha_1,M, h, T, R) \delta 
    \|  \fhi_1  \|^2_{L^2 (]0, T[; H^1 (]0, R[))}. \phantom{\int}
\end{split}
\end{equation}
By recalling that $\delta \leq 1$ we arrive at
\begin{equation}
\label{e:geiunodue2}
 J_{12} \stackrel{\eqref{e:geiunodue}} \le
  J_{121}+J_{122} + J_{123} \stackrel{\eqref{e:geiunodueunofinal},\eqref{e:geiunoduedue},\eqref{e:geiunoduetre}}{\leq}
  C(\alpha_1, h, M, R, T) 
    \|  \fhi_1  \|^2_{L^2 (]0, T[; H^1 (]0, R[))}. 
\end{equation}
Finally, we recall~\eqref{e:geiuno2} and we conclude that
\begin{equation}
\label{e:geiuno3}
    J_1 \stackrel{\eqref{e:geiuno2}}{\leq} 
    C(h, R)  \|  \fhi_1  \|^2_{L^2 (]0, T[; H^1 (]0, R[))} + J_{11} + J_{12}
    \stackrel{\eqref{e:geiunounofinal},~\eqref{e:geiunodue2}}{\leq}
    C(\alpha_1, h, M, R, T) 
    \|  \fhi_1  \|^2_{L^2 (]0, T[; H^1 (]0, R[))}. 
\end{equation}
We now focus on the term $J_2$. We recall that $J_2$ is defined as in~\eqref{e:jeis} and we preliminary point out that
$$
   \frac{\fhi_1 }{\fhi_\ast} \pxx \fhi_1 = \px \left[
   \frac{\fhi_1 }{\fhi_\ast} \px \fhi_1
   \right] -
   \frac{(\px \fhi_1)^2}{\fhi_\ast} +
   \frac{\fhi_1 \px \fhi_1}{\fhi^2_\ast}
   \px \fhi_\ast.
$$
Owing to the Integration by Parts Formula, this implies 
\begin{equation}
  \label{e:geiduescarico}
  \begin{split}
    \int_0^T \int_0^t \sup_{x \in ]0, R[} & \left| \int_0^R \! \! \!
      D(t-s, x, y) \frac{\fhi_1 }{\fhi_\ast} \pxx \fhi_1(s, y) dy
    \right| ds dt 
    \\
    & \leq \underbrace{ \int_0^T \int_0^t \sup_{x \in
        ]0, R[} \left| \int_0^R \! \! \! \partial_y D(t-s, x, y) \frac{\fhi_1
        }{\fhi_\ast} \partial_y \fhi_1(s, y) dy \right| ds
      dt  }_{J_{21}} 
    \\
    & + \underbrace{ \int_0^T \int_0^t \sup_{x \in ]0, R[} \left|
        \int_0^R \! \! \! D(t-s, x, y) \frac{( \px \fhi_1)^2
        }{\fhi_\ast} (s, y) dy \right| ds
      dt }_{J_{22}} 
    \\
    & + \underbrace{ \int_0^T \int_0^t \sup_{x \in ]0, R[} \left|
        \int_0^R \! \! \! D(t-s, x, y) \frac{ \fhi_1 }{\fhi^2_\ast}
        \px \fhi_1 \px \fhi_\ast (s, y) dy \right| ds dt .}_{J_{23}}
  \end{split}
\end{equation} 
To control $J_{21}$ we argue as follows:
\begin{equation}
  \label{e:geidueuno}
  \begin{split}
    J_{21} & \stackrel{\eqref{e:grande}}{\leq} C(h) \int_0^T \int_0^t
    \sup_{x \in ]0, R[} \| \partial_y D(t-s, x, \cdot) \|_{L^2(]0,R[)}
    \| \fhi_1(s, \cdot) \|_{L^\infty (]0, R[ )} \| \px \fhi_1(s,
    \cdot) \|_{L^2 (]0, R[)} ds
    dt
    \\
    & \stackrel{\eqref{e:immersione}}{\leq} C(h, R) \int_0^T \int_0^t
    \sup_{x \in ]0, R[} \| \partial_y D(t-s, x, \cdot) \|_{L^2(]0,R[)} \|
    \fhi_1(s, \cdot) \|^2_{H^1 (]0, R[ )} ds
    dt
    \\
    & \stackrel{\eqref{e:Dxdue}}{\leq} C( h,R) \int_0^T \int_0^t
    \frac{1}{(t-s)^{3/4}} \| \fhi_1(s, \cdot) \|^2_{H^1 (]0, R[ )} ds
    dt
    \\
    & = C(h, R) \int_0^T \| \fhi_1(s, \cdot) \|^2_{H^1 (]0, R[ )}
    \int_s^T \frac{1}{(t-s)^{3/4}} dt
    ds
    \\
    & \leq C(h, R, T) \| \fhi_1 \|^2_{L^2 (]0, T[ ;H^1 (]0, R[
      ))}. \phantom{\int}
  \end{split}
\end{equation}
Next, we control $J_{22}$ by arguing as follows: 
\begin{equation}
\label{e:geiduedue}
\begin{split}
   J_{22} 
   & \stackrel{\eqref{e:grande}}{\leq} C(h)
    \int_0^T  \int_0^t  \sup_{x \in ]0, R[}  
     \| D(t-s, x, \cdot)\|_{L^\infty(]0, R[)} 
  \| \px \fhi_1 (s, \cdot)\|^2_{L^2 (]0, R[)}   ds 
  dt \\
  &  \stackrel{\eqref{e:Dlinfty}}{\leq} C(h)
   \int_0^T  
  \int_0^t   \frac{1}{\sqrt{t-s}} 
  \| \fhi_1 (s, \cdot)\|^2_{H^1 (]0, R[)}   ds 
   dt \\ 
   & \leq C(h) \int_0^T  \| \fhi_1 (s, \cdot)\|^2_{H^1 (]0, R[)} \int_s^T 
   \frac{1}{\sqrt{t-s}} dt ds 
   \\ & \leq C(h, T) \| \fhi_1 \|^2_{L^2 (]0, T[ ;H^1 (]0, R[ ))}. \phantom{\int} 
\end{split}
\end{equation}
Finally, we control $J_{23}$:
\begin{equation}
\label{e:geiduetre}
\begin{split}
J_{23}  & \stackrel{\eqref{e:grande}}{\leq} C(h)
\int_0^T 
  \int_0^t  \sup_{x \in ]0, R[} 
    \int_0^R \! \! \!   
  \left| D(t-s, x, y) \fhi_1  \px \fhi_1 \px \fhi_\ast (s, y)   \right| dy
   ds 
  dt \\
  &  \stackrel{\text{H\"older}}{\leq}
  C(h)
   \int_0^T \sup_{x \in ]0, R[} 
  \int_0^t \| D(t-s, x, \cdot) \|_{L^\infty(]0, R[)}  \\
  & \qquad \qquad \times 
  \| \fhi_1  (s, \cdot)  \|_{L^\infty(]0, R[)}  
  \| \px \fhi_1 (s, \cdot) \|_{L^2(]0, R[)}
  \| \px \fhi_\ast (s, \cdot)\|_{L^2(]0, R[)}   
   ds 
  dt \\
  &  \stackrel{\eqref{e:Dlinfty},\eqref{e:immersione}}{\leq}
  C(h,R)
   \int_0^T
  \int_0^t \frac{1}{\sqrt{t-s}}  
  \| \fhi_1  (s, \cdot)  \|^2_{H^1(]0, R[)} 
  \| \px \fhi_\ast (s, \cdot)\|_{L^2(]0, R[)}   
   ds 
  dt \\
   &  \stackrel{\eqref{e:linftyelledue}}{\leq}
  C(\alpha_1, M,h, T, R) \delta 
   \int_0^T
   \| \fhi_1  (s, \cdot)  \|^2_{H^1(]0, R[)}   
  \int_s^t \frac{1}{\sqrt{t-s}} 
    dt
  ds \\ & \leq 
  C(\alpha_1, M, h,T, R) \delta \| \fhi_1 \|^2_{L^2 (]0, T[ ;H^1 (]0, R[ ))}. \phantom{\int} 
\end{split}
\end{equation}
By combining the above inequalities and recalling that $\delta\leq 1$
we arrive at
\begin{equation}
\label{e:geidue2}
   J_2 \stackrel{\eqref{e:geiduescarico}}{\le}
   J_{21}+J_{22}+J_{23} \stackrel{\eqref{e:geidueuno},\eqref{e:geiduedue},\eqref{e:geiduetre}}{\leq}
    C(\alpha_1, M,h, T, R)  \| \fhi_1 \|^2_{L^2 (]0, T[ ;H^1 (]0, R[ ))}.   
\end{equation}
We now focus on the term $J_3$, which is defined as in~\eqref{e:jeis}, and we control it by arguing as follows: 
\begin{equation}
  \label{e:geitre}
  \begin{split}
    J_3 & \stackrel{\eqref{e:grande}}{\leq} C(h) \int_0^T \int_0^t \| \fhi_1
    (s, \cdot) \|^2_{L^\infty(]0, R[)} \|\mu^\ast_s \|_{\mathcal M
      (]0, R[)} \sup_{x \in ]0, R[} \| D(t-s, x, \cdot)
    \|_{L^\infty(]0, R[)} ds dt 
    \\
    &
    \stackrel{\eqref{eq:initass},\eqref{e:Dlinfty}}{\leq} C(h) \delta
    \int_0^T \int_0^t \! \! \frac{1}{\sqrt{t-s}} \| \fhi_1 (s, \cdot)
    \|^2_{L^\infty(]0, R[)} ds dt 
    \\
    &
    \stackrel{\eqref{e:immersione}}{\leq} C(h, R) \delta \int_0^T \|
    \fhi_1 (s, \cdot) \|^2_{H^1(]0, R[)} \int_s^T \frac{1}{\sqrt{t-s}}
    dt ds
    \\
    & \leq C(h, T, R) \delta \| \fhi_1 \|^2_{L^2 (]0, T[ ;H^1 (]0, R[
      ))}. \phantom{\int}
  \end{split}
\end{equation}
Recalling that $J_4$ is defined in~\eqref{e:jeis}, we have:
\begin{equation}
\label{e:geiquattro}
\begin{split}
   J_4 & \leq K
   \int_0^T
  \int_0^t   \sup_{x \in ]0, R[}      
  \| D(t-s, x, \cdot) \|_{L^1(]0, R[)}
  \| \fhi_1  (s, \cdot) \|^2_{L^\infty(]0, R[)} 
  \| \partial_\fhi f(s, \cdot, \fhi_\ast) 
  \|_{L^\infty (]0, R[)}
   ds 
  dt \\ &
  \stackrel{\eqref{eq:2.2},\eqref{e:Duno}}{\leq}
  C(\alpha_1) \int_0^T 
  \int_0^t  \! \!     
  \| \fhi_1  (s, \cdot) \|^2_{L^\infty(]0, R[)} 
   ds 
  dt \\ 
  & \stackrel{\eqref{e:immersione}}{\leq} 
  C(\alpha_1, T, R) \| \fhi_1 \|^2_{L^2 (]0, T[ ;H^1 (]0, R[ ))}. \phantom{\int}
\end{split}
\end{equation}
Finally, for the term $J_5$, defined in~\eqref{e:jeis}, we have:
\begin{equation}
  \label{e:geicinque}
  \begin{split}
    J_5 & \stackrel{\eqref{e:grande}}{\leq}
    K \int_0^T \int_0^t \sup_{x \in ]0, R[} \| D(t-s, x,
    \cdot) \|_{L^1(]0, R[)} \| \fhi_1 (s, \cdot) \|^2_{L^\infty(]0,
      R[)} \| f(s, \cdot, \fhi_\ast) \|_{L^\infty (]0,
      R[)} ds dt
    \\
    & \stackrel{\eqref{e:maxprin},\eqref{e:boundsueffe}, \eqref{e:Duno}}{\leq}
    C(\alpha_1, M, F) \int_0^T \int_0^t \! \!  \| \fhi_1 (s, \cdot)
    \|^2_{L^\infty(]0, R[)} ds
    dt \\
    & \stackrel{\eqref{e:immersione}}{\leq} C(\alpha_1, M, F, T, R) \|
    \fhi_1 \|^2_{L^2 (]0, T[ ;H^1 (]0, R[ ))}.
  \end{split}
\end{equation}
By combining~\eqref{e:geiuno2},~\eqref{e:geidue2},~\eqref{e:geitre},~\eqref{e:geiquattro}, and~\eqref{e:geicinque} we eventually arrive at 
\begin{equation}
  \label{e:finegei}
  \int_0^T  I(t) dt \stackrel{\eqref{e:jeis}}{=} J_1 + J_2 + J_3 + J_4 + J_5 
  \leq
  C(\alpha_1, M, h, T, R, h, F)  \| \fhi_1 \|^2_{L^2 (]0, T[ ;H^1 (]0, R[ ))}
\end{equation}
and this concludes the proof of Lemma~\ref{l:iditi}.  
\end{proof}
\subsection{Conclusion of the proof of Lemma~\ref{l:uni1}}
\label{ss:tuttorigoroso}
We use~\eqref{e:intfhi2} and we get 
\begin{equation}
\begin{split}
   \left| \int_0^T \! \! \int_0^R 
   \fhi_2 (t, x) d \mu_t^* (x) dt \right|
  &  \leq \int_0^T
   \| \mu_t^* \|_{\mathcal M (]0, R[)} 
   \| \fhi_2 (t, \cdot) \|_{L^\infty (]0, R[)}dt  \\
   & \leq \sup_{t \in ]0, R[} 
   \| \mu_t^* \|_{\mathcal M (]0, R[)}  
   \int_0^T 
   \| \fhi_2 (t, \cdot) \|_{L^\infty (]0, R[)}dt  \\
   & \stackrel{\eqref{eq:initass},\eqref{e:intfhi2}}{\leq}
   C (\alpha_1, \alpha_2, M, h, T, R, F)
   \| \fhi_1 \|^2_{L^2(]0, T[ ; H^1(]0, R[ ))}
   \delta . 
   \end{split}
\end{equation}
Next, we combine~\eqref{e:integral} with the above inequality and we conclude that 
\begin{equation}
\label{e:stimafinale}
\begin{split}
   2\int_0^T  \! \! \! \int_0^R \vfi_1(t,x) d\nu_t(x)dt +& 
   \int_0^T \! \! \! \int_0^R 
   \fhi_2 (t, x) d \mu^\ast_t (x) dt \\ & 
\leq  \| \fhi_1 \|^2_{L^2 (  ]0, T[; H^1 (]0, R[))} 
\Big[ 
- C(M, h_\ast) +  C(\alpha_1, \alpha_2, M, h, T, R, F)  \delta
\Big].
\end{split}
\end{equation}
In the previous expression, the quantity at the right hand side is negative provided  
that the constant $\delta$ is sufficiently small. This establishes~\eqref{e:keypoint} and concludes the formal proof of Lemma~\ref{l:uni1}. 

To complete the proof of Lemma~\ref{l:uni1} we are left to make rigorous the formal argument given so far. To this end, we rely on an approximation argument. 
First, we recall the equality
\begin{equation}
\label{e:essupnu}
  \mathrm{ess \, sup}_{t \in ]0, T[} \| \nu_t \|_{\mathcal M(]0, R[)} \stackrel{\eqref{e:nueps},\eqref{e:nu}}{=} 1
\end{equation} 
and we point out that, by passing to the limit in the inequality~\eqref{e:conclusione2}, we get that $\fhi_1$ satisfies 
\begin{equation}
  \label{e:rigfhiuno}
  \begin{split}
    \int_0^R \fhi_1^2 (t, x) dx + \int_0^T \! \! \int_0^R (\px
    \fhi_1)^2 dx dt & + \int_0^T \! \! \int_0^R \fhi_1^2 d \mu^\ast_t
    (x) dt
    \\
    &
    \leq C(\alpha_1, M, F, T, R) \mathrm{ess \, sup}_{t \in ]0, T[} \| \nu_t \|^2_{\mathcal M (]0, R[)} 
    \\
    & \stackrel{\eqref{e:essupnu}}{\leq}C(\alpha_1, M, F, T, R), \quad
    \text{for every $t \in ]0, T[.$ } \phantom{\int}
  \end{split}
\end{equation}
By relying on analogous computations and by using~\eqref{e:rigfhiuno} we infer that 
\begin{equation}
\label{e:rigfhidue}
\begin{split}
   \int_0^R \fhi_2^2 (t, x) dx &
   + \int_0^T \! \! \int_0^R (\px \fhi_2)^2 dx dt 
   +  \int_0^T \! \! \int_0^R \fhi_2^2 d \mu^\ast_t (x) dt \\
   &
  \leq C(\alpha_1, \alpha_2, M, F, T, R),
   \quad \text{for every $t \in ]0, T[.$ } \phantom{\int}
   \end{split}
\end{equation}
Next, we fix three sequences of smooth functions 
$$
\nu_k : ]0, T[ \times ]0, R[ \to \R, \qquad  
\mu^\ast_k: ]0, T[ \times ]0, R[ \to \R \quad \text{and}
\quad 
\fhi_{0k}: ]0, R[ \to \R
$$
such that 
\begin{eqnarray}
& \nu_{kt} \weaks \nu_t \; \text{weakly in $\mathcal M (]0, R[)$}, \;
\| \nu_{kt} \|_{\mathcal M (]0, R[)} \leq 
\esssup\limits_{t \in ]0, T[} \| \nu_t \|_{\mathcal M (]0, R[)}
\; \text{for a.e. $t \in ]0, T[$,} \\
& \mu^\ast_{kt} \weaks \mu^\ast_t \; \text{weakly in $\mathcal M (]0, R[)$}, \;
\| \mu^\ast_{kt} \|_{\mathcal M (]0, R[)} \leq 
\esssup\limits_{t \in ]0, T[} \| \mu^\ast_t \|_{\mathcal M (]0, R[)}
\; \text{for a.e. $t \in ]0, T[$,} \\ 
& \fhi_{0k} \to \fhi_0 \; \text{strongly in $H^1(]0, R[)$}.
\end{eqnarray}
We term $\fhi_{\ast k}$, $\fhi_{1k}$ and $\fhi_{2k}$ the corresponding solutions of the initial-boundary value problems~\eqref{e:fhiast},~\eqref{e:defhiuno} and~\eqref{eq:Euler2}, respectively. Since the coefficient $\mu^\ast_k$ and $\nu_k$ 
and the initial datum $\fhi_{0k}$ are all smooth, then one can show that the solutions 
$\fhi_{\ast k}$, $\fhi_{1k}$ and $\fhi_{2k}$ are also smooth.  This implies that the formal argument given at the previous paragraphs is completely justified and one gets 
\begin{equation}
\label{e:dovepassoalimite}
\begin{split}
2 \int_0^T  \! \! \! \int_0^R \vfi_{1k}(t,x) d\nu_{tk}(x)dt
&  +
   \int_0^T \! \! \! \int_0^R 
   \fhi_{2k} (t, x) d \mu^\ast_{tk} (x) dt \\
   & \leq 
   \| \fhi_{1k} \|^2_{L^2 (  ]0, T[; H^1 (]0, R[))} 
\Big[ 
- C(M, h) +  C(\alpha_1, \alpha_2, M, h, T, R)  \delta
\Big] \\
& \leq -\frac{C(M, h_\ast)}{2}   \| \fhi_{1k} \|^2_{L^2 (  ]0, T[; H^1 (]0, R[))}
 \quad 
   \text{for every $k$}
   \end{split} 
\end{equation}
provided that the constant $\delta$ is sufficiently small. 
Next, we point out that $\fhi_{\ast k}$, $\fhi_{1k}$ and $\fhi_{2k}$ satisfies the inequality~\eqref{e:reglim},~\eqref{e:rigfhiuno} and~\eqref{e:rigfhidue}, respectively. By arguing as in the proof of Theorem~\ref{th:main1}, one can show that 
\begin{equation*}
   \fhi_{\ast k} \to \fhi_\ast, \quad 
   \fhi_{1k} \to \fhi_1, \quad 
   \fhi_{2k} \to \fhi_2 \quad 
   \text{strongly in $L^2 (]0, T[, C^0([0, R]))$.}
\end{equation*}
Also,
\begin{equation*}
   \px \fhi_{\ast k} \weak \px \fhi_\ast, \quad 
   \px \fhi_{1k} \weak \px \fhi_1, \quad 
   \px \fhi_{2k} \weak \px \fhi_2 \quad 
   \text{weakly in $L^2 (]0, T[ \times ]0, R[)$.}
\end{equation*}
In particular, the above convergence results imply that 
\begin{equation}
\label{e:liminf}
  \| \fhi_1 \|^2_{L^2 (]0, T[; H^1 (]0, R[)) }
  \leq \liminf_{k \to + \infty}
  \| \fhi_{1k} \|^2_{L^2 (]0, T[; H^1 (]0, R[)) }
\end{equation}
and, by arguing as in the estimate of~\eqref{e:al:ultimo}, that
\begin{equation}
\label{e:convergenzaint}
 \int_0^T  \! \! \! \int_0^R \vfi_{1k} d\nu_{tk}(x)dt \to
 \int_0^T  \! \! \! \int_0^R \vfi_{1}d\nu_{t}(x)dt, \quad 
 \int_0^T \! \! \! \int_0^R 
   \fhi_{2k}  d \mu^\ast_{tk} (x) dt
    \to
 \int_0^T \! \! \! \int_0^R 
   \fhi_{2} d \mu^\ast_{t} (x) dt.
 \end{equation}
By passing to the limit in~\eqref{e:dovepassoalimite} we get
\begin{equation}
\begin{split}
  2 \int_0^T  \! \! \! \int_0^R \vfi_{1}(t,x) d\nu_{t}(x)dt
&  +
   \int_0^T \! \! \! \int_0^R 
   \fhi_{2} (t, x) d \mu^\ast_{t} (x) dt \\
   & \stackrel{\eqref{e:convergenzaint}}{=}
    \lim_{k \to + \infty}
     2 \int_0^T  \! \! \! \int_0^R \vfi_{1k}(t,x) d\nu_{tk}(x)dt
  +
   \int_0^T \! \! \! \int_0^R 
   \fhi_{2k} (t, x) d \mu^\ast_{tk} (x) dt \\
   & \stackrel{\eqref{e:dovepassoalimite}}{\leq}
   \limsup_{k \to + \infty}
   -\frac{C(M, h_\ast)}{2}
    \| \fhi_{1k} \|^2_{L^2 (]0, T[; H^1 (]0, R[)) } \\
    & = - \frac{C(M, h_\ast)}{2} \liminf_{k \to + \infty}
    \| \fhi_{1k} \|^2_{L^2 (]0, T[; H^1 (]0, R[)) } \\
    & \stackrel{\eqref{e:liminf}}{\leq}
    - \frac{C(M, h_\ast)}{2} 
    \| \fhi_1 \|^2_{L^2 (]0, T[; H^1 (]0, R[)) }.
\end{split}
\end{equation}
This establishes~\eqref{e:keypoint} and hence concludes the proof of Lemma~\ref{l:uni1}. 
\section{Solutions of the differential game and Nash equilibria}
\label{sec:7}
This section aims at the discussing the differential game modeling the case when there are several competing fish companies and at establishing the existence of Nash equilibria. 
More precisely, we define our differential game as follows: we assume that there are $m>1$ players (i.e., fish companies) and we denote by $\mu_i$ the fishing intensity of the $i$-th company. We term $\fhi$ the fish population density and we consider the initial-boundary value problem 
\begin{equation}
\label{eq:2N}
\begin{cases}
 \pt\vfi = \pxx\vfi - \vfi \displaystyle{\sum_{i=1}^m} \mu_m + f(t,x,\vfi) \fhi,
    & \text{in $]0, T[ \times ]0, R[$},
    \\
      \px \fhi(t,0)=\px\vfi(t,R) = 0, & t \in \, ]0, T[,
     \\
    \vfi(0,x)=\vfi_0(x), 
    & x\in{]0,R[}. \phantom{\displaystyle{\int}}
  \end{cases}
\end{equation}
The goal of the $i$-th player (i.e., fish company) is to maximize his payoff $J_i$, which is defined by setting
\begin{equation}
\label{eq:3.1N}
J_i(\mu):=\int_0^T \! \! 
\int_0^R \vfi(t,x) d\mu_{i,t}(x)dt-\Psi_i\left(\int_0^T  \! \!  \int_0^R c_i(t,x)d\mu_{i,t}(x)dt\right). 
\end{equation}
The admissible controls satisfy $\mu_i \in L^\infty(]0,T[; {\mathcal M}_+(]0,R[))$ and the constraint 
\begin{equation}
\label{eq:3.2N}
\int_0^R b_i(t,x)d\mu_{i,t}(x)  \leq 1, \quad \text{for a.e. $t \in ]0, T[$}.
\end{equation}
The functions $\Psi_i$, $c_i$ and $b_i$ in~\eqref{eq:3.1N} and~\eqref{eq:3.2N}  satisfy the following assumptions. 
\begin{itemize}
\item[{\bf (H.7)}]
\label{h:accasette} 
The functions $\Psi_1, \dots,\Psi_m: \R \to \R$ are twice continuously differentiable, nondecreasing, and convex. 
\item[{\bf (H.8)}]
\label{h:accaotto}
The functions $c_1, \dots ,c_m: [0,T]\times[0,R] \to \R^+ \cup \{ + \infty \}$  are all lower semi-continuous. The functions $b_1, \dots ,b_m: [0,T]\times[0,R] \to \R \cup \{ + \infty \}$ are lower semi-continuous and satisfy
\begin{equation}
\label{eq:3.3N}
b_i(t,x) \ge b_0>0, \qquad \text{for all $(t,x)\in [0,T]\times[0,R]$ and $i=1,...,m$},
\end{equation}
for some positive constant $b_0> 0$.
\end{itemize}
We now provide the definition of Nash equilibrium. 
\begin{definition}
  \label{def:solN}
  A Nash equilibrium solution for the differential game~\eqref{eq:2N}
  is an $m$-tuple $(\mu_1,..., \mu_m)$ such that, for every
  $i \in \left\{1, \ldots, m\right\}$, $\mu_i
  \in L^\infty(]0,T[; {\mathcal M}_+(]0,R[))$ is
  a solution of the problem
  \begin{equation*}
    \text{maximize $J_i(\mu_i)$, defined as in~\eqref{eq:3.1N} among
      $\mu_i \in L^\infty \big(]0, T[; \mathcal M_+ (]0, R[) \big)$
      satisfying~\eqref{eq:3.2N}}.
  \end{equation*}
\end{definition}
The main result of the present section establishes the existence of Nash equilibria. 
\begin{theorem}
\label{th:mainNash}
Assume {\bf (H.1)}-{\bf (H.2)} and {\bf (H.6)}-{\bf (H.8)}.  There is a constant $\delta>0$, which only depends on the constants $\alpha_1$, $\alpha_2$, $M$, $h$, $T$, $R$ and $h_\ast$ such that, if\begin{equation}
\label{eq:initassN}
\norm{\vfi_0-h}_{H^1(]0,R[)}\le \delta,
\end{equation}
then the differential game~\eqref{eq:2N} has a Nash equilibrium $(\mu_1,...,\mu_m)$ such that 
\begin{equation}
\label{eq:measassN}
\mathrm{ess \; sup}_{t \in ]0, T[} \| \mu_{i, t} \|_{\mathcal M(]0, R[)} \leq \delta
\quad \text{for every $i=1, \dots, m$.}
\end{equation}
\end{theorem}
\begin{proof}
We follow the same argument as in~\cite[\S6]{BS1} and,  to simplify the exposition, we only discuss the case when $m=2$ and we assume $c_1=c_2$, $\Psi_1=\Psi_2$. The proof straightforwardly extends to the general case. We proceed according to the following steps. \\
{\sc Step 1:} we introduce some notation and make some preliminary considerations. 
First, we fix ${\eta \in L^\infty (]0, T[ ; \mathcal M_+ (]0, R[))}$ and 
we consider the function 
\begin{equation}
\label{e:N:gei}
\begin{split}
    & \qquad \qquad J_\eta: 
    L^\infty (]0, T[ ; \mathcal M_+ (]0, R[))\to \R, \\
    &J_\eta (\mu) : =
    \int_0^T \int_0^R 
    \fhi (t, x) d \mu_t (x) dt - \Psi \left( 
    \int_0^T \int_0^R 
    c (t, x) d \mu_t (x) dt \right), 
    \end{split}
\end{equation}
where $\fhi$ is the weak solution of the initial-boundary value problem
\begin{equation}
\label{e:N:ibvp}
  \begin{cases}
\pt \fhi = \pxx \fhi -  [\eta + \mu] \fhi + g(t,x, \fhi), & \text{in $]0, T[ \times ]0, R[$} ,  \\
\px \fhi(t,0)=\px\vfi(t,R) = 0, & t \in ]0, T[, \\
\fhi(0,x) = \fhi_0(x), & x \in ]0,R[.
    \end{cases}
\end{equation}
Next, we fix a small constant $\delta>0$. The precise value of $\delta$ will be determined in the following. We define the set $\mathcal C_\delta$ by setting 
\begin{equation}
\label{e:cidelta}
  \mathcal C_\delta : = 
  \left\{ 
  \mu \in L^\infty (]0, T[ ; \mathcal M_+ (]0, R[)): \;
  \mathrm{ess \; sup}_{t \in ]0, T[} \| \mu_t \|_{\mathcal M (]0, R[)}  
  \leq \delta 
  \right\}. 
\end{equation} 
By using the same argument as in the proof of Theorem~\ref{th:main2} (existence) and Theorem~\ref{th:main3} (uniqueness) we arrive at the following result. 
\begin{lemma}
\label{l:exunash}
Under the same assumptions as in the statement of Theorem~\ref{th:mainNash}, there is a sufficiently small constant $\delta$ such that, if~\eqref{eq:initassN} holds, then for every $\eta \in C_\delta$ there is a unique 
$\mu^{opt} (\eta) \in C_\delta$ such that 
\begin{equation}
  \label{e:N:opt}
  J_\eta (\mu^{opt}(\eta) ) \ge J_\eta (\mu) 
  \qquad \text{for every $\mu \in C_\delta$}.  
\end{equation}
\end{lemma}
By relying on Lemma~\ref{l:exunash} we can define the map $T$ by setting 
\begin{equation}
\label{e:N:T}
\begin{split}  
      & \quad \quad T: C_\delta \times C_\delta \to 
      C_\delta \times C_\delta \\
      &
      T(\mu_1, \mu_2) : = \left(\mu^{opt}(\mu_2) , \mu^{opt}(\mu_1) \right). 
\end{split}
\end{equation} 
We now show that $C_\delta$ is compact with respect to the weak-$^\ast$ convergence. First, we fix a sequence $\{ \mu_n \} $ in $C_\delta$ and we recall that the Borel measure $\mu_n$ 
on $]0,T[ \times ]0,R[$ is defined by setting
\begin{equation}
\label{e:cosaemuenne}
  \mu_{n} (E) : = \int_0^T \! \! \int_0^R  
    \mathbbm{1}_E (t, x) d \mu_{n,t} (x) dt.
\end{equation}
Note that, if $\mu_n \in C_\delta$, then the total variation $|\mu_n | \leq \delta T$. Hence, there is a Borel measure $\mu$ such that, up to subsequences, $\mu_n \weaks \mu$ in $\mathcal M (]0, T[ \times ]0, R[)$, namely 
$$
   \int_0^T \! \! \int_0^R  
   v (t, x) d \mu_{n, t} (x) dt \to 
   \int_0^T \! \! \int_0^R  
   v (t, x) d \mu (t, x) .
$$
for every $v \in C\left([0,T] \times [0,R]\right)$.
We now have to show that the limit measure $\mu \in C_\delta$, namely it admits a representation like~\eqref{e:cosaemuenne}.  To this end, we term $\pi$ the projection 
$$
    \pi : ]0, T[ \times ]0, R[ \to ]0, T[, \qquad (t, x) \mapsto t
$$
and we point out that
\begin{equation}
\label{e:project}
   \pi_\sharp \mu_n = f_n \mathcal  L^1 \big|_{]0, T[} \qquad \text{for every $n$}.
\end{equation}
In the above expression, $\pi_\sharp \mu_n$ denotes the push-forward of the measure $\mu_n$ and $\mathcal L^1 \big|_{]0, T[} $ the restriction of the Lebesgue measure. Also, the density $f_n$ is given by 
$$
    f_n (t) : = \| \mu_{ n, t} \|_{\mathcal M(]0, R[)} \leq \delta \quad \text{for a.e. $t \in ]0, T[$}.    
$$ 
Next, we point out that, by possibly extracting a further subsequence, we can assume that the sequence
$$
    f_n \weaks f \quad \text{weakly$^\ast$ in $L^\infty (]0, T[)$}
$$
for some accumulation point $f$ satisfying $\| f \|_{L^\infty (]0, T[ )} \leq \delta$. 
By passing to the weak$^\ast$ limit on both sizes of the equality~\eqref{e:project} we obtain 
$$
     \pi_\sharp \mu = f \mathcal  L^1 \big|_{]0, T[} .
$$ 
By using the Disintegration Theorem~\cite[Theorem 2.28]{AmbrosioFuscoPallara} and recalling the inequality $\| f \|_{L^\infty (]0, T[ )} \leq \delta$ we eventually conclude that $\mu \in C_\delta$.  This implies that the set $C_\delta$ is compact. Also, it is obviously convex. Owing to the  
Schauder-Tychonoff Fixed Point Theorem, if the map $T$ defined as in~\eqref{e:N:T} is continuous, then it admits a fixed point, which is by construction a Nash equilibrium in the sense of Definition~\ref{def:solN}. Hence, the proof of Theorem~\ref{th:mainNash} boils down to the proof of the continuity of $T$. \\
{\sc Step 2:} we prove that the map $T$ defined as in~\eqref{e:N:T} is continuous with respect to the weak-$^\ast$ convergence. To prove the continuity of $T$  it suffices to show that the map $\eta \mapsto \mu^{opt}(\eta)$ defined as in the statement of Lemma~\ref{l:exunash} is continuous. Hence, we fix 
\begin{equation}
\label{e:sigmaenne}
  \sigma_n \weaks \sigma. 
\end{equation}
We want to show that 
\begin{equation}
\label{e:continuityT}
      \tau_n : = \mu^{opt}(\sigma_n)
      \weaks \mu^{opt}(\sigma): = \tau \quad \text{as $n \to + \infty$}.
\end{equation}
Owing to the weak-$^\ast$compactness of $C_\delta$ we have that, up to subsequences, 
\begin{equation}
\label{e:continuityTi}
      \tau_n       \weaks \tau_\infty \quad \text{as $n \to + \infty$}
\end{equation}
for some $\tau_\infty \in C_\delta$. Owing to the uniqueness part in 
Lemma~\ref{l:exunash}, to establish~\eqref{e:continuityT} it suffices to show that
\begin{equation}
\label{e:maggioreuguale}
        J_\sigma (\tau_\infty) \ge J_\sigma \left( \tau \right).
\end{equation}
To establish~\eqref{e:maggioreuguale} we argue as follows. First, we term  $\fhi_n$ the weak solution of the initial-boundary value problem~\eqref{e:N:ibvp} in the case when $\eta=\sigma_n$ and $\mu=\tau_n$, namely
\begin{equation}
\label{e:N:ibvpn}
  \begin{cases}
\pt \fhi_n = \pxx \fhi_n -  [\sigma_n + \tau_n] \fhi_n + g(t,x, \fhi_n), & \text{in $]0, T[ \times ]0, R[$} ,  \\
\px \fhi_n(t,0)=\px \vfi_n(t,R) = 0, & t \in ]0, T[, \\
\fhi_n(0,x) = \fhi_0(x), & x \in ]0,R[.
    \end{cases}
\end{equation}
Note that, owing to~\eqref{e:N:gei}, 
\begin{equation}
\label{e:N:geienne}
   J_{\sigma_n} \left( \tau_n \right)=
    \int_0^T \int_0^R 
    \fhi_n (t, x) d \tau_{n,t}(x)  dt - \Psi \left( 
    \int_0^T \int_0^R 
    c (t, x) d \tau_{n,t}(x)  dt \right).
\end{equation}
Also, we term $\tilde \fhi_n$ the solution of the initial-boundary value problem~\eqref{e:N:ibvp} in the case when $\eta= \sigma_n$ and $\mu= \tau$, namely
\begin{equation}
\label{e:N:ibvpnt}
  \begin{cases}
\pt \tilde \fhi_n = \pxx \tilde \fhi_n -  [\sigma_n + \tau] \tilde \fhi_n + g(t,x, \tilde
\fhi_n), & \text{in $]0, T[ \times ]0, R[$} ,  \\
\px \tilde \fhi_n(t,0)=\px \tilde \vfi_n(t,R) = 0, & t \in ]0, T[, \\
\tilde \fhi_n(0,x) = \fhi_0(x), & x \in ]0,R[.
    \end{cases}
\end{equation}
We recall that $\tau_n = \mu^{opt} (\sigma_n)$ and we infer that  
\begin{equation}
\label{e:maggioreuguale2}
J_{\sigma_n} \left( \tau_n \right) \ge 
J_{\sigma_n} \left( \tau \right) =
\int_0^T \int_0^R 
   \tilde  \fhi_n (t, x) d \tau_t (x) dt - \Psi \left( 
    \int_0^T \int_0^R 
    c (t, x) d \tau_t(x) dt \right).
\end{equation}
Next, we recall the estimate~\eqref{e:stability2}, we use the Aubin-Lions Lemma as in the proof of Theorem~\ref{th:main1} and we conclude that 
there are functions $\fhi$ and $\tilde \fhi$ such that
\begin{equation}
\label{e:convergenzafhi}
    \fhi_n \to \fhi, \quad \tilde \fhi_n \to \tilde \fhi
    \quad \text{strongly in $L^2 (]0, T[; C^0 ([0, R]))$ as $n \to + \infty$}.
\end{equation}
Also, by arguing again as in the proof of Theorem~\ref{th:main1} we get that we can pass to the limit in the distributional formulation of the initial-boundary value problem. We conclude that $\fhi$ and $\tilde \fhi$ are solutions of the initial-boundary value problem~\eqref{e:N:ibvp} in the case when $\eta=\sigma$, $\mu= \tau_\infty$ and 
$\eta=\sigma$, $\mu= \tau$, respectively, namely
\begin{equation}
\label{e:N:ibvp2}
  \begin{cases}
\pt  \fhi = \pxx \fhi -  [\sigma + \tau_\infty] \fhi + g(t,x, 
\fhi)\\
\px \fhi(t,0)=\px \vfi(t,R) = 0 \\
\fhi(0,x) = \fhi_0(x) \\
    \end{cases}
    \qquad 
    \begin{cases}
\pt \tilde \fhi = \pxx \tilde \fhi -  [\sigma+ \tau] \fhi + g(t,x, \tilde
\fhi)\\
\px \tilde \fhi(t,0)=\px \tilde \vfi(t,R) = 0 \\
\tilde \fhi(0,x) = \fhi_0(x). \\
    \end{cases}
\end{equation}
Next, we use the convergence $\fhi_n \to \fhi$ and the lower semicontinuity of $c$ to pass to the limit in~\eqref{e:N:geienne}. By using the fact that $\Psi$ is nondecreasing, we get 
\begin{equation}
\label{e:N:limsup}
\begin{split}
   \limsup_{n \to + \infty} J_{\sigma_n} (\tau_n) 
   & \leq  \limsup_{n \to + \infty}
   \int_0^T \int_0^R 
   \fhi_n (t, x) d \tau_{n,t} (x) dt + \limsup_{n \to + \infty} -
   \Psi \left( 
    \int_0^T \int_0^R 
    c (t, x) d \tau_{n,t}(x) dt \right) \\
    & = \lim_{n \to + \infty}
   \int_0^T \int_0^R 
   \fhi_n (t, x) d \tau_{n,t} (x) dt - \liminf_{n \to+ \infty}
    \Psi \left( 
    \int_0^T \int_0^R 
    c (t, x) d \tau_{n,t}(x) dt \right)  \\
    & \stackrel{\eqref{e:continuityT}}{=}
     \int_0^T \int_0^R 
   \fhi (t, x) d \tau_{\infty,t} (x) dt -
    \Psi \left(  \liminf_{n \to + \infty}
    \int_0^T \int_0^R 
    c (t, x) d \tau_{n,t}(x) dt \right)  \\  
  &  \stackrel{\Psi' \ge 0}{\leq}
     \int_0^T \int_0^R 
   \fhi (t, x) d \tau_{\infty,t} (x) dt -
    \Psi \left( 
    \int_0^T \int_0^R 
    c (t, x) d \tau_{\infty,t}(x) dt \right) \\
    & \stackrel{\eqref{e:N:ibvp2}}{=}
    J_\sigma (\tau_\infty). 
    \end{split}
\end{equation}
By using the convergence $\tilde \fhi_n \to \tilde \fhi$ we can then pass to the limit
in the expression at the right hand side of~\eqref{e:maggioreuguale2} 
and conclude that 
\begin{equation*}
\begin{split}
  J_\sigma (\tau_\infty) 
  \ge  \limsup_{n \to + \infty} J_{\sigma_n} (\tau_n)
   \ge \lim_{n \to + \infty} J_{\sigma_n} (\tau) & =
   \int_0^T \! \! \! \int_0^R 
   \tilde  \fhi (t, x) d \tau_t (x) dt - \Psi \left( 
    \int_0^T \! \! \! \int_0^R 
    c (t, x) d \tau_t(x) dt \right) \\
    & \stackrel{\eqref{e:N:ibvp2}}{=}  J_\sigma (\tau).
    \end{split}
\end{equation*}
The above chain of inequalities implies~\eqref{e:maggioreuguale} and hence concludes the proof
of Theorem~\ref{th:mainNash}. 
\end{proof}
\appendix
\section{Fundamental solutions of the heat equation}
For the readers' convenience, we collect in this section some basic facts about the fundamental solutions of the heat equation in one-dimensional, bounded domains. We refer to~\cite{Cannon} for an extended discussion.   

First, we fix an interval $]0, R[$ and we define the function $D$ by setting 
\begin{equation}
 \label{e:Di}
     D: ]0, + \infty[ \times ]0, R[ \times ]0, R[ \to \R \qquad 
     D(t, x, y) : = \sum_{m= - \infty}^{m = + \infty} G(t, x + 2m R - y)+
     G(t, x + 2m R +y), 
 \end{equation}
 where $G$ is the standard Green kernel 
 $$
     G(t, x) : = \frac{1}{2 \sqrt{\pi t}} \exp \left( \frac{-x^2}{4t}\right). 
 $$
Note that, for every $u_0 \in L^2 (]0, R[)$, the function
$$
     u(t, x) : = \int_0^R D(t, x, y) u_0(y) dy 
$$
is a weak solution of the initial-boundary value problem
\begin{equation*}
\begin{cases}
\pt u = \pxx u & \text{in $]0, T[ \times ]0, R[$} ,  \\
\px u(t,0)=\px u(t,R) = 0, & t \in ]0, T[, \\
u(0,x) = u_0(x), & x \in {]0,R[}.
    \end{cases}
\end{equation*}
Note furthermore that, owing to Duhamel's principle, for every measurable, bounded 
function \\${\ell: ]0, + \infty[ \times ]0, R[ \to \R}$
the function
$$
    u(t, x) : = \int_0^R D(t, x, y) u_0(y) dy +
    \int_0^t \! \! \int_0^R D(t -s , x, y) \ell (s, y) dy ds 
$$ 
is a weak solution of the initial-boundary value problem
\begin{equation*}
\begin{cases}
\pt u = \pxx u + \ell(t, x) & \text{in $]0, T[ \times ]0, R[$} ,  \\
\px u(t,0)=\px u(t,R) = 0, & t \in ]0, T[, \\
u(0,x) = u_0(x), & x \in {]0,R[}.
    \end{cases}
\end{equation*}
By direct computations, one can show that the kernel $D$ satisfies the following estimates:
\begin{eqnarray}
    \| D(t, x, \cdot ) \|_{L^\infty (]0, R[)}
    \leq \frac{C(R)}{\sqrt{t}}
    \qquad \text{for every $t>0$, $x \in ]0, R[$} 
    \label{e:Dlinfty} \\
    \| D(t, x, \cdot ) \|_{L^1(]0, R[)}
    \leq K
    \qquad \text{for every $t>0$, $x \in ]0, R[$} 
    \label{e:Duno} \\
     \| \partial_y D(t, x, \cdot ) \|_{L^2 (]0, R[)}, 
     \; \| \partial_x D(t, x, \cdot ) \|_{L^2 (]0, R[)}
    \leq \frac{K}{t^{3/4}}
    \qquad \text{for every $t>0$, $x \in ]0, R[$} 
    \label{e:Dxdue} \\
     \| D(t, \cdot, y ) \|_{L^1(]0, R[)}
    \leq K
    \qquad \text{for every $t>0$, $y \in ]0, R[$} 
    \label{e:Dunox} 
    \\
     \| \partial_x D(t, \cdot, y ) \|_{L^2 (]0, R[)}
    \leq \frac{K}{t^{3/4}}
    \qquad \text{for every $t>0$, $y \in ]0, R[$}.
    \label{e:Dxduex} 
\end{eqnarray}
Finally, we define the kernel $\tilde D$ associated with the Dirichlet boundary conditions by setting 
\begin{equation}
 \label{e:Ditilde}
     \tilde D: ]0, + \infty[ \times ]0, R[ \times ]0, R[ \to \R \qquad 
     \tilde D(t, x, y) : = \sum_{m= - \infty}^{m = + \infty} G(t, x + 2m R - y)-
     G(t, x + 2m R +y), 
 \end{equation}
and we point out that 
\begin{equation}
\label{e:dualita}
    \int_0^R \partial_x D(t, x, y) u_0(y) dy =
   \int_0^R \tilde D(t, x, y) u'_0(y) dy, 
   \quad 
   \int_0^R \partial_{xx} D(t, x, y) u_0(y) dy =
   \int_0^R \px  \tilde D(t, x, y) u'_0(y) dy
\end{equation}
for every continuously differential function $u_0$. By direct computations, we get the estimates 
\begin{eqnarray}
    \| \tilde D(t, x, \cdot ) \|_{L^\infty (]0, R[)}
    \leq \frac{C(R)}{\sqrt{t}}
    \qquad \text{for every $t>0$, $x \in ]0, R[$} 
    \label{e:tDlinfty} \\
    \| \tilde D(t, x, \cdot ) \|_{L^1(]0, R[)}
    \leq K
    \qquad \text{for every $t>0$, $x \in ]0, R[$} 
    \label{e:tDuno} \\
     \| \partial_x \tilde D(t, x, \cdot ) \|_{L^2 (]0, R[)}
    \leq \frac{K}{t^{3/4}}
    \qquad \text{for every $t>0$, $x \in ]0, R[$} 
    \label{e:tDxdue} \\
    \| \tilde D(t, \cdot, y ) \|_{L^1(]0, R[)}
    \leq K
    \qquad \text{for every $t>0$, $y \in ]0, R[$.} 
    \label{e:tDunox} 
\end{eqnarray}
\section*{Acknowledgments}
All the authors are members of the Gruppo Nazionale per l'Analisi Matematica, la Probabilit\`a e le loro Applicazioni (GNAMPA) of the Istituto Nazionale di Alta Matematica (INdAM) and are supported by the MIUR-PRIN Grant ``Nonlinear Hyperbolic Partial Differential Equations, Dispersive and Transport Equations: theoretical and applicative aspects''.

\end{document}